\documentclass[11pt]{amsart}

\usepackage{}
\usepackage{bbm}
\usepackage{amsfonts}
\usepackage{amsthm}
\usepackage{amsmath}
\usepackage{amssymb}
\usepackage{mathrsfs}
\usepackage{amscd}
\usepackage{mathdots}
\usepackage[all]{xy}
\usepackage{amssymb,latexsym}
\usepackage{graphicx}
\usepackage{chngcntr}
\usepackage{fancyhdr}
\usepackage[top=1in,bottom=1in,left=1.25in,right=1.25in]{geometry}

\newcommand{\DF}[2]{{\displaystyle\frac{#1}{#2}}}

\newcommand{\INT}{\mathbb{Z}}
\newcommand{\INTP}{\mathbb{Z}_{(p)}}
\newcommand{\COMP}{\mathbb{Z}_p}
\newcommand{\RAT}{\mathbb{Q}}

\newcommand{\CPLX}{\mathbb{C}}

\newcommand{\A}{\mathcal {A}}
\newcommand{\D}{\mathcal {D}}
\newcommand{\fP}{\mathcal {P}}
\newcommand{\V}{\mathcal {V}}
\newcommand{\ES}{\mathscr{S}}

\newcommand{\dr}{\mathrm{dR}}
\newcommand{\cris}{\mathrm{cris}}
\newcommand{\sdr}{s_{\dr}}

\newcommand{\intG}{G_{\mathbb{Z}_p}}

\newcommand{\inttG}{G_{\mathbb{Z}_{(p)}}}

\newcommand{\IIsom}{\mathbf{Isom}}
\newcommand{\Isom}{\mathrm{Isom}}
\newcommand{\AAut}{\mathbf{Aut}}

\newcommand{\uAut}{\underline{\mathrm{Aut}}}

\newcommand{\Sh}{\mathrm{Sh}}

\newcommand{\Hdr}{\mathrm{H}^1_{\mathrm{dR}}}
\newcommand{\lin}{\mathrm{lin}}
\newcommand{\GL}{\mathrm{GL}}

\newcommand{\adj}{\mathrm{ad}}
\DeclareMathOperator{\Spec}{\mathrm{Spec}}

\DeclareMathOperator{\gr}{\mathrm{gr}}

\def\st{\stackrel}
\def\Q{\mathbb{Q}}
\def\ra{\rightarrow}

\usepackage[pdftex]{hyperref}

\begin{document}
\newtheorem{theorem}[subsubsection]{Theorem}
\newtheorem{lemma}[subsubsection]{Lemma}
\newtheorem{proposition}[subsubsection]{Proposition}
\newtheorem{corollary}[subsubsection]{Corollary}
\theoremstyle{definition}
\newtheorem{definition}[subsubsection]{Definition}
\theoremstyle{definition}
\newtheorem{construction}[subsubsection]{Construction}
\theoremstyle{definition}
\newtheorem{notations}[subsubsection]{Notations}
\theoremstyle{definition}
\newtheorem{asp}[subsubsection]{Assumption}
\theoremstyle{definition}
\newtheorem{set}[subsubsection]{Setting}
\theoremstyle{remark}
\newtheorem{remark}[subsubsection]{Remark}
\theoremstyle{remark}
\newtheorem{example}[subsubsection]{Example}
\newtheorem{examples}[subsubsection]{Examples}
\newtheorem{conjecture}[subsubsection]{Conjecture}
\theoremstyle{plain}
\newtheorem{introth}{Theorem}
\renewcommand{\theintroth}{\Alph{introth}}

\makeatletter
\newenvironment{subeqn}{\refstepcounter{subsubsection}
$$}{\leqno{\rm(\thesubsubsection)}$$\global\@ignoretrue}
\makeatother
\author{Xu Shen and Chao Zhang}

\address{Morningside Center of Mathematics, Academy of Mathematics and Systems Science, Chinese Academy of Sciences, No. 55, Zhongguancun East Road, Beijing 100190, China}
\address{University of Chinese Academy of Sciences, Beijing 100049, China}

\email{shen@math.ac.cn}

\address{Shing-Tung Yau Center of Southeast University, Yifu Architecture Building, No. 2,
Sipailou, Nanjing 210096, China}

\email{zhangchao1217@gmail.com}

\title[Stratifications in Shimura varieties of abelian type]{Stratifications in good reductions of Shimura varieties of abelian type}

\maketitle 
\setcounter{tocdepth}{1}
\begin{abstract}
In this paper we study the geometry of good reductions of Shimura varieties of abelian type. More precisely, we construct the Newton stratification, Ekedahl-Oort stratification, and central leaves on the special fiber of a Shimura variety of abelian type at a good prime. We establish several basic properties of these stratifications, including the non-emptiness, closure relation and dimension formula, generalizing those previously known in the PEL and Hodge type cases. We also study the relations between these stratifications, both in general and in some special cases, such as those of fully Hodge-Newton decomposable type. We investigate the examples of quaternionic and orthogonal Shimura varieties in details.
\end{abstract}
\tableofcontents

\newpage

\section*[Introduction]{Introduction}
Understanding the geometric properties of Shimura varieties in mixed characteristic has been a central problem in arithmetic algebraic geometry and Langlands program. In this paper we study the geometry of good reductions of Shimura varieties of abelian type, based on the works of Kisin \cite{CIMK}, Kim-Madapusi Pera \cite{2-adic int} and Vasiu \cite{INTV} where smooth integral canonical models for these Shimura varieties were already available, and following the genereal guideline proposed by He-Rapoport in \cite{He-Rap} (see also \cite{guide to red mod p}) where basic axioms were postulated to study various stratifications on the special fibers of certain integral models of Shimura varieties.\\
 
 Let $(G,X)$ be a Shimura datum with reflex field $E$.  For any open compact (neat) group $K\subseteq G(\mathbb{A}_f)$, by the works of Shimura, Deligne, Milne and Borovoi, we have the attached Shimura variety $\Sh_K(G,X)$ over $E$.
The datum $(G,X)$ is said to have good reduction at a prime $p$, if $G_{\RAT_p}$ extends to a reductive group $\intG$ over $\COMP$. We will fix a place $v$ of $E$ over $p$, and write $O_{E,(v)}$ for the ring of $v$-integers. For $K_p=G_{\COMP}(\COMP)$, Langlands and Milne conjectured (cf. \cite{milne} section 2) that the pro-variety $$\Sh_{K_p}(G,X):=\varprojlim_{K^p}\Sh_{K_pK^p}(G,X),$$ where $K^p$ runs through compact open subgroups of $G(\mathbb{A}_f^p)$, has an integral canonical model \[\ES_{K_p}(G,X)\] over $O_{E,(v)}$. The prime to $p$ Hecke action of $G(\mathbb{A}_f^p)$ on $\Sh_{K_p}(G,X)$ should extend to $\ES_{K_p}(G,X)$, and when $K^p$ varies the inverse system of \[\ES_{K_pK^p}(G,X):=\ES_{K_p}(G,X)/K^p\] should be a system of smooth models of $\Sh_{K_pK^p}(G,X)$ with \'{e}tale transition morphisms. Thanks to the works of Kisin \cite{CIMK}, Vasiu \cite{INTV} and Kim-Madapusi Pera \cite{2-adic int} (for $p=2$), smooth integral canonical models are known to exist if the Shimura datum $(G, X)$ is of abelian type. Thus it is natural to investigate geometry of the (geometric) special fibers \[\ES_{K_pK^p,0}(G,X)\] over\footnote{Here in the introduction we work uniformly over $\overline{\kappa}$ for simplicity. We remind the reader that in the body part of this paper, we denote by $\ES_{K_pK^p,0}(G,X)$ the special fiber over $\kappa$ and by $\ES_{K_pK^p,\overline{\kappa}}(G,X)$ the geometric special fiber over $\overline{\kappa}$.} $\overline{\kappa}$ of these models, where $\kappa$ is the residue field of $O_{E,(v)}$. In the following, $(G,X)$ will always be a Shimura datum of abelian type with good reduction at $p$.

It turns out that the geometry of Shimura varieties in characteristic $p$  is much finer than that in characteristic 0, in the sense that there are several invariants in characteristic $p$, which are stable under the prime to $p$ Hecke action, leading to various natural stratifications of the special fiber $\ES_{K_pK^p,0}(G,X)$. Following Oort (in the Siegel case, see \cite{foliation-Oort} for example), Viehmann-Wedhorn (in the PEL type case, cf. \cite{VW}) and many others (see the references of \cite{VW, Vieh} for example), we mainly concentrate on the \emph{Newton stratification},  the \emph{Ekedahl-Oort stratification}, and the \emph{central leaves} in this paper. In fact in this paper we will only be concerned with some basic properties of these stratifications, and the relations between these strata. Our study here can be put\footnote{In fact the main part of \cite{He-Rap} is to work with all parahoric levels at $p$. Here we restrict to the hyperspecial levels, as a first step toward the verification of the axioms in \cite{He-Rap} in the abelian type case.} in the general framework proposed by He-Rapoport in \cite{He-Rap}, where more group theoretic aspects are emphasized (compare also \cite{guide to red mod p, coxeter type, fully H-N decomp}).\\

 If $(G,X)$ is of PEL type, then we can use the explicit moduli interpretation to treat the geometry of the special fibers. In the more general Hodge type case, at the current stage we do not know whether there exists moduli interpretation in mixed characteristic. However, there still exists an abelian scheme together with certain tensors over the special fiber of a Hodge type Shimura variety, and we can make use of it to study the geometry modulo $p$, cf. \cite{Foliation-Hamacher,muordinary,EOZ,level m} for example. If now $(G,X)$ is a general abelian type Shimura datum, which is the case we want to treat in this paper, then there is no abelian schemes nor $p$-divisible groups over the associated Shimura varieties at all. Nevertheless, we can study them by choosing some related Hodge type Shimura varieties. This usually requires the study of some finer geometric structures on these Hodge type Shimura varieties. Along the way, we will also see some close relations between the strata of different Shimura varieties. 
 
 To a certain extent, many of our following main results were previously known in the PEL type and Hodge type cases. Our modest goal here is to extend them to the abelian type case and hence in the full generality when integral canonical models exist, and to provide a useful documentary literature with a point of view toward possible applications to Langlands program. On the other hand, there are many natural examples of Shimura varieties of abelian type but not of Hodge type: for example, the Shimura varieties associated to a general (not totally indefinite) quaternion algebra or the special orthogonal group SO. We discuss our constructions for these  Shimura varieties in details, which we hope to find interesting applications. For example, the orthogonal Shimura varieties play very important roles in Kudla's program (\cite{Kud}) and the arithmetic Gan-Gross-Prasad conjecture (\cite{GGP}). We expect that our results will be found useful to these fields.\\
 
Now we state our main results.
Let $\{\mu\}$ be the Hodge cocharacter attached to the Shimura datum $(G, X)$. The parametrizing set of the Newton stratification is the finite Kottwitz set $B(G,\mu)$ (cf. \cite{isocys with addi 2} section 6), which may be viewed as the set of isomorphism classes of $F$-isocrystals with $G$-structure associated to (geometric) points in $\ES_0:=\ES_{K_pK^p,0}(G,X)$. Recall that there is a partial order $\leq$ on $B(G,\mu)$, cf. \ref{group theo settings for NP}. In the classical Siegel case, one can realize $B(G,\mu)$ as the set of Newton polygons of the polarized $p$-divisible groups attached to points on the special fiber. The basic properties of the Newton stratification are as follows\footnote{In fact the Newton stratification is defined over $\kappa$, and these properties are also true over $\kappa$, see subsections 2.2 and 2.3.} (cf. Theorem \ref{Newton for abelian type}).
\begin{introth}
Each Newton stratum $\ES_0^b$ is non-empty, and it is an equi-dimensional locally closed subscheme of $\ES_0$ of dimension \[\langle\rho,\mu+\nu_G(b)\rangle-\frac{1}{2}\mathrm{def}_G(b).\] Here $\rho$ is the half-sum of positive roots of $G$, $\nu_G(b)$ is the Newton point associated to $[b]\in B(G,\mu)$, and $\mathrm{def}_G(b)$ is the number defined in Definition \ref{D:def}. Moreover, $\overline{\ES_0^b}$, the closure of $\ES_0^b$, is the union of strata $\ES_0^{b'}$ with $[b']\leq [b]$, and $\overline{\ES_0^b}-\ES_0^b$ is either empty or pure of codimension 1 in $\overline{\ES_0^b}$.
\end{introth}
We remark that the non-emptiness was conjectured by Rapoport (cf. \cite{guide to red mod p} Conjecture 7.1) and by Fargues (cf. \cite{corr local L and coho R-Z} page 55), and it has been proved by Viehmann-Wedhorn in the PEL type case (\cite{VW}), and Dong-Uk Lee, Kisin-Madapusi Pera and Chia-Fu Yu respectively in the Hodge type case, see \cite{Newton non-epty, Yu} for example. The other statements are due to Hamacher in the PEL type and Hodge type cases, cf. \cite{Foliation-Hamacher,geo of newt PEL}. The dimension formula in the Hodge type case was proved independently by the second author in \cite{level m}.\\
 
Let $W=W_G$ be the (absolute) Weyl group of $G$, and we have a certain subset ${}^JW\subset W$ defined by  $\{\mu\}$ equipped with a partial order $\preceq$, cf. \ref{subsection group EO}. The parametrizing set of the Ekedahl-Oort stratification is the set ${}^JW$, which classifies isomorphism classes of $G$-zips (or ``$F$-zips with $G$-structure'') associated to (geometric) points in $\ES_0=\ES_{K_pK^p,0}(G,X)$. In the classical Siegel case, ${}^JW$ classifies the $p$-torsions of the polarized abelian varieties attached to points on the special fiber. The basic properties of the Ekedahl-Oort stratification are as follows (cf. Theorem \ref{Th EO abelian type}).
\begin{introth}
\begin{enumerate}
\item Each Ekedahl-Oort stratum $\ES_0^w$ is an equi-dimensional locally closed subscheme of $\ES_{0}$. Moreover, $\overline{\ES_0^w}$, the closure of $\ES_0^w$,  is the union of strata $\ES_0^{w'}$ with $w'\preceq w$.
\item  For $w\in {}^JW$,  $\ES_0^w$ is of dimension of $l(w)$, the length of $w$, if non-empty. Moreover, each $\ES_0^w$ is non-empty if $p>2$.
\item Each stratum $\ES_0^w$ is smooth and quasi-affine.
\end{enumerate}
\end{introth}
We remark that the non-emptiness is due to Viehmann-Wedhorn in the PEL type case (\cite{VW}), and Chia-Fu Yu in the Hodge type case (\cite{Yu}). In the projective Hodge type case, Koskivirta proved the non-emptiness independently, cf. \cite{Kosk}. We also reamrk that the non-emptiness here (as well as in Theorem C) relies on \cite{LRKisin} Proposition 1.4.4, where $p>2$ has to be assumed. The other statements in the PEL type case are due to Viehmann-Wedhorn (\cite{VW}). In the Hodge type case, the quasi-affiness is due to Goldring-Koskivirta (\cite{Hasse invariants}), and the closure relation and dimension formula are due to the second author (\cite{EOZ}).\\
 
Attached to the Shimura datum $(G, X)$ we have a set\footnote{Here the more natural set should be $C(G,\mu)$; however, there will be no difference if the center $Z_G$ of $G$ is connected, cf. Lemma \ref{iden C(G,mu)}.} $C(G^\adj,\mu)$, which may be viewed as the set (often infinite) of isomorphism classes of $F$-crystals with $G^\adj$-structure associated to points in $\ES_{K_pK^p,0}(G,X)$. Here $G^\adj$ is the adjoint group associated to $G$. We have surjections $C(G^\adj,\mu)\twoheadrightarrow B(G^\adj,\mu)\simeq B(G,\mu)$ and $C(G^\adj,\mu)\twoheadrightarrow {}^JW_{G^\adj}\simeq {}^JW_{G}$ which, roughly speaking, send $F$-crystals with $G^\adj$-structure to the associated $F$-isocrystals with $G^\adj$-structure and $G^\adj$-zips respectively. Associated to an element $[c]\in C(G^\adj,\mu)$, we can define a central leaf, which is a finer structure than the above Newton and Ekedahl-Oort strata. In the Siegel case, a central leaf is the locus where one fixes an isomorphism class of the polarized $p$-divisible groups. The basic properties of central leaves are as follows (cf. Theorem \ref{cent abelian ty-ad}).
\begin{introth}
Each central leaf is a smooth, equi-dimensional locally closed subscheme of $\ES_{0}$. It is closed in the Newton stratum containing it. Any central leaf in a Newton stratum $\ES_0^b$ is of dimension $\langle2\rho,\nu_G(b)\rangle$ if non-empty. Here as above $\rho$ is the half sum of positive roots of $G$. Moreover, central leaves are non-empty if $p>2$.
\end{introth}
The non-emptiness in the abelian type case follows from that in the Hodge type case, which is in turn a consequence of the non-emptiness of the Newton strata. In the PEL type case, see \cite{VW} Theorem 10.2.
The other statements in the Hodge type case are due to Hamacher (cf. \cite{Foliation-Hamacher}; see also \cite{geo of newt PEL} in the PEL type case) and the second author (\cite{level m}) respectively.\\

The ideas to prove the above theorems are as follows. We consider first the Hodge type case, where most of the above are known, see the above remarks after each theorems.
To extend to the abelian type case, we first work with a Shimura datum of abelian type such that the group $G$ is \emph{adjoint}. By using a lemma of Kisin (cf. Lemma \ref{Kisin's lemma}), we can find a Hodge type Shimura datum $(G_1, X_1)$ such that
\begin{enumerate}
	\item $(G_1^\adj,X_1^\adj)\stackrel{\sim}{\longrightarrow} (G,X)$ and $Z_{G_1}$ is a torus;
	
	\item if $(G,X)$ has good reduction at $p$, then $(G_1,X_1)$ in (1) can be chosen to have good reduction at $p$, and such that $E(G,X)_p=E(G_1,X_1)_p$.
\end{enumerate}
Then the integral canonical model for $(G,X)$ is given by \[\begin{split}\ES_{K_p}(G,X)&=[\mathscr{A}(G_{\INTP})\times \ES_{K_{1,p}}(G_1,X_1)^+]/\mathscr{A}(G_{1,\INTP})^\circ\\
&=[\mathscr{A}(G_{\INTP})\times \ES_{K_{p}}(G,X)^+]/\mathscr{A}(G_{\INTP})^\circ,\end{split}\] 
where $\mathscr{A}(G_{\INTP}), \mathscr{A}(G_{1,\INTP})^\circ$ and $\mathscr{A}(G_{\INTP})^\circ$ are the groups defined in \cite{CIMK} 3.3.2 (see also \ref{group-compn to whole}). On geometrically connected components we have
\[\ES_{K_{p}}(G,X)^+=\ES_{K_{1,p}}(G_1,X_1)^+/\Delta\]  with
\[\Delta=\mathrm{Ker}(\mathscr{A}(G_{1,\INTP})^\circ\ra  \mathscr{A}(G_{\INTP})^\circ).\]
To show that the induced Newton stratification, Ekedahl-Oort stratification, central leaves on $\ES_{K_{1,p},0}(G_1,X_1)^+$ desecend to $\ES_{K_p,0}(G,X)^+$, we need to show that the Newton strata, Ekedahl-Oort strata, and central leaves of $\ES_{K_{1,p},0}(G_1,X_1)^+$ are stable under the action of $\Delta$, and their quotients by $\Delta$ are well defined. By \cite{Kisin-Pappas} 4.4 the action of $\Delta$ can be described by certain construction of twisting of abelian varieties. This leads us to study the effect to $p$-divisible groups with additional structures under the construction of twisting abelian varieties in \cite{Kisin-Pappas}. Using the fact that $Z_{G_1}$ is a torus, we can show that this twisting does not change the associated $p$-divisible groups with additional structures, and thus  the Newton strata, Ekedahl-Oort strata, and central leaves of $\ES_{K_{1,p},0}(G_1,X_1)^+$ are stable under the action of $\Delta$, and their quotients by $\Delta$ are indeed well defined. For a general Shimura datum of abelian type $(G, X)$, we first pass to the associated adjoint Shimura datum $(G^\adj, X^\adj)$ and apply the above construction to $(G^\adj, X^\adj)$. Then we define the Newton stratification, Ekedahl-Oort stratification, and central leaves on $\ES_{K_p,0}(G,X)$ by pullling back those on $\ES_{K_p^\adj,0}(G^\adj,X^\adj)$ under the natural morphism \[\ES_{K_p,0}(G,X)\ra \ES_{K_p^\adj,0}(G^\adj,X^\adj).\]

In fact, there is an alternative way (however we need to assume $p>2$ here) to define the Newton stratification, Ekedahl-Oort stratification, and central leaves on $\ES_{K_p,0}(G,X)$, by using the filtered $F$-crystal with $G^c$-structure \[\omega_{\text{cris}}:\mathrm{Rep}_{\mathbb{Z}_p}(G^c)\rightarrow  \mathrm{FFCrys}_{\widehat{\ES_{K_p}}(G,X)}\] on $\ES_{K_p}(G,X)$ constructed by Lovering in \cite{Loverig 2}, which may be viewed as a crystalline model of the universal de Rham bundle $\omega_{\mathrm{dR}}: \mathrm{Rep}_{\mathbb{Q}_p}(G^c)\rightarrow  \mathrm{Fil}^\nabla_{\widehat{\ES_{K_p}}(G,X)^{rig}}$, see \cite{Liu-Zhu}.  Here $G^c=G/Z_G^{nc}$ and $Z_G^{nc}\subset Z_G$ is the largest subtorus of $Z_G$ that is split over $\mathbb{R}$ but anisotropic over $\mathbb{Q}$, $\widehat{\ES_{K_p}}(G,X)$ is the $p$-adic completion of $\ES_{K_p}(G,X)$ along its special fiber, $\widehat{\ES_{K_p}}(G,X)^{rig}$ is the associated adic space, and $\mathrm{FFCrys}_{\widehat{\ES_{K_p}}(G,X)}$ (resp. $\mathrm{Fil}^\nabla_{\widehat{\ES_{K_p}}(G,X)^{rig}}$) is the category of filtered $F$-crystals (resp. filtered isocrystals) on $\widehat{\ES_{K_p}}(G,X)$ (resp. $\widehat{\ES_{K_p}}(G,X)^{rig}$), cf. \ref{subsection filtered}. This construction in turn uses ideas from \cite{Loverig 1} where one constructs an auxiliary Shimura datum of abelian type $(\mathcal{B},X')$, such that there is a commutative diagram of Shimura data $$\xymatrix{
	(\mathcal{B},X')\ar[d]\ar[r]& (G_1,X_1)\ar[d]\\
	(G,X)\ar[r]& (G^\mathrm{ad},X^\mathrm{ad})}$$ inducing a commutative diagram of (integral models of) Shimura varieties
$$\xymatrix{\ES_{K_{\mathcal{B},p}}(\mathcal{B},X')\ar[d]\ar[r]& \ES_{K_{1,p}}(G_1,X_1)\ar[d]\\
	\ES_{K_p}(G,X)\ar[r]& \ES_{K_p^\adj}(G^\mathrm{ad},X^\adj).}$$
Using the auxiliary Shimura datum of abelian type $(\mathcal{B},X')$, one can then construct the universal filtered $F$-crystal with $G^c$-structure
on $\ES_{K_p}(G,X)$  from that on $\ES_{K_{1,p}}(G_1,X_1)$.
If $(G,X)$ is of Hodge type, it is easy to see the construction of the Newton stratification, Ekedahl-Oort stratification, and central leaves using  the filtered $F$-crystal with $G^c$-structure coincides with the construction above. From this we can deduce that the two constructions of the Newton and Ekedahl-Oort stratifications via passing to adjoint and via using filtered $F$-crystal with $G^c$-structure respectively coincide for a general abelian type Shimura datum $(G,X)$, cf. \ref{Newton using crys cano}, \ref{E-O using crys cano}. If the center $Z_G$ is connected, we can show the two constructions of central leaves also coincide. In the general case, except the non-emptiness, all the other statements in the above Theorem C also hold for the \emph{canonical} central leaves defined via the $F$-crystal with $G^c$-structure
on $\ES_{K_p,0}(G,X)$ in Definition \ref{def canonical central leaves}. For more details, see \ref{central leaves using crys cano}.\\

We also study the relations between the Newton stratification, Ekedahl-Oort stratification, and central leaves using the group theoretic methods in \cite{Nie Sian, coxeter type, fully H-N decomp}. The main results are summarized as follows, cf. Proposition \ref{minimal E-O for ab type}, Corollary \ref{C:ordinary}, Examples \ref{E:ordinary}, Propositions \ref{P: EO and Newton} and \ref{E-O NP for ab cox}. As above, after fixing a prime to $p$ level $K^p\subset G(\mathbb{A}_f^p)$, we simply write $\ES_0=\ES_{K_pK^p,0}(G,X)$. Note that there is no confusion for the notion of central leaves in the following theorem.
\begin{introth}
\begin{enumerate}
	\item Assume that each central leaf is non-empty (which holds when  $p>2$). Each Newton stratum contains a minimal Ekedahl-Oort stratum (i.e. an Ekedahl-Oort stratum which is a central leaf). Moreover, if $G$ splits, then each Newton stratum contains a unique minimal Ekedahl-Oort stratum.
	\item The ordinary Ekedahl-Oort stratum (i.e. the open Ekedahl-Oort stratum) coincides with the $\mu$-ordinary locus (i.e. the open Newton stratum), which is a central leaf. In particular the $\mu$-ordinary locus is open dense in $\ES_0$.
	\item Assume that each central leaf is non-empty (which holds when  $p>2$).	For any $[b]\in B(G,\mu)$ and $w\in {}^JW$ (which we view as an element of the $\mu$-admissible subset of the extended affine Weyl group, cf. \ref{admi and EO elements} and \ref{EO and G-zips}), we have
	\[\ES_0^b\cap \ES_0^w\neq\emptyset \Longleftrightarrow X_w(b)\neq \emptyset,\]where $X_{w}(b):=\{gK\mid g^{-1}b\sigma(g)\in K\cdot_\sigma \mathcal{I}w\mathcal{I}\}\subseteq G(L)/K$, with $L=W(\overline{\kappa})_\Q, W=W(\overline{\kappa}), K=G(W)$, $\sigma$ is the Frobenius on $L$ and $W$, $\mathcal{I}\subset G(L)$ is the Iwahori subgroup, and $K\cdot_\sigma \mathcal{I}w\mathcal{I}$ is as in \ref{the set Z}.
	\item Let $(G,X)$ be a Shimura datum of abelian type with good reduction at $p$ whose attached pair $(G_{\mathbb{Q}_p},\mu)$ is fully Hodge-Newton decomposable (cf. Definition \ref{D:fully HN} and \cite{fully H-N decomp}), then
	\begin{enumerate}
		\item each Newton stratum of $\ES_0$ is a union of Ekedahl-Oort strata;
		
		\item each Ekedahl-Oort stratum in a non-basic Newton stratum is a central leaf, and it is open and closed in the Newton stratum, in particular, non-basic Newton strata are smooth;
		
		\item if $(G_{\mathbb{Q}_p},\mu)$ is of Coxeter type (cf. \ref{subsubsection coxeter} and \cite{coxeter type}), then for two Ekedahl-Oort strata $\ES_0^{1}$ and $\ES_0^{2}$, $\ES_0^{1}$ is in the closure of $\ES_0^{2}$ if and only if $\mathrm{dim}(\ES_0^{2})>\mathrm{dim}(\ES_0^{1})$.
	\end{enumerate}
\end{enumerate}
\end{introth}
Here the first statement was proved in the PEL case by Viehmann-Wedhorn (\cite{VW}) under certain condition and by Nie in \cite{Nie Sian}. The proof of \cite{Nie Sian} is based on some group theoretic results, thus it also applies to our situation. In the Hodge type case the statement (2) is due to Wortmann, see \cite{muordinary}. In the statement (3), the set $X_w(b)$ can be viewed as an Ekedahl-Oort stratum of the affine Deligne-Lusztig variety $X(\mu,b)$, see \ref{the set Z}. The statements in (4) are first due to G\"ortz-He-Nie (\cite{fully H-N decomp}, see also \cite{coxeter type, coxeter erratum}) under the assumption that the axioms of \cite{He-Rap} are verified. Here we do not use any unproved hypothesis or axioms.\\

We now  briefly describe the structure of this article. In the first section, we  first review the construction of integral canonical models for Shimura varieties of abelian type following \cite{CIMK}. We then study twisting of $p$-divisible groups in a general setting which will be used later. In sections 2-4, we construct and study the Newton stratification, Ekedahl-Oort stratification, and central leaves respectively by using the approach of passing to adjoint. We discuss the example of quaternionic Shimura varieties in each section. In section 5, we revisit our constructions of stratifications using the filtered $F$-crystal with $G^c$-structure of \cite{Loverig 2}. In section 6, we study the relations between the Newton stratification, the Ekedahl-Oort stratification, and the central leaves both in the general and special setting. Finally, in section 7 we discuss our results in the setting of GSpin and SO Shimura varieties in details.\\
 \\
\textbf{Acknowledgments.}  Part of this work was done while both the authors were visiting the Academia Sinica in Taipei. We thank Chia-Fu Yu for the invitation and for helpful conversitions. We also thank the Academia Sinica for its hospitality. We thank the referee sincerely for careful readings and useful suggestions on some improvements of the main results.
The first author was partially supported by the Chinese Academy of Sciences grants 50Y64198900, 29Y64200900, the Recruitment Program of Global Experts of China, and the NSFC grants No. 11631009 and No. 11688101, and the National Key R\&D Program of China 2020YFA0712600.

\section[Good reductions of abelian type]{Good reductions of Shimura varieties of abelian type}
In this section, we recall the construction of integral canonical models for Shimura varieties of abelian type in \cite{CIMK} and \cite{2-adic int}. We will start with the construction for those of Hodge type, and then pass to abelian type as in \cite{CIMK}. 

\subsection[Integral models for Shimura varieties of Hodge type]{Integral models for Shimura varieties of Hodge type}\label{model Hodge type}

Let $(G,X)$ be a Shimura datum of Hodge type with good reduction at $p$.
We recall the construction and basic results about the integral canonical models for the associated Shimura varieties.

For a symplectic
embedding $i:(G,X)\hookrightarrow (\mathrm{GSp}(V,\psi),X')$, by \cite{CIMK} Lemma 2.3.1, there exists a
$\mathbb{Z}_p$-lattice $V_{\mathbb{Z}_p}\subseteq
V_{\mathbb{Q}_p}$, such that
$i_{\mathbb{Q}_p}:G_{\mathbb{Q}_p}\subseteq\mathrm{GL}(V_{\mathbb{Q}_p})$
extends uniquely to a closed embedding $\intG\hookrightarrow
\mathrm{GL}(V_{\mathbb{Z}_p}).$ So there is a $\mathbb{Z}$-lattice
$V_{\INT}\subseteq V$ such that $\inttG$, the Zariski closure of $G$ in
$\mathrm{GL}(V_{\mathbb{Z}_{(p)}})$, is reductive, as the base
change to $\mathbb{Z}_p$ of $\inttG$ is $\intG$. Moreover, by Zarhin's trick, we can assume that $\psi$ is perfect on
 $V_{\mathbb{Z}}$. Let $K_p=G_{\mathbb{Z}_p}(\mathbb{Z}_p)$ and $K=K_pK^p$ for a sufficientally small open compact subgroup $K^p\subset G(\mathbb{A}_f^p)$. The
integral canonical model $\ES_K(G,X)$ of $\mathrm{Sh}_K(G,X)$ is
constructed as follows. Let $K'_p=\mathrm{GSp}(V_{\mathbb{Z}_{(p)}},\psi)(\COMP)$, we can choose
$K'=K'_pK'^p\subseteq\mathrm{GSp}(V,\psi)(\mathbb{A}_f)$ with $K'^p$ small enough and containing $K^p$, such that
$\mathrm{Sh}_{K'}(\mathrm{GSp}(V,\psi),X')$ affords a moduli
interpretation, and that the natural morphism \[f:\mathrm{Sh}_K(G,X)\rightarrow
\mathrm{Sh}_{K'}(\mathrm{GSp}(V,\psi),X')_E\] is a closed embedding.
Let $g=\frac{1}{2}\mathrm{dim}(V)$, and $\mathscr{A}_{g,1,K'}$ be the moduli scheme
of principally polarized abelian schemes over $\mathbb{Z}_{(p)}$-schemes with level $K'^p$ structure. Then
$\mathrm{Sh}_{K'}(\mathrm{GSp}(V,\psi),X')$ is the generic fiber of $\mathscr{A}_{g,1,K'}$, and the integral
canonical model \[\ES_K(G,X)\] is defined to be the normalization\footnote{By the recent work of Xu \cite{Xu}, the normalization step in the construction of $\ES_K(G,X)$ is in fact redundant. } of the Zariski
closure of $\mathrm{Sh}_K(G,X)$ in
$\mathscr{A}_{g,1,K'}\otimes O_{E,(v)}$.
\begin{theorem}[\cite{CIMK} Theorem 2.3.8, \cite{2-adic int} Theorem 4.11]
The $O_{E,(v)}$-scheme $\ES_K(G,X)$ is smooth, and morphisms in the inverse system $\varprojlim_{K^p}\ES_K(G,X)$ are \'{e}tale.
\end{theorem}

The scheme $\ES_K(G,X)$ is uniquely determined by the Shimura datum and the group $K$ in the sense that $\ES_{K_p}(G,X):=\varprojlim_{K^p}\ES_K(G,X)$ satisfies a certain extension property (see \cite{CIMK} 2.3.7 for the precise statement). This implies that the $G(\mathbb{A}_f^p)$-action on $\varprojlim_{K^p}\Sh_K(G,X)$ extends to $\varprojlim_{K^p}\ES_K(G,X)$.

Let $\A\rightarrow \ES_K(G,X)$ be the pull back to $\ES_K(G,X)$ of the universal abelian scheme on $\mathscr{A}_{g,1,K'/\mathbb{Z}_{(p)}}$. Consider the vector bundle \[\V:=\Hdr(\A/\ES_K(G,K))\] over $\A/\ES_K(G,K)$.  There are certain sections of $\V^\otimes$ which will play an important role in this paper. Let $\V_{\Sh_K(G,X)}$ be the base change of $\V$ to $\Sh_K(G,X)$, which is $\Hdr(\A/\Sh_K(G,K))$ by base change of de Rham cohomology. By \cite{CIMK} Proposition 1.3.2 and \cite{2-adic int} Lemma 4.7, there is a tensor $s\in V_{\INT_{(p)}}^\otimes$ defining $\inttG\subseteq \GL(V_{\INT_{(p)}})$. This tensor gives a section $s_{\dr/E}$ of $\V_{\Sh_K(G,X)}^\otimes$, which is actually defined over $O_{E,(v)}$. More precisely, we have the following result.

\begin{proposition}[\cite{CIMK} Corollary 2.3.9, \cite{2-adic int} Proposition 4.8]
The section $s_{\dr/E}$ of $\V_{\Sh_K(G,X)}^\otimes$ extends to a section $s_{\dr}$ of~$\V^\otimes$.
\end{proposition}

Let $\mathbb{D}(\A)$ be the Dieudonn\'{e} crystal of $\A[p^\infty]$, then $s_{\dr}$ (and hence $s$) induces an injection of crystals $s_\cris:\mathbf{1}\rightarrow \mathbb{D}(\A)^\otimes$, such that $s_\cris[\frac{1}{p}]:\mathbf{1}[\frac{1}{p}]\rightarrow \mathbb{D}(\A)^\otimes[\frac{1}{p}]$ is Frobenius equivariant. We will simply call $s_\cris$ a tensor of $\mathbb{D}(\A)^\otimes$.

\subsubsection[]{}\label{conn comp Hodge}We need to work with geometrically connected components. Fix a connected component $X^+\subseteq X$. For a compact open subgroup $K\subseteq G(\mathbb{A}_f)$ as before, i.e. $K=K_pK^p$ with $K_p=\inttG(\mathbb{Z}_p)$ and $K^p\subseteq G(\mathbb{A}_f^p)$ open compact and small enough, we denote by $\Sh_{K}(G,X)^+\subseteq \Sh_{K}(G,X)_{\CPLX}$ the geometrically connected component which is the image of $X^+\times 1$. Then by \cite{CIMK} 2.2.4, $\Sh_{K}(G,X)^+$ is defined over $E^p$, the maximal extension of $E$ which is unramified at $p$. Let $O_{(p)}$ be the localization at $(p)$ of the ring of integers of $E^p$, we write \[\ES_{K}(G,X)^+\] for the closure of $\Sh_{K}(G,X)^+$ in $\ES_{K}(G,X)\otimes O_{(p)}$, and set \[\ES_{K_{p}}(G,X)^+:=\varprojlim_{K^p}\ES_{K}(G,X)^+.\]

Recall that by \cite{CIMK} 3.2 there exists an adjoint action of $G^{\mathrm{ad}}(\RAT)^+$ on $\Sh_{K_p}(G,X)$ induced by conjugation of $G$.
The adjoint action of $G^{\mathrm{ad}}(\INTP)^+$ on $\Sh_{K_p}(G,X)$  extends to $\ES_{K_{p}}(G,X)$. It leaves $\Sh_{K_p}(G,X)^+$ stable, and hence induces an action on $\ES_{K_{p}}(G,X)^+$. We will describe this action following \cite{Kisin-Pappas} in the next subsection.

We remark that the special fiber of $\ES_{K_{p}}(G,X)^+$ is connected. Indeed, by \cite{cpt Hodge type}, it has a smooth compactification $\ES_{K_{p}}(G,X)^+_{\mathrm{tor}}$ such that the boundary is either empty or a relative divisor. Let $H^0$ be the ring of regular functions on $\ES_{K_{p}}(G,X)^+_{\mathrm{tor}}$. It is a finite $O_{(p)}$-algebra in $E^p$. Noting that $H^0$ is normal, we have $H^0=O_{(p)}$. By Zariski's connectedness theorem, the special fiber of $\ES_{K_{p}}(G,X)^+_{\mathrm{tor}}$ is connected, and hence that of $\ES_{K_{p}}(G,X)^+$ is connected.

\subsection[Integral models for Shimura varieties of abelian type]{Integral models for Shimura varieties of abelian type}

Recall that a Shimura datum $(G, X)$ is said to be of abelian type, if there is a Shimura datum of Hodge type $(G_1, X_1)$ and a central isogeny $G_1^{\mathrm{der}}\rightarrow G^{\mathrm{der}}$ which induces an isomorphism of adjoint Shimura data $(G_1^{\mathrm{ad}},X_1^{\mathrm{ad}})\st{\sim}{\rightarrow} (G^{\mathrm{ad}},X^{\mathrm{ad}})$.

\subsubsection[]{}In order to explain the construction of integral canonical models for Shimura varieties of abelian type, and also for the convenience of the next subsection, we recall briefly Kisin's construction of twisting abelian varieties. The main reference is \cite{Kisin-Pappas} 4.4.

Let $R$ be a commutative ring, $Z$ be a flat affine group scheme over $\Spec R$, and $\mathcal{P}$ be a $Z$-torsor. Then $\mathcal{P}$ is flat and affine. We write $O_Z$ and $O_{\mathcal{P}}$ for the ring of regular functions on $Z$ and $\mathcal{P}$ respectively. Let $M$ be a $R$-module with $Z$-action, i.e. a homomorphism of fppf sheaves of groups $Z\rightarrow \AAut(M)$, then the subsheaf $M^Z$ is a $R$-submodule of $M$. By \cite{Kisin-Pappas} Lemma 4.4.3, the natural homomorphism \begin{subeqn}\label{baisc twist}
(M\otimes_RO_{\mathcal{P}})^Z\otimes_RO_{\mathcal{P}}\rightarrow M\otimes_RO_{\mathcal{P}}
\end{subeqn} is an isomorphism.

\subsubsection[]{}\label{twist abs} Let $R\subseteq \RAT$ be a normal subring. For a scheme $S$, we define the $R$-isogeny category of abelian schemes over $S$ to be the category of abelian schemes over $S$ by tensoring the $\mathrm{Hom}$ groups by $\otimes_{\INT}R$. An object $\A$ in this category is called an abelian scheme up to $R$-isogeny over $S$. For $T$ an $S$-scheme, we set $\A(T)=\mathrm{Mor}_S(T,\A)\otimes_{\INT} R$. We will write $\uAut_R(\A)$ for the $R$-group whose points in an $R$-algebra $A$ are given by \[\uAut_R(\A)(A)=((\mathrm{End}_S\A)\otimes_RA)^\times.\]

Now we assume that $Z$ is of finite type over $R\subseteq \RAT$. Suppose that we are given a homomorphism of $R$-groups $Z\rightarrow \uAut_R(\A)$, we define a pre-sheaf $\A^{\mathcal{P}}$ by setting \[\A^{\mathcal{P}}(T)=(\A(T)\otimes_RO_{\mathcal{P}})^Z.\] By \cite{Kisin-Pappas} Lemma 4.4.6, $\A^{\mathcal{P}}$ is a sheaf, represented by an abelian scheme up to $R$-isogeny.

\subsubsection[]{}\label{group-compn to whole}Before describing the construction of integral canonical models for Shimura varieties of abelian type, we need to fix some notations. Let $H/\INTP$ be a reductive group. For a subgroup $A\subseteq H(\INTP)$, we write $A_+$ for the pre-image in $A$ of $H^\mathrm{ad}(\mathbb{R})^+$, the connected component of identity in $H^\mathrm{ad}(\mathbb{R})$; and $A^+$ for $A\cap H(\mathbb{R})^+$.  We write $H(\INTP)^-$ (resp. $H(\INTP)^-_+$) for the closure of $H(\INTP)$ (resp. $H(\INTP)_+$) in $H(\mathbb{A}_f^p)$. Let $Z$ be the center of $H$, we set $$\mathscr{A}(H)=H(\mathbb{A}_f^p)/Z(\INTP)^-*_{H(\INTP)_+/Z(\INTP)}H^{\mathrm{ad}}(\INTP)^+$$
and
$$\mathscr{A}(H)^\circ=H(\INTP)_+^-/Z(\INTP)^-*_{H(\INTP)_+/Z(\INTP)}H^{\mathrm{ad}}(\INTP)^+,$$
where $X*_YZ$ is the quotient of $X\rtimes Z$ defined in \cite{varideshi} 2.0.1. By \cite{varideshi} 2.0.12 and \cite{CIMK} 3.3.2, $\mathscr{A}(H)^\circ$ depends only on $H^\mathrm{der}$ and not on $H$.

Now we turn to the construction of integral models.
\subsubsection[]{} Let $(G,X)$ be a Shimura datum of abelian type with good reduction at $p$. By \cite{CIMK} Lemma 3.4.13, there is a Shimura datum of Hodge type $(G_1,X_1)$ with good reduction at $p$, such that there is a central isogeny $G_1^{\mathrm{der}}\rightarrow G^{\mathrm{der}}$ inducing an isomorphism of Shimura data $(G_1^{\mathrm{ad}},X_1^{\mathrm{ad}})\st{\sim}{\rightarrow} (G^{\mathrm{ad}},X^{\mathrm{ad}})$. Let $G_{\INTP}$ be a reductive group over $\INTP$ with generic fiber $G$. By the proof of \cite{CIMK} Corollary 3.4.14, there exists a reductive model $G_{1,\INTP}$ of $G_1$ over $\INTP$, such that the central isogeny $G_1^{\mathrm{der}}\rightarrow G^{\mathrm{der}}$ extends to a central isogeny $G_{1,\INTP}^{\mathrm{der}}\rightarrow G_{\INTP}^{\mathrm{der}}$.

We can now follow discussions as in \ref{conn comp Hodge}. Let $X^+_1\subseteq X_1$ be a connected component. For $K_1=K_{1,p}K_1^p$, let $\Sh_{K_1}(G_1,X_1)^+\subseteq \Sh_{K_1}(G_1,X_1)$ be the geometrically connected component which is the image of $X^+_1\times 1$. Then $\Sh_{K_1}(G_1,X_1)^+$ is defined over $E_1^p$, where $E_1$ is the reflex field of $(G_1,X_1)$, and $E_1^p$ is the maximal extension of $E_1$ which is unramified at $p$. Let $O_{(p)}$ be the localization at $(p)$ of the ring of integers of $E_1^p$, we write \[\ES_{K_1}(G_1,X_1)^+\] for the closure of $\Sh_{K_1}(G_1,X_1)^+$ in $\ES_{K_1}(G_1,X_1)\otimes O_{(p)}$, and \[\ES_{K_{1,p}}(G_1,X_1)^+:=\varprojlim_{K_1^p}\ES_{K_1}(G_1,X_1)^+.\] The $G_1^{\mathrm{ad}}(\INTP)^+$-action on $\Sh_{K_{1,p}}(G_1,X_1)^+$ extends to $\ES_{K_{1,p}}(G_1,X_1)^+$, which (after converting to a right action) induces an action of $\mathscr{A}(G_{1,\INTP})^\circ$ on $\ES_{K_{1,p}}(G_1,X_1)^+$. Here $\mathscr{A}(G_{1,\INTP})^\circ$ is as we introduced in \ref{group-compn to whole}.

The action of $G_1^{\mathrm{ad}}(\INTP)^+$ on $\ES_{K_{1,p}}(G_1,X_1)$  is described in \cite{Kisin-Pappas} as follows. Let $(\A,\lambda, \varepsilon)$ be the pull back to $\ES_{K_{1,p}}(G_1,X_1)$ of the universal abelian scheme (up to $\INTP$-isogeny) with weak $\INTP$-polarization and level structure, and $Z$ be the center of $G_{1,\INTP}$. By \cite{Kisin-Pappas} Lemma 4.5.2, there is a natural embedding \[Z\rightarrow \underline{\mathrm{Aut}}_{\INTP}(\A),\] where $\underline{\mathrm{Aut}}_{\INTP}(\A)$ is as in \ref{twist abs}. For $\gamma\in G^{\mathrm{ad}}(\INTP)^+$, and $\fP$ the fiber of $G_{1,\INTP}\rightarrow G^{\mathrm{ad}}_{\INTP}$ over $\gamma$, by \ref{twist abs} again, we have $\A^{\fP}$, an abelian scheme up to $\INTP$-isogeny. Moreover, by \cite{Kisin-Pappas} Lemma 4.4.8 (resp. Lemma 4.5.4), $\lambda$ (resp. $\varepsilon$) induces a weak $\INTP$-polarization $\lambda^{\fP}$(resp. level structure $\varepsilon^{\fP}$) on $\A^{\fP}$. By \cite{Kisin-Pappas} Lemma 4.5.7, this gives a morphism \[\ES_{K_{1,p}}(G_1,X_1)\rightarrow \ES_{K_{1,p}}(G_1,X_1),\] such that on generic fiber it agrees with the morphism induced by conjugation by $\gamma$. This action stabilizes 
$\ES_{K_{1,p}}(G_1,X_1)^+$.

\begin{theorem}\label{int can abelian type} The quotient $$\ES_{K_{p}}(G,X):=[\mathscr{A}(G_{\INTP})\times \ES_{K_{1,p}}(G_1,X_1)^+]/\mathscr{A}(G_{1,\INTP})^\circ$$
is represented by a scheme over $O_{(p)}$ which descends to $O_{E,(v)}$. Moreover, it is the integral canonical model of $\Sh_{K_{p}}(G,X)$.
\end{theorem}
\begin{proof}This is \cite{CIMK} Theorem 3.4.10 when $p>2$, and  \cite{2-adic int} Theorem 4.11 when $p=2$. See also the Errata for [Ki 2] in \cite{LRKisin} for a fully corrected proof.
\end{proof}

We have also\[\ES_{K_{p}}(G,X)
=[\mathscr{A}(G_{\INTP})\times \ES_{K_{p}}(G,X)^+]/\mathscr{A}(G_{\INTP})^\circ,\] where $\ES_{K_{p}}(G,X)^+\subset \ES_{K_{p}}(G,X)$ is a geometrically connected component over $O_{(p)}$ given by
\[\ES_{K_{p}}(G,X)^+:=\ES_{K_{1,p}}(G_1,X_1)^+/\Delta\]  with
\[\Delta:=\mathrm{Ker}(\mathscr{A}(G_{1,\INTP})^\circ\ra  \mathscr{A}(G_{\INTP})^\circ).\]
For each open compact subgroup $K^p\subset G(\mathbb{A}_f^p)$ which is small enough, we get the integral canonical model \[\ES_{K_pK^p}(G,X):=\ES_{K_{p}}(G,X)/K^p\] of $\Sh_{K_pK^p}(G, X)$.
In this paper, we are mainly interested in the geometry of the special fiber \[\ES_{K_{p},0}(G,X)\] of $\ES_{K_{p}}(G,X)$, i.e. the special fibers $\ES_{K_pK^p,0}(G,X)$ of $\ES_{K_pK^p}(G,X)$ when $K^p$ varies, so we will sometimes work with $\ES_{K_{p}}(G,X)\otimes O_{E,v}$. Here $O_{E,v}$ is the $p$-adic completion of $O_{E,(v)}$. 

We consider the following example of Shimura varieties of abelian type, which will be investigated continuously in the rest of this paper. Another interesting example will be given in section \ref{S:orthogonal}.

\begin{example}\label{quaternion SHV}
Let $D$ be a quaternion algebra over a totally real extension $F$ of $\mathbb{Q}$ of degree $n$. Let $\infty_1,\infty_2,\cdots, \infty_d$ be the infinite places of $F$ at which $D$ is split. We will always assume that $d>0$ in the discussion. Let $G=\mathrm{Res}_{F/\mathbb{Q}}(D^\times)$ and \[h:\mathbb{S}\rightarrow \mathrm{GL}_{2,\mathbb{R}}^d\subseteq D^\times_\mathbb{R}=G_\mathbb{R}\] be the homomorphism given by $z\mapsto (z,z,\cdots, z)\in \mathrm{GL}_{2,\mathbb{R}}^d$. One checks easily that $h$ induces a Shimura datum denoted by $(G, X)$. The associated Shimura variety is of dimension $d$, and it is defined over the totally real number field \[E=\mathbb{Q}(\sum_{i=1}^d\infty_i(f)\mid f\in F)\subseteq \mathbb{C},\] here we view $\infty_i$ as an embedding $F\rightarrow \mathbb{R}$.

If $d=n$, then $(G, X)$ is of PEL type; and if $d< n$, it is of abelian type but not of Hodge type, as the weight cocharacter is not defined over $\mathbb{Q}$. We are mainly interested in the second case here. By \cite{semi-simple zeta} Part I $\S$1, fixing an imaginary quadratic extension $K/F$ together with a subset $P_K$ of archimedean places of $K$ such that the restriction to $F$ induces a bijection of from $P_K$ to $\{\infty_{d+1}, \infty_{d+2},\cdots \infty_n\}$, then one can construct a PEL (coarse) moduli variety $M_{\overline{C}}/E'$ with an open and closed embedding \[\Sh_C(G,X)\otimes E' \hookrightarrow \widetilde{M}_{\overline{C}}.\] Here $E'\supseteq E$ is the reflex field of the zero-dimensional Shimura datum determined by $K$ and $P_K$, $G'$ a certain unitary group associated to $D$ and $K$, $\overline{C}\subset G'(\mathbb{A}_f)$ is the open compact subgroup of $G'(\mathbb{A}_f)$ associated to $C$, and $\widetilde{M}_{\overline{C}}$ is a certain twist of $M_{\overline{C}}$, cf. \cite{semi-simple zeta} p. 11-13.  

If moreover $D$ is split at $p$, the integral canonical model can be constructed as follows. Let $v$ be a place of $E$ over $p$, and $O_{E,v}$ be the $p$-adic completion of the ring of integers at $v$. By assumption $G$ is hyperspecial at $p$, and we set $C_p:=G(\mathbb{Z}_p)$ and we consider open compact subgroups of $G(\mathbb{A}_f)$ in the form $C=C_pC^p$. Consider the pro-varieties over $E_v$: \[M_{C_p}=\varprojlim_{C^p}M_{\overline{C_pC^p}}, \quad \widetilde{M}_{C_p}=\varprojlim_{C^p}\widetilde{M}_{\overline{C_pC^p}}.\]
By \cite{semi-simple zeta} Part I $\S$2, one can choose $K$ and $P_K$, such that $E'\subseteq E_{v}$, and $M_{C_p}$ has an integral model $\mathcal{M}/O_{E,v}$, which is smooth (thus it is the integral canonical model) by our assumption that $D$ is split at $p$. Indeed, one can check that $\mathcal{M}$ coincides with Kisin's construction of canonical integral models for general Hodge type Shimura varieties, cf. \cite{CIMK}. We get a twist $\widetilde{\mathcal{M}}$ of $\mathcal{M}$ with generic fiber $\widetilde{M}_{C_p}$. By construction we have an open and closed embedding \[\Sh_{C_p}(G,X)_{E_v}\hookrightarrow \widetilde{M}_{C_p}.\] The integral model of $\Sh_{C_p}(G,X)_{E_v}$ is then its closure in $\widetilde{\mathcal{M}}$.
\end{example}

\subsection[Twisting $p$-divisible groups]{Twisting $p$-divisible groups}

In order to study stratifications induced by $p$-divisible groups, it will be helpful to have a theory of twisting $p$-divisible group. For our applications, it suffices to think about $p$-divisible groups coming from abelian schemes. But we insist to give a general theory here, as it might be useful to study general Rapoport-Zink spaces.

\subsubsection[]{}\label{twist p-div notations}Consider the setting of \ref{twist abs} with $R=\COMP$. We will fix a group scheme $Z$ over $\Spec R$ which is flat, affine and of finite type as well as a $Z$-torsor $\mathcal{P}$ over $R$. Their rings of regular functions will be denoted by $O_Z$ and $O_{\mathcal{P}}$ respectively.

Let $\D$ be a $p$-divisible group over a scheme $S$. Then $\mathrm{End}_S\D$ is a $R$-module. We will write $\uAut_R(\D)$ for the $R$-group whose points in an $R$-algebra $A$ are given by \[\uAut_R(\D)(A)=((\mathrm{End}_S\D)\otimes_RA)^\times.\]

Suppose now that we are given a homomorphism of $R$-groups $Z\rightarrow \uAut_R(\D)$. For each positive integer $n$, we define a pre-sheaf $\D^{\mathcal{P}}[p^n]$ by setting $$\D^{\mathcal{P}}[p^n](T)=(\D[p^n](T)\otimes_RO_{\mathcal{P}})^Z.$$ They form a direct system denoted by $\D^{\mathcal{P}}$.
\begin{proposition} $\D^{\mathcal{P}}[p^n]$ is represented by a truncated $p$-divisible group of level $n$ over $S$, and $\D^{\mathcal{P}}$ is a $p$-divisible group.\end{proposition}
\begin{proof}
We proceed as in \cite{Kisin-Pappas} Lemma 4.4.6, and take a finite, integral, torsion free $R$-algebra $R'$ such that $\mathcal{P}(R')$ is non-empty. Specializing \ref{baisc twist} by the map $O_{\mathcal{P}}\rightarrow R'$, we obtain an isomorphism $\D^{\mathcal{P}}[p^n]\otimes_R R'\cong\D[p^n]\otimes_R R'$. $\D^{\mathcal{P}}[p^n]\otimes_R R'$ is a truncated $p$-divisible group of level $n$ as $\D[p^n]\otimes_R R'$ is isomorphic to the sum of $[R':R]$ copies of $\D[p^n]$.

We may assume that $\mathrm{Fr}(R')$ is Galois over $\RAT_p$, then $\D^{\mathcal{P}}[p^n]$ is the $\mathrm{Gal}(\mathrm{Fr}(R')/\RAT_p)$-invariants of $\D^{\mathcal{P}}[p^n]\otimes_R R'$. So $\D^{\mathcal{P}}[p^n]$ is the kernel of a homomorphism of truncated $p$-divisible groups of level $n$, and hence is a group scheme over $S$. It is necessarily flat as it is a direct summand of $\D^{\mathcal{P}}[p^n]\otimes_R R'$. By the same argument, after applying the exact funtor $(\ )^{\mathcal{P}}$ to \[\xymatrix{0\ar[r]&\D[p^{n-i}]\ar[r]&\D[p^n]\ar[r]^{p^{n-i}}&\D[p^i]\ar[r]&0,}\] we have an exact sequence \[\xymatrix{0\ar[r]&\D^{\mathcal{P}}[p^{n-i}]\ar[r]&\D^{\mathcal{P}}[p^n]\ar[r]^{p^{n-i}}&\D^{\mathcal{P}}[p^i]\ar[r]&0}.\]
This implies that $\D^{\mathcal{P}}$ is a $p$-divisible group.
\end{proof}
\begin{remark}
The ways that we twist abelian schemes and $p$-divisible groups are compatible. More precisely, notations and hypothesis as in \ref{twist abs}, but with $R\subseteq \INTP$. Let $R'=\COMP$ and $\D=\A[p^\infty]$. The map $Z\rightarrow \uAut_R(\A)$ induces a map $Z_{R'}\rightarrow \uAut_{R'}(\D)$, and we have $\A^{\mathcal{P}}[p^\infty]=\D^{\mathcal{P}_{R'}}$.
\end{remark}
\subsubsection[]{}\label{twist p-div with addi struc}We will need to work with $p$-divisible groups with additional structure. Notations as in \ref{twist p-div notations}, we assume that $S$ is an integral scheme which is flat over $\mathbb{Z}_{(p)}$, and that $Z$ is \textsl{smooth with connected fibers}. Let $T_p(\D)$ be the $p$-adic Tate module of $\D$ over the generic point of $S$, and $t\in T_p(\D)^\otimes$ be a $Z$-invariant tensor. Using the proof of \cite{LRKisin} Lemmas 4.1.7 and 4.1.5, we have a canonical isomorphism $T_p(\D^\mathcal{P})\cong T_p(\D)^\mathcal{P}$, and the tensor $t\in T_p(\D)^\otimes$ is naturally an element of $T_p(\D^\mathcal{P})^\otimes$.

\begin{corollary}\label{iso of kisin twist}Assumptions as above, there exists an isomorphism $\D^{\mathcal{P}}\cong \D$ respecting $t$.\end{corollary}
\begin{proof}
Noting that $Z$ is smooth with connected fibers, $\mathcal{P}$ is a trivial $Z$-torsor. Indeed, by Lang's theorem (\cite{gpfinifield}) the special fiber $\mathcal{P}_{\mathbb{F}_p}$ is a trivial $Z_{\mathbb{F}_p}$-torsor, and the rational points on $\mathcal{P}_{\mathbb{F}_p}$ lift to rational points of $\mathcal{P}$.  Specializing \ref{baisc twist} at $w\in \mathcal{P}(R)$, we get an isomorphism $\D^{\mathcal{P}}\cong \D$. It is by definition that its induced map on Tate modules respects $t$.
\end{proof}

\section[Newton stratifications]{Newton stratifications}

We study the Newton stratifications on the special fibers of the Shimura varieties introduced in the last section.

\subsection[Group theoretic preparations]{Group theoretic preparations}\label{group theo settings for NP}
Let $G$ be a reductive group over $\mathbb{Z}_p$, and $\mu$ be a cocharacter of $G$ defined over $W(\kappa)$ with $\kappa|\mathbb{F}_p$ a finite field. Let $W=W(\overline{\kappa})$, $L=W[1/p]$ and $\sigma$ be the Frobenius on them. We need the following objects. Let $C(G)$ (resp. $B(G)$) be the set of $G(W)$-$\sigma$-conjugacy (resp. $G(L)$-$\sigma$-conjugacy) classes in $G(L)$, $C(G,\mu)$ be the set of $G(W)$-$\sigma$-conjugacy classes in $G(W)\mu(p)G(W)$, and $B(G,\mu)$ be the image of \[C(G,\mu)\hookrightarrow C(G)\twoheadrightarrow B(G).\] The set $B(G)$ parametrizes isomorphism classes of $F$-isocrystals with $G$-structure over an algebraically closed field of characteristic $p$, cf. \cite{class of F-isocrys} Remark 3.4 (i).

\subsubsection[]{}\label{gp theo NP}Let $T$ be a maximal torus of $G$, and $X_*(T)$ be its group of cocharacters. Let $\pi_1(G)$ be the quotient of $X_*(T)$ by the coroot lattice, and $W_G$ be the Weyl group of $G$. Since $G$ is unramified, we can fix a Borel subgroup $T\subset B\subset G$. To a $G(L)$-$\sigma$-conjugacy class $[b]\in B(G)$, Kottwitz defines two functorial invariants \[\nu_G([b])\in (X_*(T)_{\mathbb{Q}}/W_G)^\Gamma\cong X_*(T)_{\mathbb{Q},\mathrm{dom}}^\Gamma\] and \[\kappa_G([b])\in \pi_1(G)_\Gamma\] in \cite{isocys with addi}. Here $\Gamma=\mathrm{Gal}(\overline{\mathbb{Q}}_p/\mathbb{Q}_p)$, and $X_*(T)_{\mathbb{Q},\mathrm{dom}}\subset X_*(T)_{\mathbb{Q}}$ is the cone spanned by dominant coweights corresponding to $B$. These two invariants determines $[b]$ uniquely. In the following, we will also write $\nu_G(b)$ and $\kappa_G(b)$ for an element $b\in G(L)$ for the two invariants of $[b]\in B(G)$, the $G(L)$-$\sigma$-conjugacy class of $b$.

We consider the partial order $\leq$ on $X_*(T)_{\mathbb{Q}}$ given by $\chi'\leq\chi$ if and only if $\chi-\chi'$ is a linear combination of non-negative coroots with positive rational coefficients. We write $\overline{\mu}$ for the average of the $\Gamma$-orbit of $\mu$. By \cite{class of F-isocrys} Theorem 4.2, we have $\nu_G(b)\leq \overline{\mu}$ and $\kappa_G(b)=\mu_*$ for $b\in G(W)\mu(p)G(W)$. Here $\mu_*$ is the image of $\mu$ in $\pi_1(G)_\Gamma$. By works of Gashi, Kottwitz, Lucarelli, Rapoport and Richartz, we have (see \cite{VW} 8.6)) $$B(G,\mu)=\{[b]\in B(G)\mid \nu_G(b)\leq \overline{\mu}\text{ and }\kappa_G(b)=\mu_*\}.$$ The partial order $\leq$ on $X_*(T)_{\mathbb{Q}}$ induces a partial order on the set $B(G,\mu)$, denoted also by $\leq$.

\begin{remark}
One can define for any algebraically closed field $k\supseteq \mathbb{F}_p$ a set $B'(G)$ exactly as how we define $B(G)$. But by \cite{class of F-isocrys} Lemma 1.3, the obvious map $B(G)\rightarrow B'(G)$ is bijective.
\end{remark}
\begin{remark}\label{R:ordinary-basic}
There is a unique maximal (resp. minimal) element in $B(G,\mu)$. For a variety $X/\kappa$ with a map $X(\overline{\kappa})\rightarrow B(G,\mu)$, the preimage of this element is called the $\mu$-ordinary locus (resp. basic locus).
\end{remark}

To each $G(L)$-$\sigma$-conjugacy class $[b]$, one defines $M_b$ to be the centralizer in $G$ of $\nu_G(b)$, and $J_b$ be the group scheme over $\mathbb{Q}_p$ such that for any $\mathbb{Q}_p$-algebra $R$, $$J_b(R)=\{g\in G(R\otimes_{\mathbb{Q}_p}L)\mid gb=b\sigma(g)\}.$$
The group $J_b$ is an inner form of $M_b$ which, up to isomorphism, does not depend on the choices of representatives in $[b]$ (see \cite{isocys with addi} 5.2). Kottwitz introduced the notion of defect in \cite{Kott}, based on earlier work of Chai \cite{Chai}.
\begin{definition}\label{D:def}
For $[b]\in B(G)$, the defect of $[b]$ is defined by $$\mathrm{def}_G(b)=\mathrm{rank}_{\mathbb{Q}_p}G-\mathrm{rank}_{\mathbb{Q}_p}J_b.$$
\end{definition}
Hamacher gives a formula for $\mathrm{def}_G(b)$ using root data.
\begin{proposition}[\cite{geo of newt PEL} Proposition 3.8]
Let $w_1,\cdots,w_l$ be the sums over all elements in a Galois orbit of absolute fundamental weights of $G$. For $[b]\in B(G)$, we have $$\mathrm{def}_G(b)=2\cdot\sum_{i=1}^l\{\langle \nu_G(b),w_i\rangle\},$$
where $\{\cdot\}$ means the fractional part of a rational number.
\end{proposition}

\subsection[Newton stratifications on Shimura varieties of Hodge type]{Newton stratifications on Shimura varieties of Hodge type}\label{subsection newton Hodge}
No surprisingly, Newton strata on Shimura varieties of abelian type are, in some manner, induced by those on Shimura varieties of Hodge type. So we will first recall definition of Newton strata on Shimura varieties of Hodge type.

\subsubsection[]{}\label{notation NP 1}Notations as in \ref{model Hodge type}. Let $\kappa$ be the residue field of $O_{E,(v)}$. The Hodge type Shimura datum $(G,X)$ determines a $G$-orbit of cocharacters. It extends uniquely to a $G_{\COMP}$-orbit of cocharacters, and hence has a representative $\mu:\mathbb{G}_m\rightarrow G_{W(\kappa)}$ which is unique up to conjugacy. We remark that $\mu$ has weights $0$ and $1$ on $V_{\COMP}^\vee\otimes W(\kappa)$. 

Let $W=W(\overline{\kappa})$ and $L=W[1/p]$. Let $K=K_pK^p$ with $K_p=G(\mathbb{Z}_p)$. For $z\in \ES_K(G,X)(\overline{\kappa})$, we will simply write $D_z$ for $\mathbb{D}(\A_z[p^\infty])(W)$. In fact, we have an $F$-crystal $\mathbb{D}(\A[p^\infty])$ with a crystalline Tate tensor $s_{\cris}$ over $\ES_{K,0}(G,X)$, the special fiber of $\ES_K(G,X)$. On a point $x\in \ES_K(G,X)(\overline{\kappa})$ it gives rise to $(D_x, s_{\cris,x})$. Two points $x,y\in \ES_K(G,X)(\overline{\kappa})$ are said to be in the same \emph{Newton stratum} if there exists an isomorphism of $F$-isocrystals \[D_x\otimes L\rightarrow D_y\otimes L\] mapping $s_{\cris, x}$ to $s_{\cris, y}$. 

For $x\in \ES_K(G,X)(\overline{\kappa})$, choosing an isomorphism $t:V_{\COMP}^\vee\otimes W\rightarrow D_x$ mapping $s$ to $s_{\cris, x}$, we get a Frobenius on $V_{\COMP}^\vee\otimes W$ which is of the form $(\mathrm{id}\otimes \sigma)\circ g_{x,t}$ with  $g_{x,t}$ lies in $G(W)\mu(p)G(W)$. Moreover, changing $t$ to another isomorphism $V_{\COMP}^\vee\otimes W\rightarrow D_x$ mapping $s$ to $s_{\cris, x}$ amounts to $G(W)$-$\sigma$-conjugacy of $g_{x,t}$. So we have a well defined map \[\ES_K(G,X)(\overline{\kappa})\rightarrow C(G,\mu).\] Similarly, changing $t$ to another isomorphism $V_{\COMP}^\vee\otimes L\rightarrow D_x\otimes L$ mapping $s$ to $s_{\cris, x}$ amounts to $G(L)$-$\sigma$-conjugacy of $g_{x,t}$ (in $B(G)$), and we have a well defined map \[\ES_K(G,X)(\overline{\kappa})\rightarrow B(G,\mu).\] It is clear that $x,y\in \ES_K(G,X)(\overline{\kappa})$ are in the same Newton stratum if and only if they have the same image in $B(G,\mu)$.

Before stating the results about Newton strata on Shimura varieties of Hodge type, we need to fix some notations. When there is no confusion about the level $K$ and the Shimura datum $(G, X)$, we simply denote by $\ES_0=\ES_{K,0}(G,X)$ the special fiber of $\ES_K(G,X)$. For $[b]\in B(G,\mu)$, we will write $\ES_0^b$ for the Newton stratum corresponding to it. It is, a priori, just a subset of $\ES_0(\overline{\kappa})$.

\begin{theorem}\label{Newton for Hodge type}
The Newton stratum $\ES_0^b$ is a non-empty equi-dimensional locally closed subscheme of $\ES_0$ of dimension \[\langle\rho,\mu+\nu_G(b)\rangle-\frac{1}{2}\mathrm{def}_G(b).\] Here $\rho$ is the half-sum of positive roots of $G$. Moreover, $\overline{\ES_0^b}$, the closure of $\ES_0^b$, is the union of strata $\ES_0^{b'}$ with $[b']\leq [b]$, and $\overline{\ES_0^b}-\ES_0^b$ is either empty or pure of codimension 1 in $\overline{\ES_0^b}$.
\end{theorem}
\begin{proof}
That $\ES_0^b$ is locally closed follows from \cite{class of F-isocrys} Theorem 3.6. Let $b_0\in B(G,\mu)$ be the basic element. The non-emptiness of $\ES_0^{b_0}$ is proved by Dong-Uk Lee, Kisin-Madapusi Pera and Chia-Fu Yu respectively, one can see for example \cite{Newton non-epty}. Fixing $x\in \ES^{b_0}_0(\overline{\kappa})$, let $X(\mu,b_0)$ be the affine Deligne-Lusztig variety attached to $b_0$, we consider the uniformization map $\tau_x: X(\mu,b_0)\rightarrow \mathscr{A}_{g,K'}$.  The dimension of $\ES^{b_0}_0$ is no bigger than that of the image of $\tau_x$, which is $\langle\rho,\mu+\nu_G(b_0)\rangle-\frac{1}{2}\mathrm{def}_G(b_0)$ by \cite{zhu-aff gras in mixed char}. But then the theorem holds by \cite{Vieh} Lemma 5.12.  When $p>2$, the dimension formula is also given in \cite{Foliation-Hamacher} and \cite{level m}.
\end{proof}

When the prime to $p$ level $K^p$ varies, by construction the Newton strata $\ES_{K_pK^p, 0}^b$ are invariant under the prime to $p$ Hecke action. In this way we get also the Newton stratification on $\ES_{K_p,0}=\varprojlim_{K^p}\ES_{K_pK^p,0}$.

\subsection[Newton stratifications on Shimura varieties of abelian type]{Newton stratifications on Shimura varieties of abelian type}
The guiding idea of our construction is as follows. Let $(G,X)$ be a Shimura datum of abelian type with good reduction at $p$, $K^p\subset G(\mathbb{A}_f^p)$ be a sufficiently small open compact subgroup, and $\ES_{K,0}(G,X)$ be the special fiber of the associated integral canonical model (with $K=K_pK^p, K_p=G(\mathbb{Z}_p)$). In order to define a stratification on $\ES_{K,0}(G,X)$, the easiest way (and also the most direct way) one could think about is to do this for $\ES_{K^\adj,0}(G^\adj,X^\adj)$ first, where $K^\adj=G^\adj(\mathbb{Z}_p)K^{p,\adj}\subset G^{\adj}(\mathbb{A}_f)$ containes the image of $K$ under the induced map $G(\mathbb{A}_f)\ra G^{\adj}(\mathbb{A}_f)$,  and then pull it back via \[\ES_{K,0}(G,X)\rightarrow\ES_{K^\adj,0}(G^\adj,X^\adj).\] The goal of this subsection is to explain how to define and study Newton stratifications for Shimura varieties of abelian type via this ``passing to adjoints'' approach.

We would like to begin with the following lemma, which says that if one wants to use $B(G,\mu)$ to parameterize all the Newton strata, then he could pass to the adjoint group freely. 
\begin{lemma}
Let $f:G\rightarrow H$ be a central isogeny of reductive groups over $\mathbb{Z}_p$,  $\mu$ a cocharacter of $G$ defined over $W(\kappa)$ with $\kappa|\mathbb{F}_p$ finite, and $\mu_H=f\circ \mu$ the associated cocharater of $H$. Then
the map $B(G,\mu)\rightarrow B(H,\mu_H)$ is a bijection respecting partial orders.
\end{lemma}
\begin{proof}
This follows from \cite{isocys with addi 2} 6.5.
\end{proof}

The technical starting point is the following result of Kisin. It implies that for an adjoint Shimura datum of abelian type with good reduction at $p$, one can always realize it as the adjoint Shimura datum of a Hodge type one with very good properties.

\begin{lemma}[\cite{LRKisin} Lemma 4.6.6]\label{Kisin's lemma}
Let $(G,X)$ be a Shimura datum of abelian type with $G$ an adjoint group. Then there exists a Shimura datum of Hodge type $(G_1,X_1)$ such that
\begin{enumerate}
\item $(G_1^\adj,X_1^\adj)\stackrel{\sim}{\longrightarrow} (G,X)$ and $Z_{G_1}$ is a torus; moreover, for any other Hodge type datum $(G_2,X_2)$ with $(G_2^\adj,X_2^\adj)\stackrel{\sim}{\longrightarrow} (G,X)$, $G_2^{\mathrm{der}}$ is a quotient of $G_1^{\mathrm{der}}$;

\item if $(G,X)$ has good reduction at $p$, then $(G_1,X_1)$ in (1) can be chosen to have good reduction at $p$, and such that $E(G,X)_p=E(G_1,X_1)_p$.
\end{enumerate}
\end{lemma}

\subsubsection[]{}\label{adjoint Newton} Let $(G,X)$ be an \emph{adjoint} Shimura datum of abelian type with good reduction at $p$, and $(G_1,X_1)$ be a Shimura datum of Hodge type satisfying the two conditions in the above lemma. Then the center of $G_{1,\INTP}$ is a torus.

Consider $\ES_{K_p}(G,X)$. By Theorem \ref{int can abelian type}, it is given by
\[\begin{split}
\ES_{K_p}(G,X)&=[\mathscr{A}(G_{\INTP})\times \ES_{K_{1,p}}(G_1,X_1)^+]/\mathscr{A}(G_{1,\INTP})^\circ\\
&=[\mathscr{A}(G_{\INTP})\times \ES_{K_{p}}(G,X)^+]/\mathscr{A}(G_{\INTP})^\circ,\end{split}\] 
where on connected components we have
\[\ES_{K_{p}}(G,X)^+=\ES_{K_{1,p}}(G_1,X_1)^+/\Delta\]  with
\[\Delta=\mathrm{Ker}(\mathscr{A}(G_{1,\INTP})^\circ\ra  \mathscr{A}(G_{\INTP})^\circ).\]
By the last subsection, there is a Newton stratification on $\ES_{K_{1,p},\overline{\kappa}}(G_1,X_1)$. We can restrict it to $\ES_{K_{1,p},\overline{\kappa}}(G_1,X_1)^+$ and then extend it trivially to $\mathscr{A}(G_{\INTP})\times \ES_{K_{1,p},\overline{\kappa}}(G_1,X_1)^+$. We will sometimes call this the \emph{induced Newton stratification} on $\mathscr{A}(G_{\INTP})\times \ES_{K_{1,p},\overline{\kappa}}(G_1,X_1)^+$. Similarly for $\mathscr{A}(G_{1,\INTP})\times \ES_{K_{1,p},\overline{\kappa}}(G_1,X_1)^+$.

\begin{proposition}\label{adjoint Newton-result}The induced Newton stratification on $\mathscr{A}(G_{\INTP})\times \ES_{K_{1,p},\overline{\kappa}}(G_1,X_1)^+$ (resp. $\mathscr{A}(G_{1,\INTP})\times \ES_{K_{1,p},\overline{\kappa}}(G_1,X_1)^+$) is $\mathscr{A}(G_{1,\INTP})^\circ$-stable. Moreover, the induced Newton stratification on  $\mathscr{A}(G_{1,\INTP})\times \ES_{K_{1,p},\overline{\kappa}}(G_1,X_1)^+$ descends to the Newton stratification on $\ES_{K_{1,p},0}(G_1,X_1)$.
\end{proposition}
\begin{proof}
To see the first statement, for $([g,h],x)\in \mathscr{A}(G_{\INTP})\times \ES_{K_{1,p,\overline{\kappa}}}(G_1,X_1)^+$, with $g\in G(\mathbb{A}_f^p)$, $h\in G(\INTP)^+$ and $x\in \ES_{K_{1,p, \overline{\kappa}}}(G_1,X_1)^+$, its $p$-divisible group is given by $\A_x[p^\infty]$. So, to prove the claim, it suffices to show that for any $[g',h']\in \mathscr{A}(G_{1,\INTP})^\circ$ with $g'\in G_1(\mathbb{Z}_{(p)})_+^-$, $h'\in G_1^{\adj}(\INTP)^+$, the $p$-divisible group attached to $([g,h],x)\cdot(g',h')$ is isomorphic to $\A_x[p^\infty]$ respecting additional structure. But this follows from Corollary \ref{iso of kisin twist}. By the same argument, we see that the induced Newton stratification on  $\mathscr{A}(G_{1,\INTP})\times \ES_{K_{1,p},\overline{\kappa}}(G_1,X_1)^+$ descends to the Newton stratification on $\ES_{K_{1,p},0}(G_1,X_1)$.
\end{proof}

The induced Newton stratification on $\mathscr{A}(G_{\INTP})\times \ES_{K_{1,p},\overline{\kappa}}(G_1,X_1)^+$ descends to a stratification on $\ES_{K_p,0}(G,X)$, and we will call it the Newton stratification. By construction, one sees easily that this does not depend on the choice of $(G_1, X_1)$. More formally, we have the following formulas for $(G_1, X_1)$:
\[\begin{split}
\ES_{K_{1,p},0}(G_1,X_1)&=\coprod_{[b]\in B(G_1,\mu_1)}\ES_{K_{1,p},0}(G_1,X_1)^b, \\ \ES_{K_{1,p},\overline{\kappa}}(G_1,X_1)^+&=\coprod_{[b]\in B(G_1,\mu_1)}\ES_{K_{1,p},\overline{\kappa}}(G_1,X_1)^{+,b}, \\ \ES_{K_{1,p},\overline{\kappa}}(G_1,X_1)^b&=[\mathscr{A}(G_{1,\INTP})\times \ES_{K_{1,p},\overline{\kappa}}(G_1,X_1)^{+,b}]/\mathscr{A}(G_{1,\INTP})^\circ,
\end{split}\]
and for  $(G, X)$:
\[\begin{split}\ES_{K_{,p},0}(G,X)&=\coprod_{[b]\in B(G,\mu)}\ES_{K_{p},0}(G,X)^b, \\ \ES_{K_{p},\overline{\kappa}}(G,X)^+&=\coprod_{[b]\in B(G,\mu)}\ES_{K_{p},\overline{\kappa}}(G,X)^{+,b}, \\ \ES_{K_{p},\overline{\kappa}}(G,X)^b&=[\mathscr{A}(G_{\INTP})\times \ES_{K_{p},\overline{\kappa}}(G,X)^{+,b}]/\mathscr{A}(G_{\INTP})^\circ.
\end{split}\]
Moreover, we have
\[\begin{split}
\ES_{K_{p},\overline{\kappa}}(G,X)^{+,b}&=\ES_{K_{1,p},\overline{\kappa}}(G_1,X_1)^{+,b}/\Delta, \\
\ES_{K_{p},\overline{\kappa}}(G,X)^b&=[\mathscr{A}(G_{\INTP})\times \ES_{K_{1,p},\overline{\kappa}}(G_1,X_1)^{+,b}]/\mathscr{A}(G_{1,\INTP})^\circ. \end{split}\]

The proposition also indicates how to relate Newton strata to the group theoretic object $B(G,\mu)$. For $x\in \ES_{K_p,0}(G,X)(\overline{\kappa})$, we can find $x_0\in \ES_{K_p,0}(G,X)^+(\overline{\kappa})$ which is in the same Newton stratum as $x$. Noting that $x_0$ lifts to $\widetilde{x_0}\in \ES_{K_{1,p},0}(G_1,X_1)^+(\overline{\kappa})$ whose image in $B(G_1,\mu_1)\simeq B(G,\mu)$ depends only on $x$, we get a well defined map \[\ES_{K_p,0}(G,X)(\overline{\kappa})\rightarrow B(G,\mu)\] whose fibers are Newton strata of $\ES_{K_p,0}(G,X)$. 

\subsubsection[]{}\label{diagram for NP abe type}Now we are ready to think about general Shimura varieties of abelian type. Let $(G,X)$ be a Shimura datum of abelian type (\emph{not adjoint in general}) with good reduction at $p$. Let $(G^\adj,X^\adj)$ be its adjoint Shimura datum, and $(G_1,X_1)$ be a Shimura datum of Hodge type satisfying the two conditions in Lemma \ref{Kisin's lemma} with respect to $(G^\adj,X^\adj)$. 

By the previous discussions, we have a commutative diagram
\[\xymatrix{&\ES_{K_{1,p},0}(G_1,X_1)(\overline{\kappa})\ar[r]\ar[d]&B(G_1,\mu_1)\ar[d]_{\simeq}\\
\ES_{K_p,0}(G,X)(\overline{\kappa})\ar[r]&\ES_{K^{\adj}_p,0}(G^\adj,X^\adj)(\overline{\kappa})\ar[r]&B(G^\adj,\mu)&B(G,\mu)\ar[l]_(0.4){\simeq}.}\]
Here for $\mu$ (resp. $\mu_1$), we use the same notation when viewing it as a cocharacter of $G^\adj$, and we identified $B(G^\adj,\mu)$ and $B(G^\adj,\mu_1)$ silently.
Now we can imitate the main results in Hodge type cases.
Before stating the results, we fix notations as follows. Choose a sufficiently small open compact subgroup $K^p\subset G(\mathbb{A}_f^p)$. We simply denote by $\ES_0$ the special fiber of $\ES_K(G,X)$, and by $\delta_{K^p}$ the induced Newton map $\ES_0(\overline{\kappa})\rightarrow B(G,\mu)$. For $[b]\in B(G,\mu)$, we will write $\ES_0^b$ for the Newton stratum corresponding to it.
\begin{theorem}\label{Newton for abelian type}
The Newton stratum $\ES_0^b$ is non-empty, and it is an equi-dimensional locally closed subscheme of $\ES_0$ of dimension \[\langle\rho,\mu+\nu_G(b)\rangle-\frac{1}{2}\mathrm{def}_G(b).\] Here $\rho$ is the half-sum of positive roots of $G$. Moreover, $\overline{\ES_0^b}$, the closure of $\ES_0^b$, is the union of strata $\ES_0^{b'}$ with $[b']\leq [b]$, and $\overline{\ES_0^b}-\ES_0^b$ is either empty or pure of codimension 1 in $\overline{\ES_0^b}$.
\end{theorem}
\begin{proof}
For $\ES_0(G^\adj,X^\adj)$, the statements for $\ES_0(G^\adj,X^\adj)^b$ follow by combining Theorem \ref{Newton for Hodge type} with Proposition \ref{adjoint Newton-result}. On geometrically connected components, the morphism \[\ES_{\overline{\kappa}}(G,X)^+\rightarrow\ES_{\overline{\kappa}}(G^\adj,X^\adj)^+\] is a finite \'{e}tale cover, and hence the statements for $\ES_0^b$ hold.
\end{proof}

Thus for a Shimura datum $(G, X)$ of abelian type with good reduction at $p$, we have the Newton stratification on the special fiber $\ES_0$ of $\ES_K(G, X)$
\[\ES_0=\coprod_{[b]\in B(G,\mu)}\ES_0^b,\quad  \overline{\ES_0^b}=\coprod_{[b']\leq [b]}\ES_0^{b'}. \]
As in Remark \ref{R:ordinary-basic}, there is a unique minimal (closed) stratum $\ES_0^{b_0}$, the basic locus, associated to the minimal element $[b_0]\in B(G,\mu)$; there is also a unique maximal (open) stratum $\ES_0^{b_{\mu}}$, the $\mu$-ordinary locus, associated to the maximal element $[b_\mu]\in B(G,\mu)$.

\begin{remark}
Historically to study the geometry of Newton strata, one usually first proves that there exists some kind of almost product structure by introducing certain Igusa varieties over central leaves (cf. section \ref{S:leaves}) and the related Rapoport-Zink spaces, and then study the geometry of the associated Igusa varieties and Rapoport-Zink spaces respectively. This was done in the PEL type case in \cite{coho of PEL, geo of newt PEL} and in the Hodge type case in \cite{Foliation-Hamacher, level m}. In the abelian type case, we could also do this, using the Rapoport-Zink spaces constructed in \cite{Sh1}. However, we will not pursue this aspect here. 
\end{remark}

\begin{example}\label{NP quaternion SHV}
Notations as in Example \ref{quaternion SHV}, we assume that $D$ is split at $p$ and that $F$ is unramified at $p$. Let $\mathfrak{p}_1,\cdots, \mathfrak{p}_t$ be places of $F$ over $p$, and $F_{\mathfrak{p}_i}$ be the $p$-adic completion of $F$. We will fix an identification $$\iota: \mathrm{Hom}(F,\mathbb{R})\cong \mathrm{Hom}(F,\overline{\mathbb{Q}}_p)\cong\coprod_i\mathrm{Hom}(F_{\mathfrak{p}_i},\overline{\mathbb{Q}}_p).$$ After reordering the $\mathfrak{p}_i$, we can find $1\leq s\leq t$, such that for $i\leq s$, $\mathrm{Hom}(F_{\mathfrak{p}_i},\overline{\mathbb{Q}}_p)$ contains some $\infty_j$ with $j\leq d$; and for $i> s$, $\mathrm{Hom}(F_{\mathfrak{p}_i},\overline{\mathbb{Q}}_p)$ contains only $\infty_j$ with $j> d$.

Then \[G_{\mathbb{Q}_p}\cong \prod_{i=1}^t\mathrm{Res}_{F_{\mathfrak{p}_i}/\mathbb{Q}_p}\GL_{2,F_{\mathfrak{p}_i}}:=\prod_{i=1}^tG_{\mathfrak{p}_i}.\] The Shimura datum gives a cocharacter $\mu:\mathbb{G}_m\rightarrow G_{\overline{\mathbb{Q}}_p}$ as in \ref{notation NP 1}. Under the isomorphism $$G_{\overline{\mathbb{Q}}_p}\cong \prod_{i=1}^t\Big(\prod_{\sigma:F_{\mathfrak{p}_i}\hookrightarrow{\overline{\mathbb{Q}}_p}}\GL_{2,\overline{\mathbb{Q}}_p}\Big),$$ the cocharacter $\mu$ decomposes into \[\mu_i:\mathbb{G}_m\rightarrow G_{\mathfrak{p}_i\overline{\Q}_p}=\prod_{\sigma:F_{\mathfrak{p}_i}\hookrightarrow{\overline{\mathbb{Q}}_p}}\GL_{2,\overline{\mathbb{Q}}_p}.\] As $\infty_j,\,1\leq j\leq d$, are all the archimedean places where $D$ splits, by our choice of ordering of the primes $\mathfrak{p}_i$,  $\mu_i$ is trivial for $i>s$, and for $1\leq i\leq s$ it is of the form (reordering the  $\sigma:F_{\mathfrak{p}_i}\hookrightarrow{\overline{\mathbb{Q}}_p}$ if necessary)
\[ z\mapsto \left(\left(\begin{array}{cc}
z&0\\
0&1
\end{array}\right),\cdots,\left(\begin{array}{cc}
z&0\\
0&1
\end{array}\right),\left(\begin{array}{cc}
1&0\\
0&1
\end{array}\right),\cdots,\left(\begin{array}{cc}
1&0\\
0&1
\end{array}\right) \right).\] 
For $1\leq i\leq s$, we will write $a_i$ for the number of non-trivial factors of $\mu_i$. Then \[B(G,\mu)\cong \prod_{i=1}^sB(G_{\mathfrak{p}_i},\mu_i)=\prod_{i=1}^s B(\mathrm{Res}_{F_{\mathfrak{p}_i}/\mathbb{Q}_p}\GL_{2,F_{\mathfrak{p}_i}},\mu_i),\] and we can use \cite{corr local L and coho R-Z} 2.1 to compute $B(G_{\mathfrak{p}_i},\mu_i)$.

Let $n_i=[F_{\mathfrak{p}_i}:\mathbb{Q}_p]$.  Let $B(G_{\mathfrak{p}_i},\mu_i)_2$ (resp. $B(G_{\mathfrak{p}_i},\mu_i)_1$) be the subset of $B(G_{\mathfrak{p}_i},\mu_i)$ with 2 slopes (resp. 1 slope). Then $B(G_{\mathfrak{p}_i},\mu_i)_2$ is the set of pairs \[(\frac{d_1}{n_i},\frac{d_2}{n_i})\] such that $d_1, d_2$ are non-negative integers with $d_1>d_2$ and $d_1+d_2=a_i$, and $B(G_{\mathfrak{p}_i},\mu_i)_1$ contains only one element which is the pair \[(\frac{a_i}{2n_i},\frac{a_i}{2n_i}).\] It is then easy to see that the cardinality of $B(G_{\mathfrak{p}_i},\mu_i)$ is $\lceil\frac{a_i}{2}\rceil+1$, where as usual for a real number $x$, $\lceil x\rceil$ is the smallest integer which is no less than $x$.
The cardinality of $B(G,\mu)$ is the product of those of $B(G_{\mathfrak{p}_i},\mu_i)$.

One sees easily that for each $i$, $B(G_{\mathfrak{p}_i},\mu_i)$ is totally ordered. For $[b]\in B(G,\mu)$, its projection to $B(G_{\mathfrak{p}_i},\mu_i)$ is of form $(\frac{\lambda_1}{n_i},\frac{\lambda_2}{n_i})$ with $\lambda_1\geq \lambda_2$ and $\lambda_1+ \lambda_2=a_i$. These $\lambda_i$ are integers unless $\lambda_1=\lambda_2$ and $a_i$ odd. Let \[l_i(b):=[\lambda_1],\] where $[x]$ is the integer part of $x$. By Theorem \ref{Newton for abelian type}, $\ES_0^b$ is non-empty and equi-dimensional. One deduces easily from purity that it is of dimension $\sum_{i=1}^sl_i(b)$.
\end{example}

\section[Ekedahl-Oort stratifications]{Ekedahl-Oort stratifications}

We study the Ekedahl-Oort stratifications on the special fibers of the Shimura varieties introduced in the first section.

\subsection[$F$-zips]{$F$-zips and $G$-zips}
In this subsection, we will follow \cite{disinv} and \cite{zipaddi}
to introduce $F$-zips and $G$-zips. They should be viewed as a kind of de Rham realizations of certain abelian motives. They are introduced by Moonen-Wedhorn and Pink-Wedhorn-Ziegler with the aim to study Ekedahl-Oort strata for Shimura varieties.

Let $S$ be a scheme, and $M$
be a locally free $O_S$-module of finite rank. By a descending
(resp. ascending) filtration $C^\bullet$ (resp. $D_\bullet$) on
$M$, we always mean a separating and exhaustive filtration such that
$C^{i+1}(M)$ is a locally direct summand of $C^i(M)$ (resp.
$D_i(M)$ is a locally direct summand of $D_{i+1}(M)$).

Let $\text{\texttt{LF}}(S)$ be the category of locally free
$O_S$-modules of finite rank, $\text{\texttt{FilLF}}^\bullet(S)$
be the category of locally free $O_S$-modules of finite rank with
descending filtration. For two objects $(M,C^{\bullet}(M))$ and
$(N,C^{\bullet}(N))$ in $\text{\texttt{FilLF}}^\bullet(S)$, a
morphism \[f:(M,C^{\bullet}(M))\rightarrow (N,C^{\bullet}(N)) \] is a
homomorphism of $O_S$-modules such that $f(C^i(M))\subseteq
C^i(N)$. We also denote by $\text{\texttt{FilLF}}_\bullet(S)$ the
category of locally free $O_S$-modules of finite rank with
ascending filtration. For two objects $(M,C^\bullet)$ and
$(M',C'^\bullet)$ in $\text{\texttt{FilLF}}^\bullet(S)$, their
tensor product is defined to be $(M\otimes M', T^\bullet)$ with
$T^i=\sum_{j}C^j\otimes C'^{i-j}$. Similarly for
$\text{\texttt{FilLF}}_\bullet(S)$. For an object $(M,C^\bullet)$
in $\text{\texttt{FilLF}}^\bullet(S)$, one defines its dual to be
$$(M,C^\bullet)^\vee=({}^{\vee}M:=M^\vee,{}^{\vee}C^i:=(M/C^{1-i})^\vee);$$ and for an
object $(M,D_\bullet)$ in $\text{\texttt{FilLF}}_\bullet(S)$, one
defines its dual to be
$$(M,D_\bullet)^\vee=({}^{\vee}M:=M^\vee,{}^{\vee}D_i:=(M/D_{-1-i})^\vee).$$
It is clear from the convention that $(M,C^\bullet)^\vee=({}^{\vee}M,{}^{\vee}C^\bullet)=(M^\vee,{}^{\vee}C^\bullet)$, and similarly for  $D_\bullet$.

If $S$ is over $\mathbb{F}_p$, we will denote by $\sigma:
S\rightarrow S$ the morphism which is the identity on the
topological space and $p$-th power on the sheaf of functions. For
an $S$-scheme $T$, we will write $T^{(p)}$ for the pull back of
$T$ via $\sigma$. For a quasi-coherent $O_S$-module $M$, $M^{(p)}$
means the pull back of $M$ via $\sigma$. For a $\sigma$-linear map
$\varphi:M\rightarrow M$, we will denote by
$\varphi^\lin:M^{(p)}\rightarrow M$ its linearization.

\begin{definition}\label{defini F-zip}
Let $S$ be an $\mathbb{F}_p$-scheme. 
\begin{enumerate}
\item By an $F$-zip over $S$, we mean a tuple $\underline{M}=(M,\ C^{\bullet},\
D_{\bullet},\ \varphi_{\bullet})$ where 

\begin{itemize}
\item $M$ is an object in $\text{\texttt{LF}}(S)$, i.e. $M$ is a
locally free sheaf of finite rank on~$S$;

\item $(M,C^{\bullet})$ is an object in
$\text{\texttt{FilLF}}^\bullet(S)$, i.e. $C^\bullet$ is a
descending filtration on~$M$;

\item $(M,D_{\bullet})$ is an object in
$\text{\texttt{FilLF}}_\bullet(S)$, i.e. $D_\bullet$ is an
ascending filtration on~$M$;

\item $\varphi_i:C^i/C^{i+1}\rightarrow D_i/D_{i-1}$ is a
$\sigma$-linear map whose linearization
$$\varphi_i^{\mathrm{lin}}:(C^i/C^{i+1})^{(p)}\rightarrow
D_i/D_{i-1}$$ is an isomorphism.
\end{itemize}

\item By a morphism of $F$-zips $$\underline{M}=(M,C^\bullet, D_\bullet,
\varphi_\bullet)\rightarrow \underline{M'}=(M',C'^\bullet,
D'_\bullet, \varphi'_\bullet),$$ we mean a morphism of
$O_S$-modules $f: M\rightarrow N$, such that for all $i\in \mathbb{Z}$,
$f(C^i)\subseteq C'^i$, $f(D_i)\subseteq D'_i$, and $f$ induces a
commutative diagram
\[\begin{CD}
C^i/C^{i+1}@>\varphi_i >>D_i/D_{i-1}\\
@V f VV @VV f V\\
C'^i/C'^{i+1}@>\varphi'_i >>D'_i/D'_{i-1}.
\end{CD}\]
\end{enumerate}
\end{definition}
\begin{example}\label{Tateobj}(\cite{zipaddi} Example 6.6)
The Tate $F$-zip of weight $d$ is $$\mathbf{1}(d):=(O_S,
C^\bullet, D_\bullet, \varphi_\bullet),$$ where
$$C^i=\left\{
\begin{aligned}
         O_S&\text{ \ \ \ for }i\leq d;\\
         0&\text{ \ \ \ for }i> d;
                          \end{aligned} \right.\ \ \ \ \ \
D_i=\left\{
\begin{aligned}
         0&\text{ \ \ \ for }i< d;\\
         O_S&\text{ \ \ \ for }i\geq d;
                          \end{aligned} \right.$$
and $\varphi_d$ is the Frobenius.
\end{example}
One can talk about tensor products and duals in the category of $F$-zips.
\begin{definition}(\cite{zipaddi} Definition 6.4)
Let $\underline{M}$, $\underline{N}$ be two $F$-zips over $S$,
then their tensor product is the $F$-zip
$\underline{M}\otimes\underline{N}$, consisting of the tensor
product $M\otimes N$ with induced filtrations $C^\bullet$ and
$D_\bullet$ on $M\otimes N$, and induced $\sigma$-linear maps
$$\xymatrix{\gr_C^i(M\otimes N)\ar[d]^{\cong}& &\gr^D_i(M\otimes N)\\
\bigoplus_j\gr^j_C(M)\otimes
\gr^{i-j}_C(N)\ar[rr]^{\bigoplus_j\varphi_j\otimes \varphi_{i-j}}
&& \bigoplus_j\gr_j^D(M)\otimes \gr_{i-j}^D(N)\ar[u]^{\cong} }$$
whose linearization are isomorphisms.
\end{definition}
\begin{definition}(\cite{zipaddi} Definition 6.5)
The dual of an $F$-zip $\underline{M}$ is the $F$-zip
$\underline{M}^\vee$ consisting of the dual sheaf of $O_S$-modules
$M^\vee$ with the dual descending filtration of $C^\bullet$ and
dual ascending filtration of $D_\bullet$, and $\sigma$-linear maps
whose linearization are isomorphisms
$$\xymatrix@1{(\gr_C^i(M^\vee))^{(p)}=((\gr_C^{-i}M)^\vee)^{(p)}\ar[rr]^(0.6){\big((\varphi_{-i}^{\mathrm{lin}})\big)^{-1\vee}} &&
(\gr_{-i}^DM)^\vee\cong \gr^D_i(M^\vee)}.$$
\end{definition}

For the Tate $F$-zips introduced in Example \ref{Tateobj}, we have natural isomorphisms $\mathbf{1}(d)\otimes
\mathbf{1}(d')\cong\mathbf{1}(d+d')$ and
$\mathbf{1}(d)^\vee\cong\mathbf{1}(-d)$. The $d$-th Tate twist of
an $F$-zip $\underline{M}$ is defined as
$\underline{M}(d):=\underline{M}\otimes \mathbf{1}(d)$, and there
is a natural isomorphism $\underline{M}(0)\cong \underline{M}$.

\begin{definition}\label{abmisizip}
A morphism between two objects in $\text{\texttt{LF}}(S)$ is said
to be admissible if the image of the morphism is a locally direct
summand. A morphism $f:(M,C^\bullet)\rightarrow(M',C'^\bullet)$ in
$\text{\texttt{FilLF}}^\bullet(S)$ (resp.
$f:(M,D_\bullet)\rightarrow(M',D'_\bullet)$ in
$\text{\texttt{FilLF}}_\bullet(S)$) is called admissible if for
all $i$, $f(C^i)$ (resp. $f(D_i)$) is equal to $f(M)\cap C'^i$
(resp. $f(M)\cap D'_i$) and is a locally direct summand of $M'$. A
morphism between two $F$-zips $\underline{M}\rightarrow
\underline{M'}$ in $F\text{-\texttt{Zip}}(S)$ is called admissible
if it is admissible with respect to the two filtrations.
\end{definition}

With admissible morphisms, tensor products, duals and the Tate object
$\mathbf{1}(0)$ as above, $F\text{-\texttt{Zip}}(S)$ becomes an
$\mathbb{F}_p$-linear exact rigid tensor category (see
\cite{zipaddi} section 6). By \cite{zipaddi} Lemma
4.2 and Lemma 6.8, for a morphism in $F\text{-\texttt{Zip}}(S)$, the property of being admissible is local for the fpqc topology.

We will introduce $G$-zips following \cite{zipaddi}. These may be viewed as $F$-zips with $G$-structure. Note that the authors of \cite{zipaddi} work with reductive groups
over a general finite field $\mathbb{F}_q$ containing
$\mathbb{F}_p$, and $q$-Frobenius. But we don't need the most
general version of $G$-zips, as our reductive groups are connected
and defined over $\mathbb{F}_p$.

\subsubsection[]{}\label{G-zip settings}Let $G$ be a connected reductive group over $\mathbb{F}_p$, and $\chi$ be a cocharacter of $G$ defined over $\kappa$, a finite extension of $\mathbb{F}_p$. Let $P_+\subseteq G_{\kappa}$ (resp. $L\subseteq G_{\kappa},\ P_-\subseteq G_{\kappa}$) be the subgroup whose Lie algebra is the submodule of $\mathrm{Lie}(G_{\kappa})$ of non-negative weights (resp. of weight 0, of non-positive weights) with respect to $\chi$ composed with the adjoint action of $G_{\kappa}$ on $\mathrm{Lie}(G_{\kappa})$. The unipotent subgroup of $P_+$ (resp. $P_-$) will be denoted by $U_+$ (resp. $U_-$).

\begin{definition}\label{G-ziptypechi}
Let $S$ be a scheme over $\kappa$. 
\begin{enumerate}
\item A $G$-zip of type
$\chi$ over $S$ is a tuple $\underline{I}=(I, I_+, I_-, \iota)$
consisting of 
\begin{itemize}
\item a right $G_k$-torsor $I$ over~$S$, 
\item a right
$P_+$-torsor $I_+\subseteq I$ (i.e. the inclusion $I_+\subseteq I$ is such that it is compatible for the $P_+$-action on $I_+$ and the $G_\kappa$-action on $I$), 
\item a right $P_-^{(p)}$-torsor
$I_-\subseteq I$ (similarly as for $I_+\subseteq I$), and 
\item an isomorphism of $L^{(p)}$-torsors
$\iota:I^{(p)}_+/U^{(p)}_+\rightarrow I_-/U^{(p)}_-$.
\end{itemize}

\item A morphism $(I, I_+, I_-, \iota)\rightarrow (I', I'_+, I'_-,
\iota')$ of $G$-zips of type $\chi$ over $S$ consists of
equivariant morphisms $I\rightarrow I'$ and $I_{\pm}\rightarrow
I'_{\pm}$ that are compatible with inclusions and the isomorphisms
$\iota$ and~$\iota'$.
\end{enumerate}
\end{definition}

Here by a torsor over $S$ of an fpqc group scheme $G/S$, we mean an fpqc scheme $X/S$ with a $G$-action $\rho:X\times_S G\rightarrow X$ such that the morphism $X\times G\rightarrow X\times_S X$, $(x,g)\rightarrow (x,x\cdot g)$ is an isomorphism.

The category of $G$-zips of type $\chi$ over $S$ will be denoted
by $G\texttt{-Zip}_\kappa^{\chi}(S)$. When $G=\GL_n$ we recover the category of $F$-zips, cf. \cite{zipaddi} subsection 8.1.
With the evident notation of pull back, the $G\texttt{-Zip}_\kappa^{\chi}(S)$ form a fibered category
over the category of schemes over $\kappa$, denoted by $G\texttt{-Zip}_\kappa^{\chi}$. Noting that morphisms in $G\texttt{-Zip}_\kappa^{\chi}(S)$ are isomorphisms, $G\texttt{-Zip}_\kappa^{\chi}$ is a category fibered in groupoids.
\begin{theorem}\label{mainthofGzip}(\cite{zipaddi} Corollary 3.12)
The fibered category $G\texttt{-}\mathtt{Zip}_\kappa^{\chi}$ is a
smooth algebraic stack of dimension 0 over $\kappa$.
\end{theorem}

\subsubsection[]{Some technical constructions about $G$-zips}\label{subsubsection group G-zips}

We need more information about the structure of
$G\texttt{-Zip}_\kappa^{\chi}$. First, we need to introduce some
standard $G$-zips as in~\cite{zipaddi}.
\begin{construction}\label{constG-zip}(\cite{zipaddi} Construction 3.4)
Let $S/\kappa$ be a scheme. For a section $g\in
G(S)$, one associates a $G$-zip of type $\chi$ over $S$ as
follows. Let $I_g=S\times_\kappa G_\kappa$ and $I_{g,+}=S\times_\kappa P_+\subseteq
I_g$ be the trivial torsors. Then $I_g^{(p)}\cong
S\times_\kappa G_\kappa=I_g$ canonically, and we define $I_{g,-}\subseteq
I_g$ as the image of $S\times_\kappa P_-^{(p)}\subseteq S\times_\kappa G_\kappa$
under left multiplication by $g$. Then left multiplication by $g$
induces an isomorphism of $L^{(p)}$-torsors
$$\iota_g:I_{g,+}^{(p)}/U_+^{(p)}=S\times_\kappa P_+^{(p)}/U_+^{(p)}\cong S\times_\kappa P_-^{(p)}/U_-^{(p)}
\stackrel{\sim}{\rightarrow}g(S\times_\kappa P_-^{(p)})/U_-^{(p)}=I_{g,-}/U_-^{(p)}.$$ We thus
obtain a $G$-zip of type $\chi$ over $S$, denoted by
$\underline{I}_g$.
\end{construction}
\begin{lemma}(\cite{zipaddi} Lemma 3.5)
Any $G$-zip of type $\chi$ over $S$ is \'{e}tale locally of the
form $\underline{I}_g$.
\end{lemma}

Now we will explain how to write $G\texttt{-Zip}_\kappa^{\chi}$ in
terms of quotient of an algebraic variety by the action of a
linear algebraic group following \cite{zipaddi} Section~3.

Denote by $\mathrm{Frob}_p:L\rightarrow L^{(p)}$ the relative
Frobenius of $L$, and by $E_{G,\chi}$ the fiber product
$$\xymatrix@C=0.7cm{E_{G,\chi}\ar[d]\ar[rr]&&P_-^{(p)}\ar[d]\\
P_+\ar[r]& L\ar[r]^{\mathrm{Frob}_p}&L^{(p)}.}$$ Then we have \begin{subeqn}\label{eqnE-act}
E_{G,\chi}(S)=\{(p_+:=lu_+,\ p_-:=l^{(p)}u_-):l\in L(S), u_+\in
U_+(S), u_-\in U_-^{(p)}(S)\}.
\end{subeqn}
It acts on $G_\kappa$ from the left hand side as follows. For $(p_+,p_-)\in E_{G,\chi}(S)$ and $g\in G_\kappa(S)$, set
$(p_+,p_-)\cdot g:=p_+gp_-^{-1}.$

To relate $G\texttt{-Zip}_\kappa ^{\chi}$ to the quotient stack
$[E_{G,\chi}\backslash G_\kappa ]$, we need the following constructions
in \cite{zipaddi}. First, for any two sections $g,g'\in G_\kappa (S)$,
there is a natural bijection between the set
$$\text{Transp}_{E_{G,\chi}(S)}(g,g'):=\{(p_+,p_-)\in E_{G,\chi}(S)\mid p_+gp_-^{-1}=g'\}$$
and the set of morphisms of $G$-zips $\underline{I}_g\rightarrow
\underline{I}_{g'}$ (see \cite{zipaddi} Lemma 3.10). So we define
a category $\mathcal {X}$ fibered in groupoids over the category
of $\kappa$-schemes as
follows. For any scheme $S/\kappa$, let $\mathcal {X}(S)$ be the small
category whose underly set is $G(S)$, and for any two elements
$g,g'\in G(S)$, the set of morphisms is the set
$\text{Transp}_{E_{G,\chi}(S)}(g,g')$.
\begin{theorem}\label{mainPWZ}(\cite{zipaddi} Proposition 3.11)
Sending $g\in \mathcal
{X}(S)=G(S)$ to $\underline{I}_{g}$ induces a fully faithful morphism $\mathcal {X}\rightarrow
G\text{-}\mathtt{Zip}_\kappa^{\chi}$. Moreover, it induces an isomorphism of algebraic stacks
$[E_{G,\chi}\backslash G_\kappa]\stackrel{\sim}{\rightarrow}
G\text{-}\mathtt{Zip}_\kappa^{\chi}$.
\end{theorem}

\subsection[Group theoretic preparations]{Some group theoretic descriptions for the geometry of $[E_{G,\mu}\backslash G_\kappa]$}\label{subsection group EO}

Let $B\subseteq G$ be a Borel subgroup, and $T\subseteq B$ be a
maximal torus. Note that such a $B$ exists by \cite{gpfinifield}
Theorem 2, and such a $T$ exists by \cite{SGA3} XIV Theorem 1.1. Let $W(B,T):=\text{Norm}_G(T)(\overline{\kappa})/T(\overline{\kappa})$ be the
Weyl group, and $I(B,T)$ be the set of simple reflections defined
by $B_{\overline{\kappa}}$. Let $\varphi$ be the Frobenius on $G$ given by the
$p$-th power. It induces an isomorphism \[(W(B,T),I(B,T))\st{\sim}{\rightarrow} (W(B,T),I(B,T))\]
of Coxeter systems still denoted by $\varphi$.

A priori the pair $(W(B,T),I(B,T))$ depends on the pair $(B,T)$. However, any other pair $(B', T')$ with $B'\subseteq G_{\overline{\kappa}}$ a Borel subgroup and $T'\subseteq B'$ a maximal torus is obtained on conjugating $(B_{\overline{\kappa}},T_{\overline{\kappa}})$ by some $g\in G(\overline{\kappa})$ which is unique up to right multiplication by $T_{\overline{\kappa}}$. So conjugation by $g$ induces isomorphisms $W(B,T)\rightarrow W(B',T')$ and $I(B,T)\rightarrow I(B',T')$ that are independent of $g$. Moreover, the morphisms attached to any three of such pairs are compatible, so we will simply write $(W,I)$ for $(W(B,T),I(B,T))$, and view it as `the' Weyl group with `the' set of simple reflections.

The cocharacter
$\mu:\mathbb{G}_m\rightarrow G_\kappa$ as in 3.1 gives a parabolic subgroup
$P_+$, and hence a subset $J\subseteq I$ by taking simple roots whose
inverse are roots of $P_+$. Let $W_J$ the subgroup of $W$ generated by $J$, and ${}^JW$ be the
set of elements $w$ such that $w$ is the element of minimal length in
some coset $W_Jw'$. Note that there is a unique element in $W_Jw'$
of minimal length, and each $w\in W$ can be uniquely written as
$w=w_J{}^Jw$ with $w_J\in W_J \text{ and } {}^Jw\in {}^JW$. In
particular, ${}^JW$ is a system of representatives of $W_J\backslash W$.

Furthermore, if $K$ is a second subset of $I$, then for each $w$,
there is a unique element in $W_JwW_K$ which is of minimal length.
We will denote by ${}^JW^K$ the set of elements of minimal length,
and it is a set of representatives of $W_J\backslash W/W_K$.

Let $w_0$ be the element of maximal length in $W$, set
$K:={}^{w_0}\!\varphi(J)$. Here we write ${}^g\!J$ for
$gJg^{-1}$. Let \[x\in {}^K\!W^{\varphi(J)}\] be the element of
minimal length in $W_Kw_0W_{\varphi(J)}$. Then $x$ is the
unique element of maximal length in ${}^K\!W^{\varphi(J)}$ (see
\cite{VW} 5.2). There is a partial order $\preceq$ on ${}^JW$,
defined by $w'\preceq w$ if and only if there exists $y\in W_J$, such that
\[yw'x\varphi(y^{-1})x^{-1}\leq w\] (see \cite{VW} Definition 5.8). Here $\leq$ is the Bruhat order (see A.2 of
\cite{VW} for the definition). The partial order $\preceq$ makes
${}^JW$ into a topological space.

Now we can state the the main result in \cite{zipdata} of Pink-Wedhorn-Ziegler that
gives a combinatorial description of the topological space of
$[E_{G,\mu}\backslash G_\kappa]$ (and hence
$G\texttt{-}\mathtt{Zip}_\kappa^{\mu}$).
\begin{theorem}\label{collectzipdata}
	
For $w\in {}^{J}W$, and $T'\subseteq B'\subseteq
G_{\overline{\kappa}}$ with $T'$ (resp. $B'$) a maximal torus
(resp. Borel) of $G_{\overline{\kappa}}$ such that $T'\subseteq
L_{\overline{\kappa}}$ and $B'\subseteq P_{-,\overline{\kappa}}^{(p)}$, let $g,\dot{w}\in \mathrm{Norm}_{G_{\overline{\kappa}}}(T')$ be a
representative of $\varphi^{-1}(x)$ and $w$ respectively, and
$G^w\subseteq G_{\overline{\kappa}}$ be the $E_{G,\mu}$-orbit of
$gB'\dot{w}B'$. Then
\begin{enumerate}

\item The orbit $G^w$ does not depends on the choices of $\dot{w}$,
$T'$, $B'$ or $g$.

\item The orbit $G^w$ is a locally closed smooth subvariety of
$G_{\overline{\kappa}}$. Its dimension is $\mathrm{dim}(P)+l(w)$.
Moreover, $G^w$ consists of only one $E_{G,\mu}$-orbit. So $G^w$
is actually the orbit of $g\dot{w}$.

\item Denote by $\big|[E_{G,\mu}\backslash G_{\kappa}]\otimes
\overline{\kappa}\big|$ the topological space of
$[E_{G,\mu}\backslash G_{\kappa}]\otimes \overline{\kappa}$, and
still write ${}^JW$ for the topological space induced by the
partial order $\preceq$. Then the association $w\mapsto G^w$
induces a homeomorphism ${}^JW\stackrel{\sim}{\rightarrow}
\big|[E_{G,\mu}\backslash G_{\kappa}]\otimes
\overline{\kappa}\big|$.
\end{enumerate}
\end{theorem}

\begin{remark}\label{maximal and minimal EO}
There is a unique maximal (resp. minimal) element in ${}^{J}W$ (with respect to $\preceq$). For a variety $X/\overline{\kappa}$ with a map $X(\overline{\kappa})\rightarrow {}^{J}W$, the preimage of this element is called the ordinary locus (resp. superspecial locus).
\end{remark}

\subsection[Ekedahl-Oort stratifications on Shimura varieties of Hodge type]{Ekedahl-Oort stratifications on Shimura varieties of Hodge type}

Now we will explain how to construct Ekedahl-Oort stratification following \cite{EOZ}. Notations as in \ref{model Hodge type}, we will write $\V$, $s$ and $s_{\dr}$ respectively for its reduction mod $p$, and $L$ (resp. $G$, $\ES_0$) for the special fiber of $V_{\INT_{(p)}}$ (resp. $\inttG$, $\ES_K(G,X)$). By \cite{EOZ} Lemma 2.3.2 1), the scheme $I=\Isom_{\ES_0}\big((L^\vee,s)\otimes O_{\ES_0},(\mathcal {V}, \sdr)\big)$ is a right $G$-torsor.

\begin{set}\label{F-zip attached to abs}
Let $F:\V^{(p)}\rightarrow \V$ and $V:\V\rightarrow \V^{(p)}$ be the Frobenius and Verschiebung on $\V$ respectively. Let $\delta:\V\rightarrow \V^{(p)}$ be the semi-linear map sending $v$ to $v\otimes 1$. Then we have a semi-linear map $F\circ \delta:\V\rightarrow \V$. There is a descending filtration \[\V\supseteq \mathrm{ker}(F\circ \delta)\supseteq 0\] and an ascending filtration \[0\subseteq \mathrm{im}(F)\subseteq \V.\] The morphism $V$ induces an isomorphism \[\V/\mathrm{im}(F)\rightarrow \mathrm{ker}(F)\] whose inverse will be denoted by $V^{-1}$. Then $F$ and $V^{-1}$ induce isomorphisms $$\varphi_0:(\V/\mathrm{ker}(F\circ \delta))^{(p)}\rightarrow \mathrm{im}(F)$$
and $$\varphi_1:(\mathrm{ker}(F\circ \delta))^{(p)}\rightarrow \V/(\mathrm{im}(F)).$$
\end{set}

\begin{set}\label{phi in standard G-zip}
Let $\mu$ be as in \ref{notation NP 1}, we use the same symbol for its reduction mod $p$. The cocharacter $\mu:\mathbb{G}_{m,\kappa}\rightarrow G_{\kappa}\subseteq \mathrm{GL}(L_{\kappa})\cong\mathrm{GL}(L^\vee_{\kappa})$ induces an $F$-zip structure on $L^\vee_{\kappa}$ as follows. Let $(L^\vee_{\kappa})^0$ (resp. $(L^\vee_{\kappa})^1$) be the subspace of $L^\vee_{\kappa}$ of weight $0$ (resp. $1$) with respect to $\mu$, and $(L^\vee_{\kappa})_0$ (resp. $(L^\vee_{\kappa})_1$) be the subspace of $L^\vee_{\kappa}$ of weight $0$ (resp. $1$) with respect to $\mu^{(p)}$. Then we have a descending filtration
\[L^\vee_{\kappa}\supseteq (L^\vee_{\kappa})^1\supseteq 0\] and an ascending filtration \[0\subseteq (L^\vee_{\kappa})_0\subseteq L^\vee_{\kappa}.\] Let $\xi:L^\vee_{\kappa}\rightarrow (L^\vee_{\kappa})^{(p)}$ be the isomorphism given by $l\otimes k\mapsto l\otimes 1\otimes k$, $\forall\ l\in L^\vee$, $\forall\ k\in \kappa$. Then $\xi$ induces isomorphisms

$$\phi_0:(L^\vee_{\kappa})^{(p)}/((L^\vee_{\kappa})^1)^{(p)}\stackrel{\mathrm{pr}_2}{\longrightarrow}((L^\vee_{\kappa})^0)^{(p)} \stackrel{\xi^{-1}}{\longrightarrow}(L^\vee_{\kappa})_0$$ and
$$\phi_1:((L^\vee_{\kappa})^1)^{(p)}\stackrel{\xi^{-1}}{\longrightarrow}((L^\vee_{\kappa})_1\simeq L^\vee_{\kappa}/(L^\vee_{\kappa})_0.$$
\end{set}

The first main result of \cite{EOZ} is as follows.

\begin{theorem}\label{G-zipES_0}(\cite{EOZ} Theorem 2.4.1)

\begin{enumerate}

\item Let $I_+\subseteq I$ be the closed subscheme
$$I_+:=\Isom_{\ES_0}\big((L_{\kappa}^\vee\supseteq(L^\vee_{\kappa})^1, s)\otimes O_{\ES_0},\ (\V\supseteq\mathrm{ker}(F\circ \delta),
\sdr)\big).$$Then $I_+$ is a $P_+$-torsor over $\ES_0$.

\item Let $I_-\subseteq I$ be the closed subscheme
$$I_-:=\Isom_{\ES_0}\big(((L^\vee_{\kappa})_0\subseteq L_{\kappa}^\vee, s)\otimes O_{\ES_0},\ (\mathrm{im}(F)\subseteq\V,
\sdr)\big).$$Then $I_-$ is a $P_-^{(p)}$-torsor over
$\ES_0$.

\item Let $\iota:I_+^{(p)}/U_+^{(p)}\rightarrow I_-/U_-^{(p)}$ be the morphism
induced by
\begin{equation*}
\begin{split}
 I_+^{(p)}&\rightarrow I_-/U_-^{(p)}\\
f&\mapsto (\varphi_0\oplus \varphi_1)\circ\mathrm{gr}(f)\circ(\phi_0^{-1}\oplus \phi_1^{-1}), \forall\ S/\ES_0\text{ and }\forall\ f\in I_+^{(p)}(S).
 \end{split}
 \end{equation*}
Then $\iota$ is an isomorphism of $L^{(p)}$-torsors.
\end{enumerate}
Hence the tuple $(I,I_+,I_-,\iota)$ is a $G$-zip of type $\mu$ over $\ES_0$.
\end{theorem}

The $G$-zip $(I,I_+,I_-,\iota)$ induces a morphism $\ES_0\rightarrow G\text{-}\mathtt{Zip}_\kappa^{\mu}\simeq[E_{G,\mu}\backslash G_\kappa]$ over $\kappa$. In the following we will simply write $\ES_{\overline{\kappa}}=\ES_{0,\overline{\kappa}}=\ES_{K,\overline{\kappa}}(G,X)$.
We will write the induced morphism over $\overline{\kappa}$ as
\[\zeta:\ES_{\overline{\kappa}}\rightarrow G\text{-}\mathtt{Zip}_\kappa^{\mu}\otimes\overline{\kappa}\simeq[E_{G,\mu}\backslash G_\kappa]\otimes\overline{\kappa},\] whose fibers are called Ekedahl-Oort strata of $\ES_{\overline{\kappa}}$. In the following we will sometimes abbreviate ``Ekedahl-Oort'' to ``E-O'' for short. The main results about the Ekedahl-Oort stratifications are as follows.
\begin{theorem}\label{th EO Hodge type}

\begin{enumerate}

\item The morphism $\zeta$ above is smooth, and it is surjective when $p>2$. In particular,

\begin{enumerate}
	
\item each E-O stratum is a smooth and locally closed subscheme of $\ES_{\overline{\kappa}}$, the closure of an E-O stratum is a union of strata;

\item all the strata are in bijection with a subset of ${}^JW$, and for $w\in {}^JW$, the corresponding stratum $\ES_0^w$ is, if non-empty, of dimension $l(w)$, the length of $w$. Moreover, all the  $\ES_0^w$ are non-empty when $p>2$.
\end{enumerate}
\item Each E-O stratum is quasi-affine.
\end{enumerate}
\end{theorem}

\begin{proof}
For (1), all the statements but non-emptiness follows from \cite{EOZ} Theorem 4.1.2 and Proposition 4.1.4: $p>2$ was assumed there, but by \cite{2-adic int} section 3, the arguments in \cite{EOZ} also work when $p=2$. To see the non-emptiness of E-O strata when $p>2$, by Theorem \ref{cent for Hodge}, each central leaf in the basic locus is non-empty, and by the proof of \cite{VW} Proposition 9.17, the minimal E-O stratum is a central leaf and hence non-empty. By flatness of $\zeta$, all the E-O strata are non-empty.

For (2), by \cite{Hasse invariants} Theorem 3.3.1 (2), each E-O is a locally closed subscheme of an affine scheme, and hence quasi-affine.
\end{proof}

When the prime to $p$ level $K^p$ varies, by construction the Ekedahl-Oort strata $\ES_{K_pK^p, \overline{\kappa}}^w$ are invariant under the prime to $p$ Hecke action. In this way we get also the Ekedahl-Oort stratification on $\ES_{K_p,\overline{\kappa}}=\varprojlim_{K^p}\ES_{K_pK^p,\overline{\kappa}}$.

\subsection[Ekedahl-Oort on Shimura varieties of abelian type]{Ekedahl-Oort stratifications  on Shimura varieties of abelian type}
We now explain how to define Ekedahl-Oort stratifications on Shimura varieties of abelian type. As what we did for Newton strata, we would like to begin with the following lemma, which says that if one wants to use the topological space of the quotient stack $[E_{G,\mu}\backslash G_\kappa]$ to parameterize all the Ekedahl-Oort strata, then he could pass to the adjoint group freely.

\begin{lemma}\label{maps induced by central isog}
Let $f:G\rightarrow H$ be a homomorphism of reductive groups over $\mathbb{F}_p$ and $\mu$ be a cocharacter of $G$ defined over a finite field $\kappa$. Denote also by $\mu$ the induced cocharacter of $H$ by $f$. Let $U_{G,-}$, $U_{H,-}$ and $E_{G,\mu}$, $E_{H,\mu}$ be as in \ref{G-zip settings} and \ref{eqnE-act} respectively.
\begin{enumerate}
\item If $U_{G,-}\rightarrow U_{H,-}$ induced by $f$ is smooth, then $f_*:[E_{G,\mu}\backslash G_\kappa]\rightarrow [E_{H,\mu}\backslash H_\kappa]$ is smooth.

\item If $f$ is a central isogeny, then $f_*$ is a smooth homeomorphism.
\end{enumerate}
\end{lemma}
\begin{proof}
To see (1), for $g\in G(\overline{\kappa})$, by the last paragraph of the proof of \cite{EOZ} Theorem 3.1.2, the $E_{G,\mu}$-equivariant morphism $U_{G,-}\times E_{G,\mu}\rightarrow G_\kappa$ given by $(u,g')\mapsto g'\cdot(ug)$ is smooth at $(1,1)\in U_{G,-}\times E_{G,\mu}$. So the induced morphism $U_{G,-}\rightarrow [E_{G,\mu}\backslash G_\kappa]$ is smooth at the identity. Similarly $f(g)\in H(\overline{\kappa})$ induces a morphism $U_{H,-}\rightarrow [E_{H,\mu}\backslash H_\kappa]$ which is smooth at the identity. Consider the commutative diagram
$$\xymatrix{U_{G,-}\ar[r]^{f|_{U_{G,-}}}\ar[d]& U_{H,-}\ar[d]\\
[E_{G,\mu}\backslash G]\ar[r]^{f_*} &[E_{H,\mu}\backslash H],}$$
the composition $U_{G,-}\rightarrow U_{H,-}\rightarrow [E_{H,\mu}\backslash H]$ is smooth at the identity, and hence $f_*$ is smooth in a neighborhood of $g$. But $g$ can be any point, so $f_*$ is smooth.

To see (2), the smoothness follows from (1), as $U_{G,-}\rightarrow U_{H,-}$ is an isomorphism. The homomorphism $f$ is faithfully flat, so is $f_*:[E_{G,\mu}\backslash G]\rightarrow [E_{H,\mu}\backslash H]$. The induced map on topological spaces is then an open surjection. Noting that they both have cardinality $|{}^JW|$, it will then be a homeomorphism.
\end{proof}

\subsubsection[]{}\label{adjoint EO}As what we did for Newton stratifications, we consider adjoint groups first. More precisely, let $(G,X)$ be an \emph{adjoint} Shimura datum of abelian type with good reduction at $p$, and $(G_1,X_1)$ be a Shimura datum of Hodge type satisfying the two conditions in \ref{Kisin's lemma}.

There is an E-O stratification on $\ES_{K_{1,p},\overline{\kappa}}(G_1,X_1)$, we can, as in \ref{adjoint Newton}, restrict it to $\ES_{K_{1,p},\overline{\kappa}}(G_1,X_1)^+$ and then extend it trivially to $\mathscr{A}(G_{\INTP})\times \ES_{K_{1,p},\overline{\kappa}}(G_1,X_1)^+$. We will sometimes call this the \emph{induced E-O stratification} on $\mathscr{A}(G_{\INTP})\times \ES_{K_{1,p},\overline{\kappa}}(G_1,X_1)^+$. Similarly for $\mathscr{A}(G_{1,\INTP})\times \ES_{K_{1,p},\overline{\kappa}}(G_1,X_1)^+$.

\begin{proposition}\label{adjoint EO-part result 1}The induced E-O stratification on $\mathscr{A}(G_{\INTP})\times \ES_{K_{1,p},\overline{\kappa}}(G_1,X_1)^+$ (resp. $\mathscr{A}(G_{1,\INTP})\times \ES_{K_{1,p},\overline{\kappa}}(G_1,X_1)^+$) is $\mathscr{A}(G_{1,\INTP})^\circ$-stable. Moreover, the induced E-O stratification on  $\mathscr{A}(G_{1,\INTP})\times \ES_{K_{1,p},\overline{\kappa}}(G_1,X_1)^+$ descends to the E-O stratification on $\ES_{K_{1,p},\overline{\kappa}}(G_1,X_1)$.
\end{proposition}
\begin{proof}
The proof is identical to that of Proposition \ref{adjoint Newton-result}.
\end{proof}
The induced E-O stratification on  $\mathscr{A}(G_{\INTP})\times \ES_{K_{1,p},\overline{\kappa}}(G_1,X_1)^+$ descends to a stratification on $\ES_{K_{p},\overline{\kappa}}(G,X)$, this will be called the E-O stratification. As before, this does not depend on the choice of $(G_1, X_1)$. More formally, we have the following formulas for $(G_1, X_1)$:
\[\begin{split}
\ES_{K_{1,p},\overline{\kappa}}(G_1,X_1)&=\coprod_{w\in {}^JW_{G_1}}\ES_{K_{1,p},\overline{\kappa}}(G_1,X_1)^w, \\ \ES_{K_{1,p},\overline{\kappa}}(G_1,X_1)^+&=\coprod_{w\in {}^JW_{G_1}}\ES_{K_{1,p},\overline{\kappa}}(G_1,X_1)^{+,w}, \\ \ES_{K_{1,p},\overline{\kappa}}(G_1,X_1)^w&=[\mathscr{A}(G_{1,\INTP})\times \ES_{K_{1,p},\overline{\kappa}}(G_1,X_1)^{+,w}]/\mathscr{A}(G_{1,\INTP})^\circ,
\end{split}\]
and for  $(G, X)$:
\[\begin{split}\ES_{K_{p},\overline{\kappa}}(G,X)&=\coprod_{w\in {}^JW_{G}}\ES_{K_{p},\overline{\kappa}}(G,X)^w, \\ \ES_{K_{p},\overline{\kappa}}(G,X)^+&=\coprod_{w\in {}^JW_{G}}\ES_{K_{p},\overline{\kappa}}(G,X)^{+,w}, \\ \ES_{K_{p},\overline{\kappa}}(G,X)^w&=[\mathscr{A}(G_{\INTP})\times \ES_{K_{p},\overline{\kappa}}(G,X)^{+,w}]/\mathscr{A}(G_{\INTP})^\circ.
\end{split}\]
Moreover, we have
\[\begin{split}
\ES_{K_{p},\overline{\kappa}}(G,X)^{+,w}&=\ES_{K_{1,p},\overline{\kappa}}(G_1,X_1)^{+,w}/\Delta, \\
\ES_{K_{p},\overline{\kappa}}(G,X)^w&=[\mathscr{A}(G_{\INTP})\times \ES_{K_{1,p},\overline{\kappa}}(G_1,X_1)^{+,w}]/\mathscr{A}(G_{1,\INTP})^\circ. \end{split}\]

One can also define E-O stratifications as follows.
\begin{proposition}\label{adjoint EO-part result 2}
We have a commutative diagram of smooth morphisms $$\xymatrix{\ES_{K_{1,p},\overline{\kappa}}(G_1,X_1)\ar[d]^f\ar[r]^(0.5){\zeta_1}&[E_{G_1,\mu}\backslash G_{1,\kappa}]\otimes\overline{\kappa} \ar[d]^{f_*}\\
\ES_{K_p,\overline{\kappa}}(G,X)\ar[r]^(0.5){\zeta_2}& [E_{G,\mu}\backslash G_{\kappa}]\otimes\overline{\kappa} }$$
\end{proposition}
\begin{proof}
The morphism $\ES_{K_{1,p},\overline{\kappa}}(G_1,X_1)\ra\ES_{K_p,\overline{\kappa}}(G,X)$ is \'{e}tale. Smoothness of $\zeta_1$ (resp. $f_*$) was proven in Theorem \ref{th EO Hodge type} (1) (resp. Lemma \ref{maps induced by central isog}). We only need to show how to construct $\zeta_2:\ES_{K_p,\overline{\kappa}}(G,X)\rightarrow [E_{G,\mu}\backslash G_{\kappa}]\otimes \overline{\kappa}$ and why it is smooth.

We use notations as in \ref{twist p-div with addi struc}. Let $\D$ be the $p$-divisible group over $\ES_{K_{1,p},\overline{\kappa}}(G_1,X_1)^+$, $\mathcal{D}[p]$ gives a $G_{1,\kappa}$-zip, and hence a $G_{\kappa}$-zip over $\ES_{K_{1,p},\overline{\kappa}}(G_1,X_1)$. For $\gamma\in G(\mathbb{Z}_{(p)})^+$, we write $\mathcal{P}$ for the fiber in $G_{1,\COMP}$ of $\gamma$ viewed as an element in $G(\mathbb{Z}_{p})$. It is a trivial torsor under the center of $G_{1,\COMP}$. For $\widetilde{\gamma}\in \mathcal{P}(\COMP)$, the isomorphism $\widetilde{\gamma}:\mathcal{D}^\mathcal{P}[p]\rightarrow \mathcal{D}[p]$ induces an isomorphism of $G_{1,\kappa}$-zips, which depends only on $\gamma$ (i.e. is independent of choices of $\widetilde{\gamma}$) when passing to $G_{\kappa}$-zips. But this means that the $G_{\kappa}$-zip attached to $\mathcal{D}[p]$ on $\ES_{K_{1,p},\overline{\kappa}}(G_1,X_1)^+$ descends to $\ES_{K_p,\overline{\kappa}}(G,X)^+$, and hence induces a morphism $\ES_{K_p,\overline{\kappa}}(G,X)^+\rightarrow [E_{G,\mu}\backslash G_{\kappa}]\otimes \overline{\kappa}$ which is necessarily smooth. Putting together these morphisms on geometrically connected components, we get $\zeta_2$ which is necessarily smooth.
\end{proof}
\begin{remark}
By \ref{F-crys to zip}, $\zeta_2$ is actually defined over $\kappa$, the field of definition of $\ES_{K_p,0}(G,X)$.
\end{remark}
\subsubsection[]{}\label{diagram for E-O abe type}Now we consider general Shimura varieties of abelian type. Let $(G,X)$ be a Shimura datum of abelian type (\emph{not adjoint in general}) with good reduction at $p$. Let $(G^\adj,X^\adj)$ be its adjoint Shimura datum, and $(G_1,X_1)$ be a Shimura datum of Hodge type satisfying the two conditions in Lemma \ref{Kisin's lemma} with respect to $(G^\adj,X^\adj)$.

By the previous discussions, we have a commutative diagram
\[\xymatrix{&\ES_{K_{1,p},\overline{\kappa}}(G_1,X_1)\ar[r]\ar[d]&[E_{G_1,\mu}\backslash G_{1,\kappa}]\otimes\overline{\kappa}\ar[d]\\
\ES_{K_p,\overline{\kappa}}(G,X)\ar[r]&\ES_{K_p^\adj,\overline{\kappa}}(G^{\adj}, X^\adj)\ar[r]&[E_{G^\mathrm{ad},\mu}\backslash G^\mathrm{ad}_{\kappa}]\otimes\overline{\kappa}& [E_{G,\mu}\backslash G_{\kappa}]\otimes\overline{\kappa}\ar[l]. }\]

Now we can imitate the main results in Hodge type cases. Fix a sufficiently small open compact subgroup $K^p\subset G(\mathbb{A}_f^p)$. Let us simply write $\ES_{\overline{\kappa}}=\ES_{K,\overline{\kappa}}(G, X)$, and  \[\zeta: \ES_{\overline{\kappa}}\rightarrow [E_{G^\mathrm{ad},\mu}\backslash G^\mathrm{ad}_{\kappa}]\otimes\overline{\kappa}.\]

\begin{theorem}\label{Th EO abelian type}

\begin{enumerate}

\item $\zeta$ is smooth, and it is surjective when $p>2$. In particular,
\begin{enumerate}

\item each stratum $\ES_{\overline{\kappa}}^w$ is a smooth and locally closed subscheme of $\ES_{\overline{\kappa}}$, the closure of $\ES_{\overline{\kappa}}^w$ is a union of strata $\overline{\ES_{\overline{\kappa}}^w}=\coprod_{w'\preceq w}\ES_{\overline{\kappa}}^{w'}$ (recall that the partial order $\preceq$ was introduced above Theorem \ref{collectzipdata});

\item all the strata are in bijection with a subset of ${}^JW$, and for $w\in {}^JW$, the corresponding stratum is of dimension $l(w)$, the length of $w$, if non-empty. When $p>2$, each $\ES_{\overline{\kappa}}^w$ is non-empty.
\end{enumerate}
\item Each E-O stratum $\ES_{\overline{\kappa}}^w$ is quasi-affine.
\end{enumerate}
\end{theorem}

\begin{proof}
Noting that $\ES_{K,0}(G,X)\rightarrow \ES_{K^\adj,0}(G^\adj,X^\adj)$ is \'{e}tale, the smoothness of $\zeta$ is a direct consequence of the previous proposition. All the other statements follow by combining Theorem \ref{th EO Hodge type} with Proposition \ref{adjoint EO-part result 1}.
\end{proof}

Recall by Remark \ref{maximal and minimal EO},
there is a unique maximal length element $w_\mu\in {}^JW$. We call the associated open E-O stratum the ordinary E-O stratum. By the above closure relation, it is dense in $\ES_{\overline{\kappa}}$.
\begin{corollary}\label{C:ordinary}
The $\mu$-ordinary locus in $\ES_{\overline{\kappa}}$ coincides with the ordinary E-O stratum. In particular, the $\mu$-ordinary locus is open dense. 
\end{corollary}
\begin{proof}
In the Hodge type case, this follows from \cite{muordinary} Theorem 6.10. The abelian type case follows from the Hodge type case by our construction.
\end{proof}

Thus for a Shimura datum $(G, X)$ of abelian type with good reduction at $p$, we have the Ekedahl-Oort stratification on the geometric special fiber $\ES_{\overline{\kappa}}$ of $\ES_K(G, X)$
\[\ES_{\overline{\kappa}}=\coprod_{w\in{}^JW}\ES_{\overline{\kappa}}^w,\quad  \overline{\ES_{\overline{\kappa}}^w}=\coprod_{w'\preceq w}\ES_{\overline{\kappa}}^{w'}. \]
As in Remark \ref{maximal and minimal EO}, there is a unique closed (minimal) stratum $\ES_{\overline{\kappa}}^{w_0}$ (the superspecial locus),  associated to the element $w_0=1\in {}^JW$; there is also a unique open (maximal) stratum $\ES_{\overline{\kappa}}^{w_{\mu}}$ (the ordinary locus),  associated to the maximal element $w_{\mu}\in {}^JW$.

\begin{example}\label{EO quaternion SHV}
Notations as in \ref{quaternion SHV}, but we will write $G$ for the special fiber and $W$ for its Weyl group. Then we have \[W\cong (\mathbb{Z}/2\mathbb{Z})^n,\quad W_J\cong (\mathbb{Z}/2\mathbb{Z})^{n-d},\] and \[{}^JW\cong (\mathbb{Z}/2\mathbb{Z})^d.\] The partial order $\preceq$ on ${}^JW$ is the Bruhat order. More explicitly, for $w,w'\in {}^JW$, $w\preceq w'$ if and only if $w$ is obtained from $w'$ by changing some of the $1$ to $0$. The dimension of $\ES_{\overline{\kappa}}^w$ is the number of $1$s in $w$. In particular, for $0\leq i \leq d$, there are ${d\choose i}$ strata of dimension $i$. We refer the reader to \cite{TX} for some related construction for these quaternionic Shimura varieties.
\end{example}

\section[Central leaves]{Central leaves}\label{S:leaves}

In this section, we consider a refinement for both the Newton and the Ekedahl-Oort stratifications studied in the previous two sections.

\subsection[Central leaves on Shimura varieties of Hodge type]{Central leaves on Shimura varieties of Hodge type}
Central leaves were first introduced and studied by Oort in the Siegel case, cf. \cite{foliation-Oort}. They were generalized by Mantovan in the PEL type case in \cite{coho of PEL} and independently by P. Hamacher (\cite{Foliation-Hamacher}) and C. Zhang (\cite{level m}) in the Hodge type case.
 
Notations as in \ref{model Hodge type}, for $z\in \ES_K(G,X)(\overline{\kappa})$, we simply write $D_z$ for $\mathbb{D}(\A_z[p^\infty])(W)$, here $W=W(\overline{\kappa})$. We will also write $L=W[1/p]$ as in \ref{group theo settings for NP}. Two points $x,y\in \ES_K(G,X)(\overline{\kappa})$ are said to be in the same \emph{central leaf} if there exists an isomorphism of Dieudonn\'{e} modules $D_x\rightarrow D_y$ mapping $s_{\cris, x}$ to $s_{\cris, y}$. It is clear from the definition that the $\overline{\kappa}$-points of a Ekedahl-Oort stratum (resp. Newton stratum) is a union of central leaves. We can also define \emph{classical central leaves} by putting together $\overline{\kappa}$-points with isomorphic Dieudonn\'{e} modules. Each classical central leaf is locally closed in $\ES_{K,\overline{\kappa}}(G,X)$.

Let $C(G,\mu)$ and $B(G,\mu)$ be as at the beginning of \ref{group theo settings for NP}. For $x\in \ES_K(G,X)(\overline{\kappa})$, choosing an isomorphism $t:V_{\COMP}^\vee\otimes W\rightarrow D_x$ mapping $s$ to $s_{\cris, x}$, we get a Frobenius on $V_{\COMP}^\vee\otimes W$ which is of the form $(\mathrm{id}\otimes\sigma)\circ g_{x,t}$ with $g_{x,t}$ lies in $G(W)\mu(p)G(W)$. Moreover, changing $t$ to another isomorphism $V_{\COMP}^\vee\otimes W\rightarrow D_x$ mapping $s$ to $s_{\cris, x}$ amounts to $G(W)$-$\sigma$-conjugacy of $g_{x,t}$. So we have a well defined map \[\ES_K(G,X)(\overline{\kappa})\rightarrow C(G,\mu)\] whose fibers are central leaves. The composition \[\ES_K(G,X)(\overline{\kappa})\rightarrow C(G,\mu)\rightarrow B(G,\mu)\] has Newton strata as fibers.

We denote by $\ES_0$ the special fiber of $\ES_K(G,X)$, and by $\nu_G(-)$ the Newton map. For $[b]\in B(G,\mu)$ (resp. $[c]\in C(G,\mu)$), we write $\ES_{\overline{\kappa}}^b$ (resp. $\ES_{\overline{\kappa}}^c$) for the corresponding Newton stratum (resp. central leaf). The main results for central leaves on Shimura varieties of Hodge type are as follows.
\begin{theorem}\label{cent for Hodge}
For $[c]\in C(G,\mu)$, $\ES_{\overline{\kappa}}^c$ is a smooth, equi-dimensional locally closed subscheme of $\ES_{\overline{\kappa}}$. It is open and closed in the classical central leaf containing it, and closed in the Newton stratum containing it. Any central leaf in a Newton stratum $\ES_{\overline{\kappa}}^b$ is of dimension $\langle2\rho,\nu_G(b)\rangle$ if non-empty (this holds when $p>2$). Here $\rho$ is the half sum of positive roots.
\end{theorem}
\begin{proof}
The non-emptiness of $\ES_{\overline{\kappa}}^c$ follows from non-emptiness of Newton strata and \cite{LRKisin} Proposition 1.4.4. All other statements are proved in \cite{Foliation-Hamacher} and \cite{level m} respectively, using different methods.
\end{proof}

When the prime to $p$ level $K^p$ varies, by construction the central leaves $\ES_{K_pK^p, \overline{\kappa}}^c$ are invariant under the prime to $p$ Hecke action. In this way we get also the central leaves on $\ES_{K_p,\overline{\kappa}}=\varprojlim_{K^p}\ES_{K_pK^p,\overline{\kappa}}$.

\subsection[Central leaves on Shimura varieties of abelian type]{Central leaves on Shimura varieties of abelian type}
We now explain how to define central leaves on Shimura varieties of abelian type. As before, we start with a group theoretic result which says that if one wants to use $C(G,\mu)$ to parameterize all central leaves, then he could pass to the adjoint group freely. But due to technical difficulties, we can only prove the following special case.
\begin{lemma}\label{iden C(G,mu)}
Let $f:G\rightarrow H$ be a central isogeny of reductive groups over $\mathbb{Z}_p$ with kernel denoted by $Z$, and $\mu$ be a cocharacter of $G$ defined over $W(\kappa)$ with $\kappa|\mathbb{F}_p$ finite. If $Z_{\mathbb{Q}_p}$ is connected, then
the map $f_*:C(G,\mu)\rightarrow C(H,\mu)$ is a bijection.
\end{lemma}
\begin{proof}
The group scheme $Z$ is of multiplicative type, so we have an exact sequence $$\xymatrix{0\ar[r]&T_Z\ar[r]&Z\ar[r]&Z^{\mathrm{fini}}\ar[r]&0},$$
where $T_Z\subseteq Z$ is the maximal torus, and $Z^{\mathrm{fini}}$ is a finite flat group scheme of multiplicative type. By our assumption, $Z^{\mathrm{fini}}_{\mathbb{Q}_p}$ is trivial, and hence $Z^{\mathrm{fini}}$ is trivial. In particular, $Z=T_Z$ is a torus.
	
Let $W$ be $W(\overline{\kappa})$ and $L$ be $W[1/p]$ as before. To see that $f_*$ is surjective, noting that any element in $C(H,\mu)$ has a representative in $G(L)$ of form $h\mu(p)$ with $h\in H(W)$, it suffices to show that $G(W)\rightarrow H(W)$ is surjective. But $f$ is smooth, so $G(W)\rightarrow H(W)$ is surjective as it is so for $\overline{\kappa}$-points.

Now we prove that $f_*$ is injective. Assume that $g_1\mu(p),g_2\mu(p)\in G(L)$ have the same image in $C(H,\mu)$, then there is $h\in H(W)$ such that \[h^{-1}\overline{g_1}\mu(p)\sigma(h)=\overline{g_2}\mu(p)\in H(L).\]
Here for $i=1,2$, $\overline{g_i}$ is the image in $H(W)$ of $g_i$. Take $g\in G(W)$ mapping to $h$, then \[g^{-1}g_1\mu(p)\sigma(g)\mu(p)^{-1}g_2^{-1}=z\in Z(W),\] here $Z=\mathrm{ker}(f)$ is a torus by the above discussion. We rewrite the above equation as \[g^{-1}g_1\mu(p)\sigma(g)=zg_2\upsilon(p).\] To prove that $f_*$ is injective, it suffices to show that \[z=t^{-1}\sigma(t)\] for some $t\in Z(W)$.

Noting that $Z$ splits over an unramified extension and we are working with $W$-points, we can assume that $Z=\mathbb{G}_{m,W}$. Consider the equation $\sigma(x)=xy$. Writing $x=(x_0,x_1,\cdots)$ and $y=(y_0,y_1\cdots)$ as Witt vectors, the equation becomes \[(x_0^p,x_1^p,\cdots)=(x_0,x_1,\cdots)(y_0,y_1\cdots).\] The multiplication on the right is given by a polynomial $P_n$ of degree $p^n$ with the assignment $\mathrm{deg}(x_i)=\mathrm{deg}(y_i)=p^i$, so for given $(x_0,x_1,\cdots,x_{n-1})$ and $(y_0,y_1\cdots, y_n)$, the equation \[x_n^p-P_n(x,y)=0\] is of form \[x_n^p+a_1x_n+a_0=0,\] and hence always has solution in $k$. But $x_0^p=x_0y_0$ has a non-zero solution for any $y_0\neq 0$, so by induction, $\sigma(x)=xy$ has a solution in $W^\times$ for any $y\in W^\times$.
\end{proof}

\subsubsection[]{}\label{adjoint cent def}Let $(G,X)$ be an \emph{adjoint} Shimura datum of abelian type with good reduction at $p$, and $(G_1,X_1)$ be a Shimura datum of Hodge type satisfying the two conditions in Lemma \ref{Kisin's lemma}. Then the center of $G_{1,\INTP}$ is a torus.

By the last subsection, we have central leaves on $\ES_{K_{1,p},\overline{\kappa}}(G_1,X_1)$. We can restrict them to $\ES_{K_{1,p},\overline{\kappa}}(G_1,X_1)^+$ and then extend them trivially to $\mathscr{A}(G_{\INTP})\times \ES_{K_{1,p},\overline{\kappa}}(G_1,X_1)^+$. We will sometimes call these the \emph{induced central leaves} on $\mathscr{A}(G_{\INTP})\times \ES_{K_{1,p},\overline{\kappa}}(G_1,X_1)^+$. Similarly for $\mathscr{A}(G_{1,\INTP})\times \ES_{K_{1,p},\overline{\kappa}}(G_1,X_1)^+$.

\begin{proposition}\label{adjoint cent result}Each induced central leaf on \[\mathscr{A}(G_{\INTP})\times \ES_{K_{1,p},\overline{\kappa}}(G_1,X_1)^+\] (resp. $\mathscr{A}(G_{1,\INTP})\times \ES_{K_{1,p},\overline{\kappa}}(G_1,X_1)^+$) is $\mathscr{A}(G_{1,\INTP})^\circ$-stable. Moreover, induced central leaves on  $\mathscr{A}(G_{1,\INTP})\times \ES_{K_{1,p},\overline{\kappa}}(G_1,X_1)^+$ descend to central leaves on $\ES_{K_{1,p},\overline{\kappa}}(G_1,X_1)$.
\end{proposition}
\begin{proof}
The proof is identical to that of Proposition \ref{adjoint Newton-result}.
\end{proof}
The central leaves of $\mathscr{A}(G_{\INTP})\times \ES_{K_{1,p},\overline{\kappa}}(G_1,X_1)^+$ descend to locally closed subschemes of $\ES_{K_p,\overline{\kappa}}(G,X)$, and we will call them central leaves of $\ES_{K_{p},\overline{\kappa}}(G,X)$. This does not depend on the choice of $(G_1, X_1)$. More formally, we have the following formulas:
\[\begin{split}
 \ES_{K_{1,p},\overline{\kappa}}(G_1,X_1)^c&=[\mathscr{A}(G_{1,\INTP})\times \ES_{K_{1,p},\overline{\kappa}}(G_1,X_1)^{+,c}]/\mathscr{A}(G_{1,\INTP})^\circ,\\
 \ES_{K_{p},\overline{\kappa}}(G,X)^c&=[\mathscr{A}(G_{\INTP})\times \ES_{K_{p},\overline{\kappa}}(G,X)^{+,c}]/\mathscr{A}(G_{\INTP})^\circ.
\end{split}\]
Moreover, we have
\[\begin{split}
\ES_{K_{p},\overline{\kappa}}(G,X)^{+,c}&=\ES_{K_{1,p},\overline{\kappa}}(G_1,X_1)^{+,c}/\Delta, \\
\ES_{K_{p},\overline{\kappa}}(G,X)^c&=[\mathscr{A}(G_{\INTP})\times \ES_{K_{1,p},\overline{\kappa}}(G_1,X_1)^{+,c}]/\mathscr{A}(G_{1,\INTP})^\circ. \end{split}\]

The proposition also indicates how to relate central leaves to the group theoretic object $C(G,\mu)$. For $x\in \ES_{K_{p},\overline{\kappa}}(G,X)(\overline{\kappa})$, we can find $x_0\in \ES_{K_{p},\overline{\kappa}}(G,X)^+(\overline{\kappa})$ which is in the same central leaf as $x$. Noting that $x_0$ lifts to $\widetilde{x_0}\in \ES_{K_{1,p},\overline{\kappa}}(G_1,X_1)^+(\overline{\kappa})$ whose image in $C(G,\mu)$ depends only on $x$, we get a well defined map \[\ES_{K_{p},\overline{\kappa}}(G,X)(\overline{\kappa})\rightarrow C(G,\mu)\] whose fibers are central leaves of $\ES_{K_{p},\overline{\kappa}}(G,X)$.

\subsubsection[]{}Now we can pass to general Shimura varieties of abelian type. Let $(G,X)$ be a Shimura datum of abelian type (\emph{not adjoint in general}) with good reduction at $p$. Let $(G^\adj,X^\adj)$ be its adjoint Shimura datum, and $(G_1,X_1)$ be a Shimura datum of Hodge type satisfying the two conditions in Lemma \ref{Kisin's lemma} with respect to $(G^\adj,X^\adj)$. In particular the center $Z_{G_1}$ is connected.
By Lemma \ref{iden C(G,mu)}, we have $C(G_1,\mu_1)\cong C(G^\adj,\mu)$, and by the above discussions, we have a commutative diagram
\[\xymatrix{&\ES_{K_{1,p},\overline{\kappa}}(G_1,X_1)(\overline{\kappa})\ar[r]\ar[d]&C(G_1,\mu_1)\ar[d]_{\simeq}\\
\ES_{K_p,\overline{\kappa}}(G,X)(\overline{\kappa})\ar[r]&\ES_{K_p^\adj,\overline{\kappa}}(G^\adj,X^\adj)(\overline{\kappa})\ar[r]&C(G^\adj,\mu).}\]
Here as in \ref{diagram for NP abe type}, we use the same notation when viewing $\mu$ (resp. $\mu_1$) as a cocharacter of $G^\adj$, and identify $B(G^\adj,\mu)$ with $B(G^\adj,\mu_1)$ silently.

Now we can imitate the main results in Hodge type cases. We fix a prime to $p$ level $K^p$. Let $\ES_0$ be the special fiber of $\ES_K(G,X)$, and by $\nu_G(-)$ the Newton map. For $[b]\in B(G,\mu)\simeq B(G^\adj,\mu)$ (resp. $[c]\in C(G^\adj,\mu)$), we write $\ES_{\overline{\kappa}}^b$ (resp. $\ES_{\overline{\kappa}}^c$) for the corresponding Newton stratum (resp. central leaf). 
\begin{theorem}\label{cent abelian ty-ad}
Each central leaf is a smooth, equi-dimensional locally closed subscheme of $\ES_{\overline{\kappa}}$. It is closed in the Newton stratum containing it. Any central leaf in a Newton stratum $\ES_{\overline{\kappa}}^b$ is of dimension $\langle2\rho,\nu_G(b)\rangle$ if non-empty (this holds when $p>2$). Here $\rho$ is the half sum of positive roots.
\end{theorem}
\begin{proof}
For $\ES_0(G^\adj,X^\adj)$ and $[b]\in B(G^\adj,\mu)$, the statement follows by combining Theorem \ref{cent for Hodge} with Proposition \ref{adjoint cent result}. But then the general case follows by noting that $\ES_0(G,X)\rightarrow\ES_0(G^\adj,X^\adj)$ is finite \'{e}tale.
\end{proof}
\begin{example}\label{CL quaternion SHV}Notations as in \ref{NP quaternion SHV}. For $[b]\in B(G,\mu)$, its projection to $B(G_{\mathfrak{p}_i},\mu_i)$ is of form $(\frac{\lambda_1}{n_i},\frac{\lambda_2}{n_i})$ with $\lambda_1\geq \lambda_2$, $\lambda_1+ \lambda_2=a_i$ and these $\lambda_i$ are integers unless $\lambda_1=\lambda_2$. Let \[c_i(b):=\lambda_1-\lambda_2,\] then central leaves in $\ES_{\overline{\kappa}}^b$ are smooth varieties of dimension $\sum_{i=1}^sc_i(b)$.
\end{example}

\

\section[Filtered $F$-crystals with $G$-structure and stratifications]{Filtered $F$-crystals with $G$-structure and stratifications}

In this section, we revisit the Newton stratification, the Ekedahl-Oort stratification, and the central leaves on Shimura varieties of abelian type studied previously from the point of view of $p$-adic Hodge theory. We assume $p>2$ in this section.

\subsection[Filtered $F$-crystals with $G$-structure]{Filtered $F$-crystals with $G$-structure}\label{subsection filtered}
We will mainly follow \cite{Loverig 2} in this and the next subsections.
For a scheme $X$, we will write $\mathrm{Bun}_X$ for the groupoid of vector bundles (of finite rank) over $X$. By a \emph{filtration} $\mathrm{Fil}^\bullet$ on a vector bundle $N/X$, we mean a separating exhaustive descending filtration such that $\mathrm{Fil}^{i+1}$ is a locally direct summand of $\mathrm{Fil}^{i}$. The groupoid of vector bundles over $X$ with a filtration is denoted by $\mathrm{Fil}_X$. Both $\mathrm{Bun}_X$ and $\mathrm{Fil}_X$ are rigid exact tensor categories.

\subsubsection[$G$-bundles and filtered $G$-bundles]{$G$-bundles and filtered $G$-bundles}

Let $G$ be an fppf affine group scheme over $S=\Spec R$. We write $\mathrm{Rep}_R(G)$ for the category of algebraic representations of $G$ taking values in finite projective $R$-modules. Let $X$ be a scheme which is faithfully flat over $S$. By a \emph{$G$-bundle} on $X$, we mean a faithful exact $R$-linear
tensor functor \[\mathrm{Rep}_R(G)\rightarrow \mathrm{Bun}_X.\] By a \emph{filtered $G$-bundle} on $X$, we mean a faithful exact $R$-linear tensor functor \[\mathrm{Rep}_R(G)\rightarrow \mathrm{Fil}_X.\]

\emph{For simplicity, we assume that $R=\mathbb{Z}_p$ and $G$ is reductive from now on}.

By \cite{G-torsor over dedk} Theorem 1.2, to give a $G$-bundle on $X$ is the same as to give a $G$-torsor on $X$. As explained in \cite{Loverig 2} 2.2.8, by putting together Propositions 2.1.5 and 2.2.7 of loc. cit., we find that to give a filtered $G$-bundle on $X$ is the same as to give a $G$-torsor $I/X$ together with a $G$-equivariant morphism $I\rightarrow \mathcal{P}$. Here $\mathcal{P}$ is the scheme of parabolic subgroups of $G$.

One can also talk about the \emph{type} of a filtered $G$-bundle. More precisely, we fix the type $\tau$ of a conjugacy class of parabolic subgroups of $G$, it is defined over a finite \'{e}tale extension $A$ of $R$. Assume that the structure map $X\rightarrow S=\Spec R$ factors through $\Spec A$. Then a filtered $G$-bundle is said to be of type $\tau$ if the associated morphism $I\rightarrow \mathcal{P}$ factors through $\mathcal{P}^\tau$. Here $\mathcal{P}^\tau\subseteq \mathcal{P}$ is the subscheme of parabolic subgroups of $G$ of type $\tau$. It is smooth over $A$ with geometrically connected fibers.

\subsubsection[The operator $\mathrm{R}$]{The functor $\mathrm{R}$}\label{oper R}

For $(N,\mathrm{Fil}^\bullet)\in \mathrm{Fil}_X$, we define $$\mathrm{R}(N):=\sum_i p^{-i}\mathrm{Fil}^i\subseteq N[p^{-1}].$$ By the proof of \cite{Loverig 2} Proposition 2.1.5, $\mathrm{R}(-)$ is an exact tensor functor from $\mathrm{Fil}_X$ to $\mathrm{Bun}_X$ compatible with taking duals. In fact, here our $\mathrm{R}(-)$ is just the specialization of the $\mathrm{Rees}(-)$ in loc. cit. to $t=p$.

\subsubsection[Filtered F-crystals]{Filtered F-crystals}

Let $\kappa|\mathbb{F}_p$ be a finite field and $Y/W(\kappa)$ be a smooth scheme. We denote by $\mathrm{Bun}_Y^\nabla$ (resp. $\mathrm{Fil}^\nabla_Y$) the category of vector bundles on $Y$ with integrable connection (resp. filtered vector bundles on $Y$ with integrable connection satisfying the Griffiths transversality). Let $W=W(\kappa), K=W[1/p]$, $X$ be the formal scheme over $W$ obtained by $p$-adic completion of $Y$, and $X_K$ be the rigid generic fibre over $\mathrm{Spa}(K,W)$. We write $\mathrm{Bun}_X^\nabla$ (resp. $\mathrm{Bun}_{X_K}^\nabla$, $\mathrm{Fil}^\nabla_X$, $\mathrm{Fil}^\nabla_{X_K}$) for the similar category but with the condition that $\nabla$ is topologically quasi-nilpotent.  An object in $\mathrm{Bun}_X^\nabla$ (resp. $\mathrm{Bun}_{X_K}^\nabla$, $\mathrm{Fil}^\nabla_X$, $\mathrm{Fil}^\nabla_{X_K}$) is called a crystal (resp. an isocrystal, a filtered crystal, a filtered isocrystal). There is an obvious commutative diagram\[\xymatrix{\mathrm{Fil}^\nabla_X\ar[r]\ar[d]& \mathrm{Fil}^\nabla_{X_K}\ar[d]\\
\mathrm{Bun}_X^\nabla\ar[r]& \mathrm{Bun}_{X_K}^\nabla	,
}\]where the horizontal arrows are the functors which take generic fibers.

Let $U\subseteq X$ be open affine, and $\sigma_U$ be a lift of the Frobenius on the special fiber of $U$. An  $F$-\emph{isocrystal} is an isocrystal $M/X_K$ together with for each pair $(U,\sigma_U)$ an isomorphism
$\varphi_{\sigma_U}: \sigma^*_UM_{U_K}\rightarrow
M_{U_K}$,
such that the $\varphi_{\sigma_U}$ are horizontal with respect to the
natural connections on both sides, and that if $(U',\sigma_{U'})$ is another pair,  the composition
$$\sigma^*_UM_{U_K\cap U'_K}\stackrel{\varphi_{\sigma_U}}{\longrightarrow}
M_{U_K\cap U'_K}\stackrel{\varphi_{\sigma_{U'}}}{\longleftarrow}\sigma^*_{U'}M_{U_K\cap U'_K}$$ is the natural isomorphism induced by the connection $\nabla$. One can define an $F$\emph{-crystal} to be a ``lattice'' of an $F$-isocrystal. More precisely, it is an $F$-isocrystal
$M/X_K$ together with a crystal $N/X$ and an identification $N[1/p]\cong M$. The category of $F$-isocrystals (resp. $F$-crystals) over $X$ is denoted by $\mathrm{FIsoCrys}_{X_K}$ (resp. $\mathrm{FCrys}_{X}$). We have a natural functor $\mathrm{FCrys}_{X}\ra \mathrm{FIsoCrys}_{X_K}$.

A \emph{filtered} $F$-\emph{crystal} on $X$ is then a filtered crystal $(N,\mathrm{Fil}^\bullet, \nabla)\in \mathrm{Fil}^\nabla_X$ together with for each pair $(U,\sigma_U)$ as above a horizontal isomorphism \[\varphi_U:\mathrm{R}(\sigma^*_UN_{U})\rightarrow N_{U}\] which forms an isocrystal after inverting $p$ (see also \cite{Loverig 2} 2.4.6, \cite{crys and p-galois} II. d) and e), and \cite{intcrys} section 3). Here $\mathrm{R}(\sigma^*_UN_U)$ as in \ref{oper R} is canonically a submodule of $\sigma^*_U(N_{U})[p^{-1}]$, and is equipped with a canonical flat connection by \cite{crys and p-galois} Page 34. In particular, the words ``horizontal'' and ``isocrystal'' make sense. The category of filtered $F$-crystals on $X$ is denoted by $\mathrm{FFCrys}_{X}$. Similarly (and more easily) we have the category of filtered $F$-isocrystals $\mathrm{FFIsoCrys}_{X_K}$. There is an obvious commutative diagram 
\[\xymatrix{\mathrm{FFCrys}_{X}\ar[r]\ar[d]& \mathrm{FFIsoCrys}_{X_K}\ar[d]\\
\mathrm{FCrys}_{X}\ar[r]& \mathrm{FIsoCrys}_{X_K}.
}\]

A \emph{filtered} $F$\emph{-crystal with} $G$-\emph{structure} is then a $\COMP$-linear faithful exact tensor functor \[\omega:\mathrm{Rep}_{\COMP}(G)\rightarrow \mathrm{FFCrys}_{X}.\] Similarly, 
a \emph{filtered} $F$\emph{-isocrystal with} $G$-\emph{structure} is then a $\Q_p$-linear exact tensor functor \[\omega:\mathrm{Rep}_{\Q_p}(G)\rightarrow \mathrm{FFIsoCrys}_{X_K}.\]These objects can be equivalently defined as  filtered $G$-bundles with a flat topologically quasi-nilpotent connection and certain further structures, for more details, see \cite{Loverig 2} 2.4.7 and 2.4.9.

\subsection[Filtered F-crystals on Shimura varieties]{Filtered $F$-crystals on Shimura varieties}

Notations and assumptions as in 1.2.5. We will write $\widehat{\ES_{K_p}}$ for the $p$-adic completion of the integral canonical model  $\ES_{K_p}:= \varprojlim_{K^p}\ES_K(G,X)$. This is a formal scheme over $O_{E_v}=W(\kappa)$ which is formally smooth. Its generic fiber, as an adic space over $\mathrm{Spa}(E_v,O_{E_v})$, is still denoted by $\Sh_{K_p}(G,X)$. We will sometimes simply write $\Sh_{K_p}$ for it.

\subsubsection[]{}\label{I_E_v} Let $Z_{nc}\subseteq Z_G$ be the largest subtorus of $Z_G$ that is split over $\mathbb{R}$ but anisotropic over $\mathbb{Q}$, and set $G^c =G/Z_{nc}$. If $(G,X)$ is a Shimura datum of Hodge type, then we have $G=G^c$.
Let $G_{\COMP}$ (resp. $G^c_{\COMP}$) be the reductive model of $G_{\mathbb{Q}_p}$ (resp. $G^c_{\mathbb{Q}_p}$). We will write $\mathrm{Rep}_{\mathbb{Q}_p}(G)$ (resp. $\mathrm{Rep}_{\mathbb{Z}_p}(G)$) for the category of algebraic representations of $G_{\mathbb{Q}_p}$ (resp. $G_{\mathbb{Z}_p}$) taking values in finite dimensional $\mathbb{Q}_p$-vector spaces (finite free $\mathbb{Z}_p$-modules). Similarly for $G^c$.

Let $\mathrm{Lisse}_{\mathbb{Z}_p}(\Sh_{K_p})$ (resp. $\mathrm{Lisse}_{\mathbb{Q}_p}(\Sh_{K_p})$ ) be the categroy of $\mathbb{Z}_p$-local systems (resp. $\mathbb{Q}_p$-local systems) on $\Sh_{K_p}$. By \cite{Liu-Zhu} page 340-341, the pro-Galois $G^c(\COMP)$-cover
$\Sh(G,X)\rightarrow \Sh_{K_p}(G,X)$
gives a $\mathbb{Z}_p$-linear faithful exact tensor functor
$$\omega_{\text{\'{e}t}}: \mathrm{Rep}_{\mathbb{Z}_p}(G^c)\rightarrow  \mathrm{Lisse}_{\mathbb{Z}_p}(\Sh_{K_p}),$$
which induces a
$\mathbb{Q}_p$-linear tensor functor
\[\omega_{\text{\'{e}t},\Q_p}: \mathrm{Rep}_{\mathbb{Q}_p}(G^c)\rightarrow  \mathrm{Lisse}_{\mathbb{Q}_p}(\Sh_{K_p}). \] By Theorem 1.2 in \cite{Liu-Zhu} of Liu and Zhu, it is de Rham and thus by comparison theorem it extends to a functor
$$\omega_{\mathrm{dR}}: \mathrm{Rep}_{\mathbb{Q}_p}(G^c)\rightarrow  \mathrm{Fil}^\nabla_{\Sh_{K_p}}.$$
This $\omega_{\mathrm{dR}}$ factors via $\mathrm{Rep}_{E_v}(G^c_{E_v})\rightarrow  \mathrm{Fil}^\nabla_{\Sh_{K_p}}$ which defines a filtered $G^c$-bundle $I_{E_v}$ with flat connection on $\Sh_{K_p}$. Liu and Zhu conjecture (see \cite{Liu-Zhu} Remark 4.1 (ii)) that this should agree with the analytification of the canonical model of the automorphic vector constructed by Milne in the case when $Z(G)^\circ$ is split by a CM
field.  By using the theory of abelian motives, this is true in the abelian type case (compare \cite{Loverig 2} 3.1.3.).

\subsubsection[]{} Lovering constructs in \cite{Loverig 2} a certain filtered $F$-crystal with $G^c_{\COMP}$-structure over $\widehat{\ES_{K_p}}$ whose underlying filtered isocrystal on the generic fiber is $\omega_{\mathrm{dR}}$. Lovering calls it the ``crystalline canonical model'' of $\omega_{\mathrm{dR}}$ (or $I_{E_v}$). It is characterized by a CPLF condition (means ``crystalline points lattice $+$ Frobenius'', see \cite{Loverig 2} 3.1.5 for the precise definition). Roughly speaking, this condition is imposed to ensure that one can have certain integral crystalline comparison theorem between $\omega_{\text{\'{e}t}}$ and $\omega_{\text{cris}}$ (see below).
By \cite{Loverig 2} Proposition 3.1.6, a crystalline canonical model, if exists, is unique up to isomorphism. We will write \[\omega_{\text{cris}}:\mathrm{Rep}_{\mathbb{Z}_p}(G^c)\rightarrow  \mathrm{FFCrys}_{\widehat{\ES_{K_p}}},\] and sometimes $I$, for the crystalline canonical model of $\omega_{\mathrm{dR}}$.

By \cite{Loverig 1}  Lemma 3.1.3, a morphism  $(G,X)\rightarrow (G',X')$ of Shimura
data induces a homomorphism $G^c\rightarrow G'^c$. If moreover, it comes from a morphism of reductive group schemes $G_{\mathbb{Z}_{(p)}}\rightarrow G'_{\mathbb{Z}_{(p)}}$, we have a natural homomorphism $G^c_{\mathbb{Z}_{(p)}}\rightarrow G'^c_{\mathbb{Z}_{(p)}}$.

\begin{theorem}(\cite{Loverig 2} 3.4.8, Proposition 3.1.6)\label{main result Lov 2}
\begin{enumerate}

\item If $(G,X)$ is of abelian type, then the crystalline canonical model of $\omega_{\mathrm{dR}}$ exists.

\item Let $f : (G,X)\rightarrow (G',X')$ be a morphism of Shimura data of abelian type induced by a homomorphism $G_{\mathbb{Z}_{(p)}}\rightarrow G'_{\mathbb{Z}_{(p)}}$ of reductive groups over $\INT_{(p)}$, and $I$ (resp. $I'$) be the crystalline canonical model over
$\widehat{\ES_{K_p}}$ (resp. $\widehat{\ES'_{K_p'}}$). Then we have a canonical isomorphism $I\times^{G^c_{\COMP}}
G'^c_{\COMP}\cong f^*I' $ of filtered $F$-crystals over $\widehat{\ES_{K_p}}$ with $G'^c_{\COMP}$-structure.
\end{enumerate}
\end{theorem}

\begin{remark}
Notations as in the above theorem. The morphism $I\times^{G^c_{\COMP}}
G'^c_{\COMP}\cong f^*I' $ in (2) is stated in \cite{Loverig 2} Proposition 3.1.6 (2) as an isomorphism of \emph{weak} filtered $F$-crystals with $G'^c_{\COMP}$-structure. But $I\times^{G^c_{\COMP}}
G'^c_{\COMP}$ given by $$\mathrm{Rep}_{\mathbb{Z}_p}(G'^c)\rightarrow \mathrm{Rep}_{\mathbb{Z}_p}(G^c)\rightarrow \mathrm{FFCrys}_{\widehat{\ES_{K_p}}} $$ is by definition a filtered $F$-crystal with $G'^c_{\COMP}$-structure, and hence $f^*I'$ is a filtered $F$-crystal with $G'^c_{\COMP}$-structure. It is in general difficult to determine whether the base-change of a filtered $F$-crystal is again a filtered $F$-crystal.
\end{remark}

\begin{remark}\label{type of fil G-crys}
Notations as above. Let $\tau$ be a type of parabolic subgroups of $G_{\COMP}$ defined over $W(\kappa)$. Then a filtered $F$-crystal with $G^c_{\COMP}$-structure (over $\widehat{\ES_{K_p}}$) is said to be of type $\tau$ if its underlying filtered $G^c_{\COMP}$-bundle is of type $\tau$. Here we view $\tau$ as a type of parabolic subgroups of $G^c_{\COMP}$. The crystalline canonical model $\omega_{\mathrm{cris}}$ of $\omega_{\mathrm{dR}}$ is of type $\mu$. Here we write $\mu$ for the type of $P_+\subseteq G^c_{W(\kappa)}$ where $\mu$ is viewed as a cocharacter of $G^c_{W(\kappa)}$.
\end{remark}

\subsection[Definitions of various stratifications]{Stratifications via filtered $F$-crystals}

We will explain in this subsection, how to define and study stratifications on Shimura varieties of abelian type using the filtered $F$-crystal with $G^c_{\COMP}$-structure $\omega_{\mathrm{cris}}$. The good point is that, this filtered $F$-crystal with $G^c_{\COMP}$-structure is intrinsically determined by the Shimura datum, and once we define stratifications using it, these stratifications will be automatically intrinsically determined by the Shimura datum. In the next subsection we will identify the stratifications with those defined previously in sections 2-4. In particular, our constructions of the Newton strata, E-O strata, central leaves are independent of the choice of Hodge type data.

\subsubsection[]{}\label{crys determined by S_0}Let $A$ be the $p$-adic completion of a formally smooth $W(\kappa)$-algebra, and $\sigma$ be a lifting of the Frobenius of $A_0:=A\otimes_{W(\kappa)}\kappa$. It is well known that an $F$-isocrystal (resp. $F$-crystal) over $A$ depends only on $A_0$ up to isomorphism. We will simply call an $F$-isocrystal (resp. $F$-crystal) over $A$ (or equivalently, over $A_0$) an $F$-isocrystal (resp. $F$-crystal), and the corresponding category is denoted by $\mathrm{FIsoCrys}_{A_0}$ (resp. $\mathrm{FCrys}_{A_0}$).

Let \[\omega_{\text{cris}}:\mathrm{Rep}_{\mathbb{Z}_p}(G^c)\rightarrow  \mathrm{FFCrys}_{\widehat{\ES_{K_p}}}\] be the filtered $F$-crystal with $G^c_{\COMP}$-structure over $\widehat{\ES_{K_p}}$. By forgetting the filtrations, we get a faithful exact tensor functor \[\omega:\mathrm{Rep}_{\mathbb{Z}_p}(G^c)\rightarrow  \mathrm{FCrys}_{\ES_{_{K_p},0}}.\]

Now we can define stratifications on $\ES_{_{K_p},0}$. We will define Newton strata and central leaves pointwise first using $\omega$, and then define Ekedahl-Oort strata using $G_0^c$-zips. For $x\in \ES_{_{K_p},0}(\overline{\kappa})$, pulling back the $F$-crystal with $G^c_{\COMP}$-structure $\omega$ over $\ES_{_{K_p},0}$ to $x$ induces an $F$-crystal with $G^c_{\COMP}$-structure over $\overline{\kappa}$, i.e. a faithful exact $\mathbb{Z}_p$-linear tensor functor \[\omega_{x}:\mathrm{Rep}_{\mathbb{Z}_p}(G^c)\rightarrow  \mathrm{FCrys}_{\overline{\kappa}}.\] Passing to isocrystals, we get an $F$-isocrystal with $G^c_{\mathbb{Q}_p}$-structure, i.e. an exact $\mathbb{Q}_p$-linear tensor functor \[\omega_{x,\mathbb{Q}_p}:\mathrm{Rep}_{\mathbb{Q}_p}(G^c_{\mathbb{Q}_p})\rightarrow  \mathrm{FIsoCrys}_{\overline{\kappa}}.\]
\begin{definition}\label{def canonical central leaves}
Two points $x,y\in \ES_{_{K_p},0}(\overline{\kappa})$ are said to be in the same \emph{central leaf} if the $F$-crystals with $G^c_{\COMP}$-structure $\omega_{x}$ and $\omega_{y}$ are isomorphic. They are said to be in the same \emph{Newton stratum} if the $F$-isocrystals with $G^c_{\mathbb{Q}_p}$-structure $\omega_{x,\mathbb{Q}_p}$ and $\omega_{y,\mathbb{Q}_p}$ are isomorphic.
\end{definition}

Let $\upsilon=\sigma(\mu)$ be the cocharacter of $G_{W(\kappa)}$ with the induced cocharacter of $G^c_{W(\kappa)}$ denoted by the same notation. For $x\in \ES_{_{K_p},0}(\overline{\kappa})$ with a lift $\widetilde{x}\in \ES_{K_p}(W(\overline{\kappa}))$, the torsor $I_{\widetilde{x}}$ is trivial, and we can take $t\in I_{\widetilde{x}}(W(\overline{\kappa}))$ such that the filtration in the filtered $F$-crystal is induced by $\upsilon$. For a faithful representation \[G^c_{\COMP}\rightarrow \GL(L),\] $I_{\widetilde{x}}$ gives a filtered $F$-crystal structure on $L_{W(\overline{\kappa})}$, and the linearization of the Frobenius $\varphi$ is of form $g\upsilon(p)$, where $g\in \GL(L)(W(\overline{\kappa}))$ is the composition
$$\xymatrix{L_{W(\overline{\kappa})}\ar[r]^\xi& L_{W(\overline{\kappa})}^\sigma\ar[r]^{\upsilon(p)^{-1}}& \mathrm{R}(L_{W(\overline{\kappa})}^\sigma)\ar[r]^{\varphi^\mathrm{lin}}& L_{W(\overline{\kappa})}}.$$
Here we use the filtration induced by $\upsilon$ to construct $\mathrm{R}(L_{W(\overline{\kappa})}^\sigma)$, $\varphi^\mathrm{lin}$ is the isomorphism induced by $\varphi$, and the isomorphism $\xi: L_{W(\overline{\kappa})}\rightarrow L_{W(\overline{\kappa})}^\sigma$ is given by $l\otimes k\mapsto l\otimes 1\otimes k$. Let $s\in L^\otimes$ be a tensor fixed by $G^c_{\COMP}$, then it is also in $\mathrm{R}(L_{W(\overline{\kappa})}^\otimes)$, and such that $\varphi(s)=s$. In paticular, $g\in G^c_{\COMP}(W(\overline{\kappa}))$, and the assignment $x\mapsto \sigma^{-1}(g)$ gives well defined maps \[\ES_{_{K_p},0}(\overline{\kappa})\rightarrow C(G^c,\mu)\] and \[\ES_{_{K_p},0}(\overline{\kappa})\rightarrow B(G^c,\mu).\] The fibers of there maps are central leaves and Newton strata respectively. 

\subsubsection[]{}\label{F-crys to zip}We now explain how to define the Ekedahl-Oort stratification. Unlike for central leaves or Newton strata, we can work directly with families using \cite{disinv} Example 7.3. Let \[\omega_{\text{cris}}:\mathrm{Rep}_{\mathbb{Z}_p}(G^c)\rightarrow  \mathrm{FFCrys}_{\widehat{\ES_{K_p}}}\] be the crystalline canonical model of $\omega_{\text{dR}}$ over $\widehat{\ES_{K_p}}$. To define the morphism \[\zeta:\ES_{{K_p},0}\rightarrow [E_{G^c,\mu}\backslash G^c_\kappa],\] we need to construct a $G^c_0$-zip $(I_0,I_{0,+},I_{0,-},\iota)$ of type $\mu$ on $\ES_{K_p,0}$. Here $G^c_0$ is the special fiber of $G^c_{\mathbb{Z}_p}$.

One could get $I_0$ and $I_{0,+}$ (almost) directly from the underlying filtered $G^c_{\mathbb{Z}_p}$-bundle of $\omega_{\text{cris}}$, and $I_{0,-},\varphi$ from the filtered $F$-crystal structure. To get started, we fix a faithful representation \[G^c_{\COMP}\rightarrow \GL(L)\] and a tensor $s\in L^\otimes$ defining $G^c_{\COMP}$. Then $\omega_{\text{cris}}$ gives a filtered $F$-crystal $(M,\mathrm{Fil}^\bullet, \nabla)$ and an embedding of filtered $F$-crystals $s_{\cris}: O_{\ES_{K_p}}\rightarrow M^\otimes$. The reduction mod $p$ of $M$ (resp. $\mathrm{Fil}^\bullet$, $s_{\cris}$) is denoted by $M_0$ (resp. $C^\bullet$, $s_{\cris,0}$).

Now set \[I_0=\IIsom((L_\kappa,s),(M_0,s_{\cris,0})).\] We can also see it without choosing any embedding, as it is the special fiber of the underlying $G^c_{\mathbb{Z}_p}$-bundle $I$. Let $L^\bullet$ be the descending filtration on $L_{W(\kappa)}$ induced by $\mu$, then set $$I_{0,+}=\IIsom((L_\kappa,L^\bullet_{\kappa},s),(M_0,C^\bullet, s_{\cris,0})).$$

We still need to show that $(M_0,C^\bullet)$ can be ``extended'' to an $F$-zip. Let $A$ be an open affine of $\ES_{K_p}$ with a Frobenius lifting $\sigma$ of $A_0:=A/(p)$. Let $D_i|_{A_0}$ be elements $m\in M_0\otimes A_0$ such that there exists $n\in M_A$ with $p^{-i}\varphi(n)\in M_A$ and the image in $M_0\otimes A_0$ of $p^{-i}\varphi(n)$ is $m$. By the discussions in \cite{disinv} Example 7.3, $D_\bullet|_{A_0}$ is a descending filtration on $M_0\otimes A_0$ as in Definition \ref{defini F-zip}, and $p^{-i}\varphi$ induces an $F$-zip \[(M_0\otimes A_0, C^\bullet\otimes A_0, D_\bullet|{A_0},p^{-i}\varphi).\]

We remark that we assumed that $\varphi$ is strongly divisible with respect to all $(A,\sigma)$, so \[(M_0\otimes A_0, C^\bullet\otimes A_0, D_\bullet|{A_0},p^{-i}\varphi)\] is always an $F$-zip. The flat connection induces canonical isomorphism for different choices of $\sigma$. In particular, these $(M_0\otimes A_0, C^\bullet\otimes A_0, D_\bullet|{A_0},p^{-i}\varphi)$ can be glued into an $F$-zip $(M_0, C^\bullet, D_\bullet,\phi_\bullet)$ on $\ES_{{K_p},0}$.

Let $L_\bullet$ be the ascending filtration on $L_{W(\kappa)}$ induced by $\sigma(\mu)$, set \[I_{0,-}=\IIsom((L_\kappa,L_{\bullet,{\kappa}},s),(M_0,D_\bullet, s_{\cris,0})),\] and let $\iota$ be simply the isomorphism induced by $\phi_\bullet$. We remark that the isomorphism $\varphi:\mathrm{R}(M^\sigma)\rightarrow M$ respects $s_{\cris}$. This implies that the morphism $\iota:I_{0,+}/U_+\rightarrow I_{0,-}/U_-^{(p)}$ is well defined.

\begin{definition}
Two points in $\ES_{{K_p},0}(\overline{\kappa})$ are in the same \emph{Ekedahl-Oort stratum} if and only if their attached $G^c_0$-zip functors are isomorphic.
\end{definition}
It is clear by construction that fibers of $\zeta\otimes\overline{\kappa}$ are the Ekedahl-Oort strata in $\ES_{{K_p},\overline{\kappa}}$.

\subsection[Properties of stratifications]{Properties of stratifications}We will study properties of various stratifications here. We will mainly deduce these properties from what we know for those of Hodge type, and also compare the definitions here and those we gave before. It should be possible to study stratications directly using the filtered $F$-crystal with $G^c_{\COMP}$-structure, but we would not do it here.

\subsubsection{Functoriality: some fundamental diagrams}Notations as in Theorem \ref{main result Lov 2}. We assume moreover that both $(G,X)$ and $(G',X')$ are of abelian type. As we have remarked, $f^*I'$ is a filtered $F$-crystal with $G'^c$-structure over $\widehat{\ES_{K_p}}(G,X)$.

We have a canonical identification $I\times ^{G^c}G'^c\cong f^*I'$ which induces, by our discussion in the previous parts, commutative diagrams

$\xymatrix{\ES_{K_p}(G,X)(\overline{\kappa})\ar[d]\ar[r]& B(G^c,\mu)\ar[d]\\
\ES_{K_p'}(G',X')(\overline{\kappa})\ar[r]& B(G'^c,\mu)}$\ \ \ \
$\xymatrix{\ES_{K_p,\kappa}(G,X)\ar[d]\ar[r]& [E_{G^c,\mu}\backslash G^c_\kappa]\ar[d]\\
\ES_{K'_p,\kappa}(G',X')\ar[r]& [E_{G'^c,\mu}\backslash G'^c_\kappa]}$
\[
\xymatrix{\ES_{K_p}(G,X)(\overline{\kappa})\ar[d]\ar[r]& C(G^c,\mu)\ar[d]\\
	\ES_{K_p'}(G',X')(\overline{\kappa})\ar[r]& C(G'^c,\mu).}\]

\subsubsection{Settings}To study properties of stratifications defined using the filtered $F$-crystal with $G^c_{\COMP}$-structure, as well as to compare them with those we defined via passing to adjoint groups (as we will see, they are usually the same thing), we introduce the following settings.

Let $(G,X)$ be Shimura datum of abelian type as above, and $(G_1,X_1)$ be a Shimura datum of Hodge type with $Z_{G_1}$ a torus and $(G^\mathrm{ad},X^\mathrm{ad})\cong (G_1^\mathrm{ad},X_1^\mathrm{ad})$ (see Lemma \ref{Kisin's lemma}). Let $(\mathcal{B},X')$ be the Shimura datum constructed in \cite{Loverig 2} Proposition 3.4.2 (see also \cite{Loverig 1} 4.6) using $G_1^{\mathrm{der}}$ and the reflex field of $(G_1,X_1)$, then there is a commutative diagram of Shimura data $$\xymatrix{
	(\mathcal{B},X')\ar[d]\ar[r]& (G_1,X_1)\ar[d]\\
(G,X)\ar[r]& (G^\mathrm{ad},X^\mathrm{ad})}$$ inducing a commutative diagram of (integral models of) Shimura varieties
$$\xymatrix{\ES_{K_{\mathcal{B},p}}(\mathcal{B},X')\ar[d]\ar[r]& \ES_{K_{1,p}}(G_1,X_1)\ar[d]\\
\ES_{K_p}(G,X)\ar[r]& \ES_{K_p^\adj}(G^\mathrm{ad},X^\adj).}$$
The reflex field of $(\mathcal{B},X')$  is the same as that of $(G_1,X_1)$ by construction (cf. \cite{Loverig 1} 4.6). By Lemma \ref{Kisin's lemma} (2), the local reflex fields of the Shimura varieties in the above diagram are the same. As before, we denote by $\kappa$ the common residue field of the local reflex field $E_v$.
\subsubsection[Newton stratifications]{Newton stratifications}\label{Newton using crys cano}Using the fundamental diagram for Newton strata, we find a commutative diagram
$$\xymatrix@C=0.3cm{
  \ES_{K_{\mathcal{B},p}}(\mathcal{B},X')(\overline{\kappa}) \ar[dr]\ar[rr]\ar[dd]
      &  & \ES_{K_{1,p}}(G_1,X_1)(\overline{\kappa}) \ar[dr]\ar[dd] \\
   & B(\mathcal{B}^c,\mu) \ar'[r][rr] \ar'[d][dd]
      &  & B(G^c_1,\mu) \ar[dd]        \\
\ES_{K_p}(G,X)(\overline{\kappa}) \ar[rr]\ar[dr]
      &  & \ES_{K_p^\adj}(G^\mathrm{ad},X^\adj)(\overline{\kappa}) \ar[dr]   \\
      & B(G^c,\mu) \ar[rr]
      &  & B(G^{\mathrm{ad}},\mu).               }$$

This implies that the Newton stratification on $\ES_{K_{1,p},\kappa}(G_1,X_1)$ (resp. $\ES_{K_p,\kappa}(G,X)$) is a refinement of the pullback of that on $\ES_{K_p^\adj,\kappa}(G^\mathrm{ad},X^\adj)$, and the Newton stratification on $\ES_{K_{\mathcal{B},p},\kappa}(\mathcal{B},X')$ is a refinement of both the pullback of that on $\ES_{K_{1,p},\kappa}(G_1,X_1)$ and that on $\ES_{K_p,\kappa}(G,X)$. However, noting that the maps on $B(-,\mu)$ are bijective, the Newton stratification on $\ES_{K_{\mathcal{B},p},\kappa}(\mathcal{B},X')$ (resp. $\ES_{K_{1,p},\kappa}(G_1,X_1)$, $\ES_{K_p,\kappa}(G,X)$) is just the pullback of that on $\ES_{K_p^\adj,\kappa}(G^\mathrm{ad},X^\adj)$.

By the construction of $\omega_{\text{cris}}$ in the Hodge type case (see \cite{Loverig 2} Theorem 3.3.3), the Newton stratification on $\ES_{K_{1,p},\kappa}(G_1,X_1)$ we defined here coincides with that we defined in \ref{subsection newton Hodge}. So the above discussions also show that the Newton stratification on $\ES_{K_p^\adj,\kappa}(G^\mathrm{ad},X^\adj)$ (and hence the Newton stratification on $\ES_{K_p,\kappa}(G,X)$) we defined here coincides with the one we defined in \ref{diagram for NP abe type}.

\subsubsection[The E-O stratifications]{Ekedahl-Oort stratifications}\label{E-O using crys cano} By the fundamental diagram for E-O stratification, we have a commutative diagram of morphisms of stacks
$$\xymatrix@C=0.3cm{
  \ES_{K_{\mathcal{B},p},\kappa}(\mathcal{B},X') \ar[dr]\ar[rr]\ar[dd]
      &  & \ES_{K_{1,p},\kappa}(G_1,X_1)\ar[dr]\ar[dd] \\
   & [E_{\mathcal{B}^c,\mu}\backslash\mathcal{B}^c_\kappa] \ar'[r][rr] \ar'[d][dd]
      &  &[E_{G^c_1,\mu}\backslash G^c_{1,\kappa}]\ar[dd]        \\
 \ES_{K_p,\kappa}(G,X)\ar[rr]\ar[dr]
      &  & \ES_{K_p^\adj,\kappa}(G^\mathrm{ad},X^\adj) \ar[dr]   \\
      & [E_{G^c,\mu}\backslash G^c_{\kappa}] \ar[rr]
      &  & [E_{G^{\mathrm{ad}},\mu}\backslash G^{\mathrm{ad}}_\kappa].               }$$
Similar to Newton stratifications, the E-O stratification on $ \ES_{K_{\mathcal{B},p},\overline{\kappa}}(\mathcal{B},X')$ (resp. $\ES_{K_{1,p},\overline{\kappa}}(G_1,$ $X_1)$,  $\ES_{K_p,\overline{\kappa}}(G,X)$) is just the pullback of that on $\ES_{K_p^\adj,\overline{\kappa}}(G^\mathrm{ad},X^\adj)$, and the E-O stratification on $\ES_{K_p^\adj,\overline{\kappa}}(G^\mathrm{ad},X^\adj)$ (and hence the E-O stratification on $ \ES_{K_p,\overline{\kappa}}(G,X)$) we defined in \ref{diagram for E-O abe type} coincides with the E-O stratification we defined here. In particular, the morphism $ \ES_{K_p,\overline{\kappa}}(G,X)\rightarrow [E_{G^{\mathrm{ad}},\mu}\backslash G^{\mathrm{ad}}_\kappa]\otimes \overline{\kappa}$ is smooth surjective.

\subsubsection[Central leaves]{Central leaves}\label{central leaves using crys cano}
We have a similar commutative diagram as in \ref{Newton using crys cano} (one only needs to replace $B(-,\mu)$ by $C(-,\mu)$). It implies that central leaves on $\ES_{K_{1,p},\overline{\kappa}}(G_1,X_1)$ (resp. $ \ES_{K_p,\overline{\kappa}}(G,X)$) are refinements of the pullback of those on $\ES_{K_p^\adj,\overline{\kappa}}(G^\mathrm{ad},X^\adj)$, and central leaves on $ \ES_{K_{\mathcal{B},p},\overline{\kappa}}(\mathcal{B},X') $ are refinements of both the pullback of those on $\ES_{K_{1,p},\kappa}(G_1,X_1)$ and those on $ \ES_{K_p,\overline{\kappa}}(G,X)$. Noting that the map $C(G_1,\mu)\rightarrow C(G^{\mathrm{ad}},\mu)$ is bijective, central leaves on $\ES_{K_{1,p},\kappa}(G_1, X_1)$ are just the pullback of those on $\ES_{K_p^\adj,\overline{\kappa}}(G^\mathrm{ad},X^\adj)$, and the central leaves on $\ES_{K_p^\adj,\overline{\kappa}}(G^\mathrm{ad},X^\adj)$ we defined in \ref{adjoint cent def} coincide with the central leaves we defined here.

If the center $Z_G$ is connected, then the central leaves on $\ES_{K_p,\overline{\kappa}}(G,X)$ defined here coincide with what we defined before in \ref{cent abelian ty-ad} by Lemma \ref{iden C(G,mu)}.
In the general case, let us call fibers of $\ES_{K_p,\overline{\kappa}}(G,X)(\overline{\kappa})\rightarrow C(G^c,\mu)$ \emph{canonical central leaves} and those of $\ES_{K_p,\overline{\kappa}}(G,X)(\overline{\kappa})\rightarrow C(G^{\mathrm{ad}},\mu)$ \emph{adjoint central leaves}. In subsequent work we plan to show the following: \emph{a canonical central leaf is a union of  connected components in the adjoint central leaf containing it}. The proof is conceptual and hence a little bit long. We only sketch the idea here: we define and study truncated displays of level $m$ with $G^c$ and $G^{\mathrm{ad}}$ structure respectively, which form algebraic stacks $\mathcal{C}^c_m$ and $\mathcal{C}^{\mathrm{ad}}_m$. The homomorphism $G^c\rightarrow G^{\mathrm{ad}}$ is central, so the induced morphism $\mathcal{C}^c_m\rightarrow\mathcal{C}^{\mathrm{ad}}_m$ has discrete fibers. By \cite{crysboud} Main Theorem 1, for $m$ big enough, $\mathcal{C}^c_m(\overline{\kappa})$ and $\mathcal{C}^{\mathrm{ad}}_m(\overline{\kappa})$ parameterize canonical central leaves and adjoint central leaves respectively, so canonical central leaves are open and closed in adjoint central leaves. In particular, a canonical central leaf is a smooth locally closed subvariety of $\ES_{K_p,\overline{\kappa}}(G,X)$, and it is closed in the Newton stratum containing it. Moreover, for a Newton stratum $\ES_{K_p,\overline{\kappa}}(G,X)^b$, any canonical central leaf, if non-empty, is of dimension $\langle 2\rho,\nu_G(b) \rangle$, where $\rho$ is the half sum of positive roots.



\section[Comparing Ekedahl-Oort and Newton stratifications]{Comparing Ekedahl-Oort and Newton stratifications}

In this section, we study the relations between Ekedahl-Oort strata and Newton strata by group theoretic methods.

\subsection[Group theoretic results]{Group theoretic results}

We will recall some group theoretic results first. The settings are as follows. We start with a pair $(G,\mu)$ where $G$ is a reductive group over $\mathbb{Z}_p$, and $\mu:\mathbb{G}_m\rightarrow G_{W(\kappa)}$ is a minuscule cocharacter defined over $W(\kappa)$ with $\kappa|\mathbb{F}_p$ a finite field. We will write $G_0$ for the special fiber of $G$, $W=W(\overline{\kappa})$, $L= W[1/p]$, $K=G(W)$, and \[K_1=\mathrm{Ker}(K\rightarrow G(\overline{\kappa})).\]

We still denote by $G$ the associated reductive group over $\Q_p$, which is in particular quasi-split.
Let $B\subseteq G$ be a Borel subgroup, $T\subseteq B$ be a maximal torus, and $\mathcal{I}\subset G(L)$ be the Iwahori subgroup attached to $B_0$, the special fiber of $B$.
Let $W_G$ be the Weyl group  with respect to $T$. 
Let \[\widetilde{W}_G:=\mathrm{Norm}_G(T)(L)/T(W)\] be the extended affine Weyl group and $W_a$ be the affine Weyl group. There is a canonical exact sequence
$$\xymatrix{0\ar[r]& X_*(T)\ar[r]&\widetilde{W}_G\ar[r]&W_G\ar[r]&0}$$ and we have $\widetilde{W}_G\cong W_G\ltimes X_*(T)$.
Let $\Omega\subseteq \widetilde{W}_G$ be the stabilizer of the alcove corresponding to the above Iwahori subgroup $\mathcal{I}$  of $G(L)$ given by the preimage of $B(\overline{\kappa})$. Then we have \[\widetilde{W}_G=W_a\rtimes \Omega.\] We define the length function on
$\widetilde{W}_G$ by
\begin{subeqn}\label{length func}
	l(wr)=l(w), \text{ for }w\in W_a, r\in \Omega.
\end{subeqn}

The choice of $B$ (resp. $\mathcal{I}$) determines simple reflections (resp. simple reflections and simple affine roots) in $W_G$ (resp. $\widetilde{W}_G$) denoted by $S$ (resp. $\widetilde{S}$). It also gives the Bruhat order on $W_G$ (resp. $\widetilde{W}_G$), denoted by $\leq$. Clearly, we have $S\subseteq \widetilde{S}$.

\subsubsection[Minimal elements and fundamental elements]{Minimal elements and fundamental elements}\label{minimal to newt cocha}
An element $x\in G(L)$ is called \emph{minimal} if for any $y\in K_1xK_1$, there is a $g\in K$ such that $y=gx\sigma(g)^{-1}$. By \cite{VW} Remark 9.1, if $x$ is minimal, then any element in the $K\text{-}\sigma$-orbit of $x$ is again minimal. For an element $[c]\in C(G)$, we call it minimal if any representive in the $K\text{-}\sigma$-orbit $[c]$ is minimal.

An element $w\in \widetilde{W}_G$ is \emph{fundamental} if $\mathcal{I}w\mathcal{I}$ lies in a single $\mathcal{I}\text{-}\sigma$-orbit. For an element  $w\in \widetilde{W}_G$, we consider the element $w\sigma\in \widetilde{W}_G\rtimes \langle\sigma\rangle$. There exists
$n\in \mathbb{N}$ such that $(w\sigma)^n = t^\lambda$ for some $\lambda\in X_*(T)$. Let $\nu_w$  be the unique
dominant element in the $W_G$-orbit of $\lambda/n$. It is known that $\nu_w$ is independent
of the choice of $n$, and it is the Newton point of $w$
when regarding $w$ as an element in $G(L)$.
We say that an element $w\in\widetilde{W}_G$ is $\sigma$\emph{-straight} if \[l((w\sigma)^n)=nl(w).\] Here $l(-)$ is the length. This is
equivalent to saying that \[l(w)= \langle \nu_w,2\rho \rangle,\] where $\rho$ is the half sum of all
positive roots in the root system of the affine Weyl group. A $\sigma$-conjugacy
class of $\widetilde{W}_G$ is called $\sigma$-straight if it contains a $\sigma$-straight element.

The main results of \cite{Nie Sian} in the above setting are as follows.
\begin{theorem}(\cite{Nie Sian} Theorems 1.3, 1.4, Proposition 1.5)\label{Nie's minimal}

\begin{enumerate}

\item For $w\in \widetilde{W}_G$, it is fundamental if and only if it is $\sigma$-straight.

\item An element $g\in G(L)$ is minimal if and only if it lies in a K-$\sigma$-conjugacy
class of some fundamental element of $\widetilde{W}_G$. Moreover, when $G$ is split, each $\sigma$-conjugacy class of $G(L)$ contains one and only one $K$-$\sigma$-conjugacy class of minimal
elements.

\item If $\mu$ is a minuscule cocharacter of $T$, then each $\sigma$-conjugacy class intersecting $K\mu(p)K$ contains a fundamental element in $W_G\mu(p)W_G$.
\end{enumerate}
\end{theorem}

\subsubsection{$\mathrm{Adm}(\mu)$, $B(G,\mu)$ and $\mathrm{EO}(\mu)$}\label{admi and EO elements} We will introduce some distinguished sets following \cite{coxeter type}. 

For any subset $J$ of $\widetilde{S}$, we denote by $W_J$ the subgroup of $\widetilde{W}_G$ generated
by the simple reflections in $J$ and by ${}^J\widetilde{W}_G$ (resp. $\widetilde{W}_G^J$ ) the set of minimal
length elements for the cosets $W_J\backslash\widetilde{W}_G$ (resp. $\widetilde{W}_G\backslash W_J$). We simply write ${}^J\widetilde{W}_G^K$
for ${}^J\widetilde{W}_G\cap \widetilde{W}_G^K$.

Let $\mu$ be the minuscule cocharacter of $G$ as in the beginning of this section.
The $\mu$-\emph{admissible set} $\mathrm{Adm}(\mu)$ is defined to be $$\mathrm{Adm}(\mu)=\{w\in \widetilde{W}_G\mid w\leq t^{x\mu}\text{ for some }x\in W_G\}.$$ Here we write $t^\lambda$ for elements in the affine part $X_\ast(T)$ of $\widetilde{W}_G$.

Let $B(\widetilde{W}_G)_{\sigma-\mathrm{str}}$ be the set of $\sigma$-straight conjugacy classes of $\widetilde{W}_G$. By \cite{fully H-N decomp} Theorem 1.3 (1), the map 
\[\Psi: B(\widetilde{W}_G)_{\sigma-\mathrm{str}}\ra B(G)\]
induced by the inclusion $N(T)(L)\subset G(L)$ is bijective. Let $\mathrm{Adm}(\mu)_{\sigma-\mathrm{str}}$ be the set of $\sigma$-straight elements in the admissible set $\mathrm{Adm}(\mu)$ and $B(\widetilde{W}_G,\mu)_{\sigma-\mathrm{str}}$ be its image in $B(\widetilde{W}_G)_{\sigma-\mathrm{str}}$. Then by \cite{fully H-N decomp} Theorem 1.3 (2), we have
\[\Psi(B(\widetilde{W}_G,\mu)_{\sigma-\mathrm{str}})=B(G,\mu).\]

The \emph{set of EO elements} $\mathrm{EO}(\mu)$ is defined to be \[\mathrm{EO}(\mu)=\mathrm{Adm}^S(\mu)\cap {}^S\widetilde{W}_G=\mathrm{Adm}(\mu)\cap {}^S\widetilde{W}_G,\] where $\mathrm{Adm}^S(\mu)=W_S\mathrm{Adm}(\mu)W_S$. Here for the second equality, see \cite{He-Rap} Theorem 6.10 for example.

There is a partial order $\preceq$ on ${}^S\widetilde{W}_G$ as follows. For $w,w'\in {}^S\widetilde{W}_G$, $w\preceq w'$ if and only if there exists $x \in W_G$, such that $xw\sigma(x)^{-1}\leq w'$. This partial order restrict to $\mathrm{EO}(\mu)$ and will still be denoted by $\preceq$.

\subsubsection{$\mathrm{EO}(\mu)$ and $G_0$-zips}\label{EO and G-zips}
Before moving on, let's explain how to identify $\mathrm{EO}(\mu)$ (with the partial order $\preceq$) with the topological space of $[E_{G, \mu}\backslash G_{\kappa}]$.

Let $\upsilon=\sigma(\mu)$. Let $\mathcal{T}\subseteq \widetilde{W}_G$ be given by \[\mathcal{T}=\{(w,\upsilon)\in W_G\times X_*(T)\mid w\in {}^\mu W\},\] where ${}^\mu W={}^JW$ using the notation of \ref{subsection group EO}. The set $\mathcal{T}$ is naturally identified with $\mathrm{EO}(\mu)$. Let $x_\upsilon=w_0w_{0,\upsilon}$ where $w_0$ denotes the longest element of $W_G$ and where $w_{0,\upsilon}$ is the longest element of $W_\upsilon$, the Weyl subgroup of the centralizer of $\upsilon$. Then $\tau_\upsilon=x_\upsilon\upsilon(p)$ is the shortest element of $W_G\upsilon(p)W_G$.

By \cite{truncat level 1} Theorem 1.1 (1), the map assigning
to $(w, \upsilon)\in \mathcal{T}$ the $K\text{-}\sigma$-conjugacy class of $K_1\sigma^{-1}(w\tau_\upsilon) K_1$ is a bijection between $\mathcal{T}$ and the set of
$K\text{-}\sigma$-conjugacy classes in $K_1\backslash K\mu(p)K/K_1$. By \cite{muordinary} Proposition 6.7, the assignment $$g_1\mu(p)g_2\mapsto E_{G,\mu}\cdot (\overline{g_2\sigma(g_1)})$$ induces a bijection from the set of $K\text{-}\sigma$-conjugacy classes in $K_1\backslash K\mu(p)K/K_1$ to the set of $\overline{\kappa}$-points of $[E_{G,\mu}\backslash G_{0,\kappa}]$. By Theorem \ref{collectzipdata}, $[E_{G,\mu}\backslash G_{0,\kappa}](\overline{\kappa})\cong {}^\mu W$. In summary, we get \[\mathrm{EO}(\mu)=\mathcal{T}\cong {}^\mu W\cong [E_{G,\mu}\backslash G_{0,\kappa}](\overline{\kappa})\cong K_1\backslash K\mu(p)K/K_1,\] and by the arguments in the proof of \cite{truncat level 1} Corollary 4.7, the induced partial order on the left hand side coincides with $\preceq$. We get a well defined surjective map \[\tilde{\zeta}: C(G,\mu)\ra K_1\backslash K\mu(p)K/K_1\cong \mathrm{EO}(\mu)\cong {}^\mu W.\]

Let $\mathrm{EO}(\mu)_{\sigma-\mathrm{str}}\subseteq \mathrm{EO}(\mu)\subseteq \widetilde{W}_G$ the subset of $\sigma$-straight (equivalently, fundamental) elements. Then for any $w\in \mathrm{EO}(\mu)_{\sigma-\mathrm{str}}$, $\tilde{\zeta}^{-1}(w)$ consists of a single element in $C(G,\mu)$. That is, we get an injection $\mathrm{EO}(\mu)_{\sigma-\mathrm{str}} \hookrightarrow C(G,\mu)$, with the image $C(G,\mu)_{\mathrm{min}}$, the subset of minimal elements in $C(G,\mu)$. Composed with the natural map $C(G,\mu)\ra B(G,\mu)$, by Theorem \ref{Nie's minimal} (3) we get a surjection $\mathrm{EO}(\mu)_{\sigma-\mathrm{str}} \twoheadrightarrow B(G,\mu)$.
In summary, we have the following commutative diagram:
$$\xymatrix{ & & B(G,\mu)\\
	\mathrm{EO}(\mu)_{\sigma-\mathrm{str}}\ar@{^(->}[r]\ar@{->>}[urr]\ar@{^(->}[drr]&C(G,\mu)\ar@{->>}[ur]\ar@{->>}[dr]\\
	& & [E_{G,\mu}\backslash G_\kappa](\overline{\kappa})={}^\mu W\cong \mathrm{EO}(\mu).}$$

\subsubsection[]{}\label{the set Z}Notations as in \ref{admi and EO elements}. Following \cite{coxeter type} 1.5 we have $$Y:=K\mu(p)K=\bigcup_{w\in \mathrm{Adm}(\mu)}KwK=\bigcup_{w\in \mathrm{Adm}^S(\mu)}\mathcal{I}w\mathcal{I}.$$ There is a $K$-action on $G(L)\times Y$ given by $g\cdot(h,y)=(hg^{-1}, gy\sigma(g)^{-1})$. Let $Z$ be the quotient of this action. The map $(h, y) \mapsto(hy\sigma(h)^{-1}, hK)$ gives a bijection $$Z\cong\{(b,gK)\in G(L)\times G(L)/K\mid g^{-1}b\sigma(g)\in Y\}.$$ The projection to the first factor induces a map $Z\rightarrow G(L)$, and its image is a union of $\sigma$-conjugacy classes indexed by $B(G,\mu)$.

For a $\sigma$-conjugacy class $[b]\in B(G,\mu)$, we write $Z_{[b]}\subseteq Z$ for the corresponding subset. The decomposition \[Z=\coprod_{[b]\in B(G,\mu)}Z_{[b]}\] is called the \emph{Newton stratification} of $Z$. For the basic class $[b_0]\in B(G,\mu)$, the corresponding stratum $Z_{[b_0]}$ is called the \emph{basic locus} in $Z$.

Writing $x\cdot_\sigma y$ for $xy\sigma(x)^{-1}$, by \cite{coxeter type} Theorem 3.2.1, we have \[Y=\coprod_{w\in \mathrm{EO}(\mu)}K\cdot_\sigma \mathcal{I}w\mathcal{I}.\] But then \[Z=\coprod_{w\in \mathrm{EO}(\mu)}Z_w,\] where $Z_w=G(L)\times^{K}(K\cdot_\sigma \mathcal{I}w\mathcal{I})$. This decomposition is called the \emph{Ekedahl-Oort stratification} on $Z$.

Given $w\in \mathrm{EO}(\mu)$ and $[b]\in B(G,\mu)$, the intersection $Z_w\cap Z_{[b]}$ is a fiber bundle over $[b]$, and the fiber over $b\in [b]$ is given by $$X_{w}(b):=\{gK\mid g^{-1}b\sigma(g)\in K\cdot_\sigma \mathcal{I}w\mathcal{I}\}\subseteq G(L)/K.$$
Recall that attached to the triple $(G, \{\mu\}, b)$ we have the affine Deligne-Lusztig variety
\[X(\mu, b)=\{gK\mid g^{-1}b\sigma(g)\in K\mu(p)K \}.\]
It admits a perfect scheme structure over $\overline{\kappa}$ by \cite{zhu-aff gras in mixed char}.
By our discussions in \ref{admi and EO elements}, \ref{EO and G-zips} and \cite{fully H-N decomp} 1.4, we have the following decomposition
\[X(\mu, b)=\coprod_{w\in {}^JW}X_w(b). \]
We remark that not every subset $X_w(b)$ in the above decomposition is non-empty (see Proposition \ref{P: EO and Newton}).

\subsubsection{$(G,\mu)$ of Coxeter type}\label{subsubsection coxeter}
We also need a subset $\mathrm{EO}_{\sigma,\mathrm{cox}}(\mu)$ of $\mathrm{EO}(\mu)$. It is the subset of elements $w$ such that $\mathrm{supp}_{\sigma}(w)$ is a proper subset of $\widetilde{S}$ and $w$ that is a $\sigma$-Coxeter element of
$W_{\mathrm{supp}_{\sigma}(w)}$. We will not explain this but just refer to \cite{coxeter type} 2.2.

A pair $(G,\mu)$ with $G$ absolutely quasi-simple is said to be \emph{of Coxeter type} if \[Z_{[b_0]}=\coprod_{w\in \mathrm{EO}_{\sigma,\mathrm{cox}}(\mu)}Z_w.\] A complete list for pairs $(G,\mu)$ of Coxeter type is given in \cite{coxeter type} Theorem 5.1.2. The Newton and Ekedahl-Oort stratifications on $Z$ have very nice properties which we will recall.

Recall that a ranked poset is a partially ordered set (poset) equipped with a
rank function $\rho$ such that whenever $y$ covers $x$, $\rho(y)=\rho(x)+1$. We say that
the partial order of a poset is almost linear if the poset has a rank function
$\rho$ such that for any $x, y$ in the poset, $x <y$ if and only if $\rho(x) < \rho(y)$.
\begin{theorem}(\cite{coxeter type} Theorems 5.2.1, 5.2.2, \cite{coxeter erratum})\label{main th for coxeter type}
Let $(G,\mu)$ be of Coxeter type.
\begin{enumerate}

\item Every Newton stratum of $Z$ is a union of Ekedahl-Oort strata.

\item For any $w\in \mathrm{EO}(\mu)-\mathrm{EO}_{\sigma,\mathrm{cox}}(\mu)$ and $b\in [b_w]$, the $\sigma$-centralizer $J_b$ acts transitively on $X_{w}(b)$.

\item The partial order of $B(G,\mu)$ is almost linear.

\item The partial order $\preceq$ of $\mathrm{EO}_{\sigma,\mathrm{cox}}(\mu)$ coincides with the usual Bruhat
order and is almost linear. Here the rank is the length function.
\end{enumerate}
\end{theorem}

\subsubsection{$(G,\mu)$ of fully Hodge-Newton decomposable type} G\"{o}rtz, He and Nie define and study in \cite{fully H-N decomp} a much more general class of pairs $(G,\mu)$ with the name of being fully Hodge-Newton decomposable. They prove that this is equivalent to property (1) in the previous theorem, and they also give a classification of such pairs. It turns out all the groups in such pairs are classical groups (i.e. reductive groups with simple factors of type A, B, C, and D), cf. \cite{fully H-N decomp} Theorem 2.5.

Let us recall the notion of fully Hodge-Newton decomposable. As always in this paper, we restrict to good reduction cases only. Recall as in \ref{group theo settings for NP} we have the Newton map $\nu=\nu_G: B(G)\ra X_*(T)_{\mathbb{Q},\mathrm{dom}}^\Gamma$.
\begin{definition}[\cite{fully H-N decomp} Definition 2.1, \cite{Chen-Fargues-Shen} 4.3]\label{D:fully HN}
\begin{enumerate}
	
\item Let $M\varsubsetneqq G_{L}$ be a $\sigma$-stable standard Levi subgroup. We say
that $[b]\in B(G,\mu)$ is Hodge-Newton decomposable with respect to $M$ if $M_{\nu(b)}\subseteq M$ and
$\overline{\mu}-\nu(b)\in \mathbb{R}_{\geq0}\Phi^\vee_M$. Here $M_{\nu(b)}\subseteq G_L$ is the centralizer of $\nu(b)$, and $\overline{\mu}=\frac{1}{n_0}\sum_{i=0}^{n_0-1}\sigma^i(\mu)$ with $n_0\in \mathbb{N}$ the order of $\sigma(\mu)$.

\item We say that a pair $(G,\mu)$ is fully Hodge-Newton decomposable if every non-basic $\sigma$-conjugacy class $[b]\in B(G,\mu)$ is Hodge-Newton decomposable with respect to some proper standard
Levi.
\end{enumerate}
\end{definition}
The following is part of \cite{fully H-N decomp} Theorem 2.3 which suffices for our applications.
\begin{theorem}\label{thm fully HN}
The following statements for $(G,\mu)$ are equivalent.
\begin{enumerate}

\item It is fully Hodge-Newton decomposable.

\item For any $w\in \mathrm{EO}(\mu)$, there is a unique $[b]\in B(G,\mu)$ such that $X_w(b)\neq \emptyset$; i.e. every Newton stratum of $Z$ is a union of Ekedahl-Oort strata. Here $Z$, $\mathrm{EO}(\mu)$ and $X_w(b)$ are as in \ref{admi and EO elements}.

\item For any non-basic $[b]\in B(G,\mu)$, $\mathrm{dim} X(\mu,b)=0$.
\end{enumerate}
\end{theorem}
We remark that \cite{fully H-N decomp} Theorem 2.3 is stated only for quasi-simple groups, but by discussions just after the theorem there, it holds in general. We also remark that although it is not stated in the main theorem there, it is true that if $(G,\mu)$ is fully Hodge-Newton decomposable, non-basic elements in $\mathrm{EO}(\mu)$ are $\sigma$-straight (see \cite{fully H-N decomp} Proposition 4.5).

\subsection[Applications to stratifications]{Applications to stratifications}We will explain how to use group theoretic results above to study relations between E-O stratifications and Newton stratifications. Unlike in \cite{coxeter type} or \cite{fully H-N decomp}, we will do this directly and without assuming any results on existence of Rapoport-Zink uniformizations nor the axioms formulated in \cite{fully H-N decomp}.

Notations as in \ref{the set Z}, for $(b,gK)\in Z$ with $b\in G(L)$ and $gK\in G(L)/K$ such that $g^{-1}b\sigma(g)\in K\mu(p)K$, the assignment $(b,gK)\mapsto g^{-1}b\sigma(g)$ induces a well defined surjective map \[Z\twoheadrightarrow C(G,\mu).\] Moreover, the maps $Z\rightarrow B(G,\mu)$ and $Z\rightarrow \mathrm{EO}(\mu)$ factor through $C(G,\mu)$. We have the following commutative diagram:
\[ \xymatrix{ & & B(G,\mu)\\
	Z\ar@{->>}[r]\ar@{->>}[urr]\ar@{->>}[drr]& C(G,\mu)\ar@{->>}[ur]\ar@{->>}[dr]\\
	& & [E_{G,\mu}\backslash G_\kappa](\overline{\kappa})=\mathrm{EO}(\mu).} \]
Let $\overline{Z}_w$ (resp. $\overline{Z}_{[b]}$) be the image of $Z_w$ (resp. $Z_{[b]}$) in $C(G,\mu)$ for $w\in \mathrm{EO}(\mu)$ (resp. $[b]\in B(G,\mu)$). By the commutativity of the above diagram, $\overline{Z}_w$ is the fiber of the canonical projection $C(G,\mu)\ra \mathrm{EO}(\mu)$, and similarly for $\overline{Z}_{[b]}$. We have (Newton and E-O) decompositions \[C(G,\mu)=\coprod_{[b]\in B(G,\mu)}\overline{Z}_{[b]}, \quad C(G,\mu)=\coprod_{w\in \mathrm{EO}(\mu)}\overline{Z}_{w}.\] Then $\overline{Z}_{[b]}=\coprod_i\overline{Z}_{w_i}$ if and only if  $Z_{[b]}=\coprod_iZ_{w_i}$. Moreover, $\overline{Z}_w\cap\overline{Z}_{[b]}\neq\emptyset$ if and only if $Z_w\cap Z_{[b]}\neq\emptyset$, which is then equivalent to that $X_{w}(b)\neq\emptyset$ for some (and hence any) $b\in [b]$.

We fix a prime to $p$ level $K^p$ and simply denote the integral canonical model over $O_{E_v}$ by $\ES=\ES_{K_pK^p}(G,X)$ for a Shimura datum  $(G, X)$ of abelian type with good reduction at $p$. Its geometric special fiber is denoted by $\ES_{\overline{\kappa}}$. We note that the map \[\ES(\overline{\kappa})\ra C(G^{\adj}, \mu)\] constructed in section 4 composed with \[\tilde{\zeta}: C(G^\adj,\mu)\ra [E_{G^{\adj},\mu}\backslash G^\adj_\kappa](\overline{\kappa})\] gives \[\zeta: \ES(\overline{\kappa})\ra [E_{G^{\adj},\mu}\backslash G^\adj_\kappa](\overline{\kappa}).\] In the rest of this section, we will study the Newton stratification, Ekedahl-Oort stratification, and (adjoint) central leaves on $\ES_{\overline{\kappa}}$.
We start with the following commutative diagrams.
\subsubsection[]{General relations}

If we consider stratifications defined by passing to the adjoint ones first, we have a commutative diagram
induced by a similar diagram attached to certain Shimura datum of Hodge type satisfying Lemma \ref{Kisin's lemma}:

\[ \xymatrix{ & & B(G^\adj,\mu)\\
\ES(\overline{\kappa})\ar@{->>}[r]\ar@{->>}[urr]\ar@{->>}[drr]& C(G^\adj,\mu)\ar@{->>}[ur]\ar@{->>}[dr]\\
 & & [E_{G^{\adj},\mu}\backslash G^\adj_\kappa](\overline{\kappa}).} \]
Note that for any $[b]\in B(G,\mu)=B(G^\adj,\mu)$ (resp. $w\in \mathrm{EO}(\mu)$), $\ES_{\overline{\kappa}}^b(\overline{\kappa})$ (resp. $\ES_{\overline{\kappa}}^w(\overline{\kappa})$) is the inverse image of $\overline{Z}_{[b]}$ (resp. $\overline{Z}_w$) under the map  \[\ES(\overline{\kappa})\ra C(G^\adj,\mu)\] and the above decomposition (for $G^\adj$) $C(G^\adj,\mu)=\coprod_{[b]\in B(G,\mu)}\overline{Z}_{[b]}$ (resp. $C(G^\adj,\mu)=\coprod_{w\in \mathrm{EO}(\mu)}\overline{Z}_{w}$).

Similarly, by \ref{F-crys to zip} and the discussions just before it, for stratifications given by $F$-crystals with additional structure, we have a commutative diagram:

\[\xymatrix{ & & B(G^c,\mu)\\
\ES(\overline{\kappa})\ar[r]\ar@{->>}[urr]\ar@{->>}[drr]&C(G^c,\mu)\ar@{->>}[ur]\ar@{->>}[dr]\\
 & & [E_{G^c,\mu}\backslash G^c_\kappa](\overline{\kappa}).}\]

Note that by Lemma \ref{iden C(G,mu)}, we have $C(G,\mu)\cong C(G^c,\mu)$ and the natural map $C(G^c,\mu)\ra C(G^\adj,\mu)$ is a bijection if $Z_{G}$ is connected.
We also remind the readers that the above two diagrams do NOT bring any differences if we just look at the E-O and Newton stratifications (cf. \ref{Newton using crys cano}, \ref{E-O using crys cano}): we have $B(G^c,\mu)=B(G^\adj,\mu), [E_{G^c,\mu}\backslash G^c_\kappa](\overline{\kappa})=[E_{G^{\adj},\mu}\backslash G^\adj_\kappa](\overline{\kappa})$ and the following commutative diagram:
\[\xymatrix{ & & B(G^c,\mu)\ar@{=}[rr]& &B(G^\adj,\mu)\\
	\ES(\overline{\kappa})\ar[r]\ar@{->>}[urr]\ar@{->>}[drr]&C(G^c,\mu)\ar@{->>}[ur]\ar@{->>}[dr]\ar[rr]& & C(G^\adj,\mu)\ar@{->>}[ur]\ar@{->>}[dr]&\\
	& & [E_{G^c,\mu}\backslash G^c_\kappa](\overline{\kappa})\ar@{=}[rr]& &[E_{G^{\adj},\mu}\backslash G^\adj_\kappa](\overline{\kappa}).}\]
\emph{So in the following discussions, we will mean either of these two constructions when talking about E-O or Newton stratification. On the other hand, by a central leaf we will mean the adjoint central leaf defined by a fiber of the map} $\ES(\overline{\kappa})\ra  C(G^\adj,\mu)$.

\begin{definition}
An E-O stratum is said to be minimal\footnote{We remind the readers that this notion is (in general) different from the superspecial locus, i.e. the unique closed E-O stratum attached to $1\in {}^JW$.} if it is a central leaf.
\end{definition}
By our previous discussion, minimal E-O strata are exactly the strata parametrized by the set $\mathrm{EO}(\mu)_{\sigma-\mathrm{str}}$.
\begin{proposition}\label{minimal E-O for ab type}
Each Newton stratum contains a minimal E-O stratum. Moreover, if $G$ splits, then each Newton stratum contains a unique minimal E-O stratum.
\end{proposition}
\begin{proof}
The statements follow from Theorem \ref{Nie's minimal}. 
\end{proof}

\begin{examples}\label{E:ordinary}
\begin{enumerate}
	\item By Corollary \ref{C:ordinary},
the ordinary E-O stratum (cf. Remark \ref{maximal and minimal EO}) coincides with the $\mu$-ordinary locus (i.e. the open Newton stratum, cf. Remark \ref{R:ordinary-basic}), which is a central leaf by Proposition \ref{minimal E-O for ab type}.
\item The superspecial locus (cf. Remark \ref{maximal and minimal EO}) is a central leaf, and thus is minimal. It is contained in the basic locus (i.e. the closed Newton stratum, cf. Remark \ref{R:ordinary-basic}).
\end{enumerate}
\end{examples}

\begin{proposition}\label{P: EO and Newton}
	For any $[b]\in B(G,\mu)$ and $w\in \mathrm{EO}(\mu)\cong {}^JW$, we have
	\[\ES_{\overline{\kappa}}^b\cap \ES_{\overline{\kappa}}^w\neq\emptyset \Longleftrightarrow X_w(b)\neq \emptyset.\]
\end{proposition}
\begin{proof}
	This follows from the fact that each central leaf is non-empty (cf. Theorem \ref{cent abelian ty-ad}) and \cite{muordinary} 6.2 consequences (3).
\end{proof}

\subsubsection{Special relations} By G\"{o}rtz, He and Nie's classification of fully Hodge-Newton decomposable pairs (\cite{fully H-N decomp} Theorem 2.5) and Deligne's classification of Shimura varieties of abelian type (\cite{varideshi} Table 2.3.8), it is natural to discuss fully Hodge-Newton decomposable Shimura data in the framework of Shimura data of abelian type, in view of Kisin's work \cite{CIMK}.
 If $(G,X)$ is fully Hodge-Newton decomposable, we have the followings.
\begin{proposition}\label{E-O NP for ab cox}
Let $(G,X)$ be a Shimura datum of abelian type with good reduction at $p$ whose attached pair $(G,\mu)$ is fully Hodge-Newton decomposable. Then
\begin{enumerate}
\item each Newton stratum of $\ES_{\overline{\kappa}}$ is a union of Ekedahl-Oort strata;

\item each E-O stratum in a non-basic Newton stratum is a central leaf, and it is open and closed in the Newton stratum, in particular, non-basic Newton strata are smooth;

\item if $(G,\mu)$ is of Coxeter type, then for two E-O strata $\ES_{\overline{\kappa}}^{1}$ and $\ES_{\overline{\kappa}}^{2}$, $\ES_{\overline{\kappa}}^{1}$ is in the closure of $\ES_{\overline{\kappa}}^{2}$ if and only if $\mathrm{dim}(\ES_{\overline{\kappa}}^{2})>\mathrm{dim}(\ES_{\overline{\kappa}}^{1})$.
\end{enumerate}
\end{proposition}
\begin{proof}
Statement (1) follows directly from Theorem \ref{thm fully HN}. For (2), the first half follows from our remarks after Theorem \ref{thm fully HN} (which is just \cite{fully H-N decomp} Proposition 4.5); and the second half follows from Theorem \ref{cent abelian ty-ad}. Statement (3) follows from Theorem \ref{main th for coxeter type} (4).
\end{proof}

\begin{example}\label{NH-no quaternion SHV}Notations as in Example \ref{NP quaternion SHV}. The pair $(G,\mu)$ is fully Hodge-Newton decomposable if and only if all the integers $a_i$ are either $1$ or $2$. The if part is clear. To see the only if part, if there is some $a_i\geq 3$, by the dimension formula in Example \ref{NP quaternion SHV} and Example \ref{CL quaternion SHV}, the dimension of the maximal non-ordinary Newton stratum is strictly bigger than that of its central leaves, and hence it is not fully Hodge-Newton decomposable.
\end{example}

\begin{examples}(See also \cite{fully H-N decomp} Theorem 2.5)
\begin{enumerate}
	\item The  unitary Shimura varieties with signature $(1, n-1)\times (0, n)\times \cdots\times (0,n)$ at a split prime $p$  studied by Harris-Taylor in \cite{HT} is fully Hodge-Newton decomposable.
	\item Consider $G=\mathrm{GU}(V,\langle,\rangle)$, the unitary similitude group over $\Q_p$ associated to a Hermintain space $(V,\langle,\rangle)$. Take $\{\mu\}$ such that it corresponds to $\big((1,\cdots,1,0), 0\big)$. Then $(G, \mu)$ is fully Hodge-Newton decomposable by the explicit description of the set $B(G,\mu)$ as in \cite{BW} 3.1. Globally, these are the unitary Shimura varieties studied by B\"ultel-Wedhorn in \cite{BW}.
	\item The pair $(\mathrm{GSp}_4,\mu)$ is fully Hodge-Newton decomposable, where $\mu$ is the cocharacter corresponding to $(1,1,0,0)$. Globally, these are the Siegel modular varieties with genus $g=2$ (Siegel threefolds).
	\item Consider $G=\mathrm{SO}(V, B)$, the special orthogonal group over $\Q_p$ associated to a quadratic space $(V, B)$ of dimension $n+2$. Take $\{\mu\}$ such that it corresponds to $(1,0,\cdots, 0, -1)$. Then $(G, \mu)$ is fully Hodge-Newton decomposable by the explicit description of the set $B(G,\mu)$. Globally, these are the SO-Shimura varieties of orthogonal type, cf. the next section. 
\end{enumerate}
\end{examples}

\section[An example: Shimura varieties of orthogonal type]{Shimura varieties of orthogonal type}\label{S:orthogonal}

We discuss our main results in the setting of Shimura varieties of orthogonal type. These Shimura varieties play very important roles in Kudla's program (\cite{Kud}) and the arithmetic Gan-Gross-Prasad conjecture (\cite{GGP}).

\subsection[Good reduction of Shimura varieties of orthogonal type]{Good reductions of Shimura varieties of orthogonal type}\label{shimura orthogonal}

\subsubsection[The $\mathrm{SO}$-Shimura varieties]{The $\mathrm{SO}$-Shimura varieties}\label{SO shv}

Let $V$ be a $n+2$-dimensional $\mathbb{Q}$-vector space equipped with a non-degenerate
bilinear form $B$ (whose associated quadratic form is) of signature $(n,2)$. Let $\mathrm{SO}(V)$ be the special orthogonal group attached to $(V,B)$, and \[h:\mathbb{S}\rightarrow \mathrm{SO}(V)_{\mathbb{R}}\] be such that
\begin{enumerate}

\item its induced Hodge structure on $V$ is of type $(-1,1)+(0,0)+(1,-1)$ with $\dim V^{-1,1}=1$;

\item $B$ is a polarization of this Hodge structure.
\end{enumerate}
It is well known that $h$ gives a Shimura datum $(\mathrm{SO}(V), X)$.

\subsubsection[The $\mathrm{CSpin}$-Shimura varieties]{The $\mathrm{GSpin}$-Shimura varieties}\label{Spin shv}
Let $C(V)$ and
$C^+(V)$ be the Clifford algebra and even Clifford algebra
respectively. Note that there is an
embedding $V\hookrightarrow C(V)$ and an anti-involution $*$ on
$C(V)$ (see \cite{intspin}, 1.1).

Let $\mathrm{GSpin}(V)$ be the stabilizer in $C^+(V)^\times$ of
$V\hookrightarrow C(V)$ with respect to the conjugation action of
$C^+(V)^\times$ on $C(V)$. Then $\mathrm{GSpin}(V)$ is a reductive
group over $\mathbb{Q}$, and the conjugation action of $\mathrm{GSpin}(V)$ on $V$ induces a homomorphism $\mathrm{GSpin}(V)\rightarrow \mathrm{SO}(V)$. We actually have an exact sequence $$\xymatrix{1\ar[r]&\mathbb{G}_m\ar[r]&\mathrm{GSpin}(V)\ar[r] &\mathrm{SO}(V)\ar[r]&1},$$
where $\mathbb{G}_m$ is identified with invertible scalars in $C^+(V)$.

The homomorphism $h$ in \ref{SO shv} lifts to $\mathrm{GSpin}(V)$ and induces a Shimura datum $(\mathrm{GSpin}(V),$ $X')$ with $X'\simeq X$. Consider the left action of
$\mathrm{GSpin}(V)$ on $C^+(V)$, there is a perfect alternating
form $\psi$ on $C^+(V)$, such that the embedding
$\mathrm{GSpin}(V)\hookrightarrow \mathrm{GL}(C^+(V))$ factors
through $\mathrm{GSp}(C^+(V),\psi)$ and induces an embedding of
Shimura data \[(\mathrm{GSpin}(V), X')\rightarrow
(\mathrm{GSp}(C^+(V),\psi),\mathbb{H}^{\pm}).\] We refer to \cite{intspin} 1.8,
1.9, 3.4, 3.5 for details.

To sum up, $(\mathrm{GSpin}(V), X')$ is a
Shimura datum of Hodge type and $(\mathrm{SO}(V), X)$ is a
Shimura datum of abelian type. One can also see that the reflex field of $(\mathrm{SO}(V), X)$ (resp. $(\mathrm{GSpin}(V), X')$) is $\mathbb{Q}$ if $n>0$. We will assume that $n>0$ from now on.

Let $(G,Y)$  be either $(\mathrm{SO}(V), X)$ or $(\mathrm{GSpin}(V), X')$. Let $K\subseteq
G(\mathbb{A}_f)$ be a compact open subgroup
which is small enough, then
$$\Sh_K:=G(\mathbb{Q})\backslash Y\times
(G(\mathbb{A}_f)/K)$$ has a canonical model over
$\mathbb{Q}$ which will again be denoted by $\Sh_K$. It has dimension $n$.
Let $K\subset \mathrm{GSpin}(V)(\mathbb{A}_f)$ be a sufficently small open compact subgroup, and $K_1\subset \mathrm{SO}(V)(\mathbb{A}_f)$ be its image induced by the map $\mathrm{GSpin}(V)\ra \mathrm{SO}(V)$. Then the induced map between the corresponding Shimura varieties \[\Sh_K(\mathrm{GSpin}(V), X')\ra \Sh_{K_1}(\mathrm{SO}(V), X)\] is a finite \'etale Galois cover, cf. \cite{intspin} 3.2.

\subsubsection[Good reductions]{Good reductions}\label{good reduc ortho type}Let $p>2$ be a prime and $L\subseteq V$ be a
$\mathbb{Z}_{(p)}$-lattice such that the bilinear form $B$ is perfect on it. Then $\mathrm{SO}(L)$ is a reductive group over $\mathbb{Z}_{(p)}$ with generic fiber $\mathrm{SO}(V)$. Similarly, we have $C(L)$, $C^+(L)$ and $\mathrm{GSpin}(L)$, and $\mathrm{GSpin}(L)$ is a reductive group over $\mathbb{Z}_{(p)}$ with generic fiber $\mathrm{GSpin}(V)$.

Let $(G,Y)$  be either $(\mathrm{SO}(V), X)$ or $(\mathrm{GSpin}(V), X')$ as above, and we still write $G$ for its reductive model over $\mathbb{Z}_{(p)}$ by abuse of notation. Let $K_p=G(\COMP)$ and $K^p\subseteq G(\mathbb{A}_f^p)$ be a compact open subgroup
which is small enough. Let $K=K_pK^p$, then by Theorem \ref{int can abelian type},
$\Sh_K$ has an integral canonical model over
$\mathbb{Z}_{(p)}$ denoted by $\ES_K$. Let $K^p\subset \mathrm{GSpin}(V)(\mathbb{A}_f^p)$ be a sufficently small open compact subgroup, and $K_1^p\subset \mathrm{SO}(V)(\mathbb{A}_f^p)$ be its image induced by the map $\mathrm{GSpin}(V)\ra \mathrm{SO}(V)$. Set $K=\mathrm{GSpin}(V)(\mathbb{Z}_p)K^p$, and $K_1= \mathrm{SO}(V)(\mathbb{Z}_p)K_1^p$. Then the induced map between the corresponding integral canonical models \[\ES_K(\mathrm{GSpin}(V), X')\ra \ES_{K_1}(\mathrm{SO}(V), X)\] is a finite \'etale Galois cover, cf. \cite{intspin} Theorem 4.4.

When the level $K$ is clear, the special fiber of $\ES_K$ is denoted by $\ES_0$, and the geometric special fiber is denoted by $\ES_{\overline{\kappa}}$.

\subsection[E-O stratifications]
{Ekedahl-Oort stratifications}\label{subsec EO for ortho}
Let $(G,Y)$ and $\ES_0$ be as above. The Shimura
datum determines a cocharacter
$\mu:\mathbb{G}_{m,\mathbb{Z}_p}\rightarrow
G_{\mathbb{Z}_p}$ which is unique up to
conjugation. The special fiber of
$\mu$ will still be denoted by $\mu$. The cocharacter $\mu$
determines a parabolic subgroup $P_+\subseteq
G_{\mathbb{F}_p}$, whose type will be denoted by
$J$. Let $W$ be the Weyl group of
$G_{\mathbb{F}_p}$, and ${}^JW$ together with the partial order $\preceq$ be as in 3.3 (before Theorem \ref{collectzipdata}). Then Theorem
\ref{Th EO abelian type} implies that the structure of Ekedahl-Oort
stratification on $\ES_{\overline{\kappa}}$ is described by ${}^JW$ together with
the partial order $\preceq$.

\subsubsection[]{A description of $({}^JW,\preceq)$}\label{JW for orthogonal}

Let's recall the description of $({}^JW,\preceq)$ in
\cite{BruhatandFzip} (see also \cite{coxeter type} 6.4 and 6.6). Let $m$ be the dimension of a maximal torus
in $\mathrm{SO}(L_{\mathbb{F}_p})$. There are two cases:

\emph{Case 1.} If $n$ is odd, then the partial order $\preceq$ on ${}^JW$ is a
total order, and the length function induces an isomorphism of
totally ordered sets \[({}^JW,\preceq)\st{\sim}{\rightarrow} \{0,1,2,\cdots,
n\}.\] Note that in this case $n+1=2m$.

\emph{Case 2.} If $n$ is even, noting that in this case $n+2=2m$, then $W$ is
generated by simple reflections $\{s_i\}_{i=1,\cdots, m},$ where
$$s_i=\begin{cases}
(i,i+1)(n-i+2,n-i+3), & \text{ for }i=1,\cdots,m-1;\\
(m-1,m+1)(m,m+2), & \text{ for }i=m.
\end{cases}$$

Let $$w_i=\begin{cases}
s_1s_2\cdots s_i, &\text{ for }i\leq m-1;\\
s_1s_2\cdots s_m, & \text{ for }i=m;\\
s_1s_2\cdots s_ms_{m-2}\cdots s_{2m-i-1}, &\text{ for }i\geq m+1.
\end{cases}$$
and \[w'_{m-1}=s_1s_2\cdots s_{m-2}s_m.\] Then
\[{}^JW=\{w_i\}_{0\leq i\leq n}\cup \{w'_{m-1}\},\] and the partial
order $\preceq$ is given by \begin{equation*}
\begin{split}
w_0=\mathrm{id}\preceq w_1\preceq \cdots &\preceq w_{m-2}\\
&\preceq w_{m-1},w'_{m-1}\\
&\preceq w_m\preceq\cdots\preceq w_n.
 \end{split}
 \end{equation*}

Applying Theorem \ref{Th EO abelian type} together with \ref{JW for
orthogonal}, we get the following description for the E-O stratification
on $\ES_{\overline{\kappa}}$.

\begin{corollary}\label{EO for orth type}
Let $m$ and $n$ be as before.

\begin{enumerate}

\item There are $2m$ Ekedahl-Oort strata on $\ES_{\overline{\kappa}}$.
\item
\begin{enumerate}
\item If $n$ is odd, then for any integer $0\leq i\leq n$, there is
precisely one stratum $\ES^i_{\overline{\kappa}}$ such that $\mathrm{dim}(\ES^i_{\overline{\kappa}})=i$.
These are all the Ekedahl-Oort strata on $\ES_{\overline{\kappa}}$. Moreover, the
Zariski closure of $\ES^i_{\overline{\kappa}}$ is the union of all the $\ES^{i'}_{\overline{\kappa}}$
such that $i'\leq i$.

\item If $n$ is even, then for any integer $i$ such
that $0\leq i\leq n$ and $i\neq n/2$, there is precisely one stratum
$\ES^i_{\overline{\kappa}}$ such that $\mathrm{dim}(\ES^i_{\overline{\kappa}})=i$. There are 2
strata of dimension $n/2$. These are all the Ekedahl-Oort strata
on $\ES_{\overline{\kappa}}$. Moreover, the Zariski closure of the stratum $\ES^w_{\overline{\kappa}}$
is the union of $\ES^w_{\overline{\kappa}}$ with all the strata whose dimensions are
smaller than $\mathrm{dim}(\ES^w_{\overline{\kappa}})$.

\end{enumerate}

\end{enumerate}

\end{corollary}

\subsection[Newton stratifications]
{Newton stratifications}

\subsubsection[]{Orthogonal groups with good reduction at $p$}\label{orthogonal good at p}

Let $(V,q)$ be a non-degenerate quadratic space of dimension $n+2$ over $\RAT_p$. Here we always assume that $n>0$ and $p>2$. If $(V,q)$ is of good reduction (i.e. the orthogonal group $\mathrm{SO}(V,q)$ is of good reduction) at $p$, then we can find a basis $\{e_1,e_2,\cdots,e_{n+2}\}$ such that \[q=a_1x_1^2+a_2x_2^2+\cdots+a_{n+2}x_{n+2}^2\] with $a_i\in \INT_p^\times$.

It is well known that this quadratic space $(V,q)$ is determined up to isomorphism by its discriminant \[\mathrm{d}(V,q):=\prod_{i=1}^{n+2}a_i\] (viewed as an element in $\RAT_p^\times/\RAT_p^{\times 2}$) and Hasse invariant \[\varepsilon(V,q):=\prod_{i<j}(a_i,a_j).\] Here $(a_i,a_j)$ are the Hilbert symbols at $p$. By our assumption, $\varepsilon(V,q)=1$ as $(a_i,a_j)=1$ for any $i<j$. So $(V,q)$ is uniquely determined by its discriminant which is, by assumption, either 1 or represented by a non-square unit $u$ in $\INT_p$.

Fixing a non-square unit $u\in \INT_p$, one can make the above discussions more explicit as follows.

\emph{Case 1.} If $n$ is odd, let $$q=x_1^2-x_2^2+x_3^2-x_4^2+\cdots+x_{2i-1}^2-x_{2i}^2+\cdots+x_{n+2}^2$$

and $$q'=x_1^2-x_2^2+x_3^2-x_4^2+\cdots+x_{2i-1}^2-x_{2i}^2+\cdots+ux_{n+2}^2.$$ Then $(V,q)$ and $(V,q')$ are non-isomorphic, and any non-degenerate quadratic space of rank $n+2$ with good reduction is isomorphic to precisely one of them. One sees easily that both $\mathrm{SO}(V,q)$ and $\mathrm{SO}(V,q')$ are split, and hence they are isomorphic in this case.

\emph{Case 2.} If $n$ is even, let $$q=x_1^2-x_2^2+x_3^2-x_4^2+\cdots+x_{2i-1}^2-x_{2i}^2+\cdots+x_{n+1}^2-x_{n+2}^2$$

and $$q'=x_1^2-x_2^2+x_3^2-x_4^2+\cdots+x_{2i-1}^2-x_{2i}^2+\cdots+x_{n+1}^2-ux_{n+2}^2.$$ Then $(V,q)$ and $(V,q')$ are non-isomorphic, and any non-degenerate quadratic space of rank $n+2$ with good reduction is isomorphic to precisely one of them. One sees easily that $\mathrm{SO}(V,q)$ is split and $\mathrm{SO}(V,q')$ is of rank $m-1$. Here we set $m=(n+2)/2$ as before. In particular, they are not isomorphic in this case.

\subsubsection[]{A description of $B(G_{\mathbb{Q}_p},\mu)$}\label{list of NP ortho} Now we come back to our usual notations (used in subsections \ref{shimura orthogonal}, \ref{subsec EO for ortho}). It is possible (and not difficult) to describe $B(G_{\mathbb{Q}_p},\mu)$ in this case using \cite{isocys with addi} Proposition 6.3. But to keep our arguments short, we use \cite{Chen-Fargues-Shen} Corollary 4.3 which describes $B(G_{\mathbb{Q}_p},\mu)$ in terms of root systems.

More precisely, we fix $T_0\subseteq T\subseteq B$ subgroups of $G_{\RAT_p}$ with $T_0$ a maximal \emph{split} torus, $T$ a maximal torus and $B$ a Borel subgroup. Let $(X^*(T),\Phi,X_*(T),\Phi^\vee)$ be the attached absolute root datum with simple roots $\Delta$, and $(X^*(T_0),\Phi_0,X_*(T_0),\Phi^\vee_0)$ be the attached relative root datum with simple (reduced) roots $\Delta_0$. For $\alpha\in \Delta_0$, we set
$$\widetilde{w}_{\alpha}=\sum_{\beta\in \Phi,\ \beta\mid_{T_0}=\alpha}w_\beta\in X^*(T_0)_{\RAT}.$$
Here $w_\beta$ is the fundamental weight corresponding to $\beta$. Let $\overline{\mu}$ be the average of the $\Gamma$-orbit of $\mu$. Then we have
$$B(G_{\mathbb{Q}_p},\mu)=\{\nu\in X_*(T_0)_{\RAT,\mathrm{dom}}\mid\nu\leq \overline{\mu}, \forall \alpha\in \Delta_0 \text{ with } \langle \nu,\alpha\rangle\neq 0, \langle \overline{\mu}-\nu,\widetilde{w}_{\alpha}\rangle\in \mathbb{N}\}.$$

Combined with \ref{orthogonal good at p}, we can describe $B(G_{\mathbb{Q}_p},\mu)$ explicitly.

\emph{Case 1.} If $n$ is odd, then $T_0=T$. Set $m=(n+1)/2$ as in \ref{subsec EO for ortho}. We can choose a $\RAT_p$-basis with respect to which $$q=x_1x_{2m+1}+x_2x_{2m}+\cdots+x_mx_{m+2}+ux_{m+1}^2.$$
Let $T=\mathrm{diag}(t_1,t_2,\cdots, t_m,1,t_m^{-1},t_{m-1}^{-1}\cdots, t_1^{-1})$ and $\alpha_i\in X^*(T)$, $1\leq i\leq m$, be given by the $i$-th projection. For $1\leq i\leq m$, let $\alpha_i^\vee\in X_\ast(T)$ be the cocharacter \[t\mapsto \mathrm{diag}(1,\cdots, 1, t, 1, \cdots, 1, t^{-1}, 1, \cdots, 1)\] where the $t$ and $t^{-1}$ are at the $i$-th and $2m+2-i$-th place respectively. Then $\mu=\overline{\mu}=\alpha_1^\vee$. For \[\nu=\sum_{i=1}^m c_i\alpha_i^\vee\in X_*(T_0)_{\RAT},\] it is dominant if and only if $c_i\geq 0$ for all $i$ and $c_i\geq c_j$ for all $i<j$. Noting that the trivial cocharacter $1$ is the basic element in $B(G_{\mathbb{Q}_p},\mu)$, we only need to consider non-basic elements, i.e. we will assume that $\nu\in B(G_{\mathbb{Q}_p},\mu)$ is such that there is $j$ with $c_i> 0$ for all $i\leq j$.

We have
$$\mu-\nu=(1-c_1)(\alpha_1^\vee-\alpha_2^\vee)+(1-c_1-c_2)(\alpha_2^\vee-\alpha_3^\vee)+\cdots+(1-\sum_{i=1}^{m-1}c_i)(\alpha_{m-1}^\vee-\alpha_m^\vee)+(1-\sum_{i=1}^mc_i)\alpha_m^\vee,$$
and hence the condition $\nu\leq \mu$ means that $\sum_{i=1}^mc_i\leq 1$. If $c_m>0$, we have $\langle \nu,\alpha_m\rangle\neq 0$. Noting that $\widetilde{w}_{\alpha_m}=\frac{1}{2}\sum_{i=1}^m\alpha_i$, so $\langle \mu-\nu,\widetilde{w}_{\alpha_m}\rangle\in \mathbb{N}$ holds only when $\sum_{i=1}^mc_i= 1$. We actually have $c_i=1/m$ for all $i$ in this case. Indeed, if there were $j< m$ with $c_j>c_{j+1}$, then we have $\langle \nu,\alpha_j-\alpha_{j+1}\rangle\neq 0$ and by similar arguments we find that $\sum_{i=1}^jc_i= 1$ which contradicts to our assumption. If $c_m=0$, we work with $j$ such that $c_{j+1}=0$ and $c_i>0$ for all $i\leq j$). By similar arguments, we find $\nu=\frac{1}{j}\sum^j_{i=1}\alpha_i^\vee$.

To sum up, we have in this case
$$B(G_{\mathbb{Q}_p},\mu)=\{\alpha_1^\vee, \DF{1}{2}(\alpha_1^\vee+\alpha_2^\vee),\cdots,\DF{1}{m}\sum_{i=1}^{m}\alpha_i^\vee, 1\}.$$
We will simply write $b_i$, $1\leq i\leq m$, for $\frac{1}{i}\sum^i_{j=1}\alpha_j^\vee$ and $b_0$ for 1. One sees easily that the partial order on $B(G_{\mathbb{Q}_p},\mu)$ is as follows:
$$b_0\leq b_m\leq b_{m-1}\leq \cdots\leq b_1.$$

\emph{Case 2.} If $n$ is even, this splits into two cases.

\emph{Case 2.a.} If $G_{\RAT_p}$ is split, we can choose a $\RAT_p$-basis with respect to which $$q=x_1x_{2m}+x_2x_{2m-1}+\cdots+x_mx_{m+1}.$$
Let $T_0=T=\mathrm{diag}(t_1,\cdots, t_m,t_m^{-1},\cdots, t_1^{-1})$, and $\alpha_i\in X^*(T)$, $1\leq i\leq m$, be given by the $i$-th projection. By similar arguments as in the previous case, we have in this case
$$B(G_{\mathbb{Q}_p},\mu)=\{\alpha_1^\vee, \DF{1}{2}(\alpha_1^\vee+\alpha_2^\vee),\cdots,\DF{1}{m-1}\sum_{i=1}^{m-1}\alpha_i^\vee, \DF{1}{m}\sum_{i=1}^{m}\alpha_i^\vee, \DF{1}{m}(\sum_{i=1}^{m-1}\alpha_i^\vee-\alpha_m^\vee), 1\}.$$
Here besides the $b_i$, $0\leq i\leq m$, which we have introduced before, we also set $$b'_m=\DF{1}{m}(\sum_{i=1}^{m-1}\alpha_i^\vee-\alpha_m^\vee).$$
The partial order on $B(G_{\mathbb{Q}_p},\mu)$ is as follows:
$$b_0\leq b_m\leq b_{m-1}\leq \cdots\leq b_1, \ \ b_0\leq b'_m\leq b_{m-1}$$

\emph{Case 2.b.} If $G_{\RAT_p}$ is non-split, we can choose a $\RAT_p$-basis with respect to which $$q=x_1x_{2m}+x_2x_{2m-1}+\cdots+x_{m-1}x_{m+2}+x_m^2-ux_{m+1}^2, \ \ u\in \COMP^\times \text{ non-square}.$$
Let $T_0=\mathrm{diag}(t_1,\cdots, t_{m-1},1,1,t_{m-1}^{-1},\cdots, t_1^{-1})$ and $T$ be its centralizer. Using similar notations and arguments as before, we have
$$B(G_{\mathbb{Q}_p},\mu)=\{\alpha_1^\vee, \DF{1}{2}(\alpha_1^\vee+\alpha_2^\vee),\cdots,\DF{1}{m-1}\sum_{i=1}^{m-1}\alpha_i^\vee, 1\}.$$
Again, we have $b_i$, $0\leq i\leq m-1$, given by the same formula. The partial order on $B(G_{\mathbb{Q}_p},\mu)$ is as follows:
$$b_0\leq b_{m-1}\leq b_{m-2}\leq \cdots\leq b_1.$$

\subsubsection[]{} Notations as in \ref{shimura orthogonal}. The pair $(\mathrm{SO}(V)_{\mathbb{Q}_p},\mu)$ is of Coxeter type if $n\neq 2$, and it is always fully Hodge-Newton decomposable. More precisely, in terms of the list of Coxeter types in \cite{coxeter type} Theorem 5.1.2,
\begin{itemize}
\item if $n\geq 5$ and odd, then it is of type $(B_m, \omega_1^\vee, \mathbb{S})$;
 \item if $n\geq 6$ and even, then it is of type $(D_m, \omega_1^\vee, \mathbb{S})$ (resp. $({}^2D_m, \omega_1^\vee, \mathbb{S})$) when $\mathrm{SO}(V)_{\mathbb{Q}_p}$ is split (resp. non-split).
\end{itemize}
For the exceptions, 
\begin{itemize}
	\item if $n=1$, it is of type $(A_1, \omega_1^\vee, \mathbb{S})$; 
	\item if $n=3$, it is of type $(C_2, \omega_2^\vee, \mathbb{S})$; 
	\item if $n=4$, it is of type $(A_3, \omega_1^\vee, \mathbb{S})$ (resp. $({}^2A'_3, \omega_1^\vee, \mathbb{S})$) when $\mathrm{SO}(V)_{\mathbb{Q}_p}$ is split (resp. non-split).
\end{itemize}
When $n=2$, it is no longer of Coxeter type as $\mathrm{SO}(V)_{\mathbb{Q}_p}$ is no longer absolutely quasi-simple. But it is still fully Hodge-Newton decomposable. It is 
\begin{itemize}
	\item  of type $(A_1, \omega_1^\vee, \mathbb{S})\times (A_1, \omega_1^\vee, \mathbb{S})$, if $\mathrm{SO}(V)_{\mathbb{Q}_p}$ is split;
	\item of type $(A_1\times A_1, (\omega_1^\vee,\omega_1^\vee), {}^1\varsigma_0)$, (see \cite{fully H-N decomp} 2.6) otherwise.
\end{itemize}

Now we can state properties of Newton strata in Shimura varieties attached to orthogonal groups, as well as relations between E-O strata, Newton strata and central leaves.
\begin{corollary}Let $\ES_{\overline{\kappa}}$ be as in the end of \ref{good reduc ortho type}, then each of its Newton stratum is equi-dimensional with closure a union of Newton strata. Moreover, each Newton stratum is a union of E-O strata, and each non-basic E-O is a central leaf in the (non-basic) Newton stratum containing it. More precisely, we have
	\begin{enumerate}
		
		\item if $n$ is odd, then ($m=\frac{n+1}{2}$):
		
		\begin{enumerate}
			\item for $b_i$, $i\,\in\,\{1, \dots, \frac{n+1}{2}\}$, the Newton stratum $\ES_{\overline{\kappa}}^{b_i}$ is of dimension $n+1-i$. Moreover, it coincides with the minimal E-O stratum $\ES_{\overline{\kappa}}^{n+1-i}$;
			\item the basic locus $\ES_{\overline{\kappa}}^{b_0}$ is of dimension $\frac{n-1}{2}$, and it is the disjoint union of E-O strata: \[\ES_{\overline{\kappa}}^{b_0}=\coprod_{i=0}^{\frac{n-1}{2}}\ES_{\overline{\kappa}}^i.\]
		\end{enumerate}
		
		\item if $n$ is even and $\mathrm{SO}(V)_{\mathbb{Q}_p}$ is non-split, then ($m=\frac{n}{2}+1$):
		
		\begin{enumerate}
			\item  for $b_i$, $i\,\in\,\{1, \dots, \frac{n}{2}\}$, the Newton stratum $\ES_{\overline{\kappa}}^{b_i}$ is of dimension $n+1-i$. Moreover, it coincides with the minimal E-O stratum $\ES_{\overline{\kappa}}^{n+1-i}$;
			\item the basic locus $\ES_{\overline{\kappa}}^{b_0}$ is of dimension $\frac{n}{2}$, and it is the disjoint union of E-O strata: \[\ES_{\overline{\kappa}}^{b_0}=\ES_{\overline{\kappa}}^{w_{\frac{n}{2}}}\coprod\ES_{\overline{\kappa}}^{w_{\frac{n}{2}}'}\coprod_{i=0}^{\frac{n}{2}-1}\ES_{\overline{\kappa}}^i.\]
		\end{enumerate}
		
		\item if $n$ is even and $\mathrm{SO}(V)_{\mathbb{Q}_p}$ is split, then ($m=\frac{n}{2}+1$):
		\begin{enumerate}
			\item for $b_i$, $i\,\in\,\{1, \dots, \frac{n}{2}\}$, the Newton stratum $\ES_{\overline{\kappa}}^{b_i}$ is of dimension $n+1-i$. Moreover, it coincides with the minimal E-O stratum $\ES_{\overline{\kappa}}^{n+1-i}$;
			
			\item for $m=\frac{n}{2}+1$, the Newton strata $\ES_{\overline{\kappa}}^{b_m}$ and $\ES_{\overline{\kappa}}^{b'_m}$ are of dimension $\frac{n}{2}$, and both of them are minimal E-O strata. More precisely,

			\begin{itemize}

				\item if $m$ is odd, then $\ES_{\overline{\kappa}}^{b_m}=\ES_{\overline{\kappa}}^{w_{\frac{n}{2}}}$ and $\ES_{\overline{\kappa}}^{b'_m}=\ES_{\overline{\kappa}}^{w_{\frac{n}{2}}'}$;

				\item if $m$ is even, then $\ES_{\overline{\kappa}}^{b_m}=\ES_{\overline{\kappa}}^{w_{\frac{n}{2}}'}$ and $\ES_{\overline{\kappa}}^{b'_m}=\ES_{\overline{\kappa}}^{w_{\frac{n}{2}}}$;

			\end{itemize}
			\item the basic locus is of dimension $\frac{n}{2}-1$, and it is the disjoint union of E-O strata: \[\ES_{\overline{\kappa}}^{b_0}=\coprod_{i=0}^{\frac{n}{2}-1}\ES_{\overline{\kappa}}^i.\]
		\end{enumerate}
		
	\end{enumerate}
\end{corollary}
\begin{proof}
	The first two sentences follow from Theorem \ref{Newton for abelian type} and Proposition \ref{E-O NP for ab cox} respectively.
	
	To see the dimension of basic locus, one can either use \cite{coxeter type} 6.4 and 6.6, and compute the length of maximal elements in the basic locus, or reduce to $\mathrm{GSpin}$-Shimura varieties and use \cite{Howard-Pappas} Theorem 6.4.1 directly. One could then use purity to deduce dimension formula for general Newton strata. All the other statements except for the second sentence of (3.b) follow from Proposition \ref{E-O NP for ab cox} and Corollary \ref{EO for orth type} by simply comparing the dimensions.

	Now we explain the second part of (3.b). The basis we have fixed in \ref{list of NP ortho} case (2.a) give a $\COMP$-lattice $L$, which is perfect with respect to the bilinear form corresponding to $q$. So $\mathrm{SO}(L,q)$ is a split reductive group over $\COMP$, and the torus $T$ we fixed there extends to a split maximal torus of $\mathrm{SO}(L,q)$ which is again denoted by $T$. We identify the Weyl group of $\mathrm{SO}(L,q)$ and that of $\mathrm{SO}(L_{\mathbb{F}_p},q)$ whose elements are viewed as elements in $\mathrm{SO}(L,q)(\COMP)$ via permutations of the chosen basis.

	Let $w_0$ (resp. $w_{J,0}$) be the maximal element in $W$ (resp. $W_J$), then $$w_n=w_{J,0}w_0=(1,2m)(m,m+1).$$ By \cite{muordinary} Remark 6.5.2, the E-O stratum corresponding to $w_{\frac{n}{2}}$ is given by the orbit of $w_{\frac{n}{2}}w_n^{-1}=w_{\frac{n}{2}}w_n$ in $\mathrm{SO}(L_{\mathbb{F}_p},q)$. Then $w_{\frac{n}{2}}w_n\mu(p)$ is obviously a preimage of it in $C(G,\mu)$. One sees by direct computation that $$(w_{\frac{n}{2}}w_n\mu(p))^m=\mathrm{diag}(p,p^{-1},\cdots, p^{-1},p)=(\alpha_1^\vee-\alpha_2^\vee-\cdots-\alpha_m^\vee)(p).$$

	So by the second paragraph of \ref{minimal to newt cocha}, the Newton cocharacter for $w_{\frac{n}{2}}w_n\mu(p)$ is given by the dominant representative (in the Weyl orbit) of $\frac{1}{m}(\alpha_1^\vee-\alpha_2^\vee-\cdots-\alpha_m^\vee)$,  which is $b_m$ (resp. $b'_m$) when $m$ is odd (resp. even).
\end{proof}

For the case $(G,Y)=(\mathrm{GSpin}(V),X')$, in \cite{Howard-Pappas} Howard and Pappas have described the basic locus $\ES_{\overline{\kappa}}^{b_0}$ in terms of some Deligne-Lusztig varieties, by using Rapoport-Zink uniformation and their local description of the GSpin Rapoport-Zink spaces. We have then a similar description of $\ES_{\overline{\kappa}}^{b_0}$ for the case $(G,Y)=(\mathrm{SO}(V), X)$, cf. \cite{Sh1} sections 7 and 8.

Finally, we refer the readers to \cite{Sh1} section 8 for some further discussions in the case $n=19$ for applications to K3 surfaces and their moduli in mixed characteristic.

\

\end{document}